\documentclass[12pt]{article}
\usepackage[utf8]{inputenc}

\RequirePackage{amsthm,amsmath,amsfonts,amssymb,mathrsfs}
\RequirePackage{lmodern}
\RequirePackage{booktabs,tabularx}
\RequirePackage{bm}
\RequirePackage{enumitem}
\RequirePackage{subcaption}
\RequirePackage{chngcntr}
\RequirePackage{multirow}
\RequirePackage{authblk}
\RequirePackage{thm-restate}
\RequirePackage[authoryear]{natbib}%% uncomment this for author-year citations
\RequirePackage{multibib}
\newcites{supp}{Supplementary References}
\RequirePackage[colorlinks,citecolor=blue,urlcolor=blue]{hyperref}%% uncomment this for coloring bibliography citations and linked URLs
\RequirePackage{graphicx,color,xcolor}%% uncomment this for including figures
\RequirePackage{cleveref}
\usepackage{caption}
\theoremstyle{plain}
\theoremstyle{plain}

\newtheorem*{thm*}{Theorem}
\newtheorem{lem}{Lemma}[section]

\theoremstyle{definition}
\newtheorem{defn}{Definition}[section]
\newtheorem*{defn*}{Definition}
\newtheorem{exmp}{Example}[section]
\newtheorem{assumption}{Assumption}

\newtheorem{rem}{Remark}[section]

\Crefname{lem}{Lemma}{Lemmas}
\Crefname{assumption}{Assumption}{Assumptions}
\Crefname{defn}{Definition}{Definitions}
\crefname{assumption}{assumption}{assumptions}
\crefname{exmp}{example}{examples}
\crefname{thm}{theorem}{theorems}
\crefname{figure}{Fig.}{Fig.}
\Crefname{figure}{Figure}{Figures}
\renewcommand{\Pr}{\mathbb{P}}

\newcommand{\R}{\mathbb{R}}

\newcommand{\intd}{\mathrm{d}}

\newcommand{\punitsphere}[1]{\mathcal{N}_{+,\|\cdot\|}^{#1}}

\newcommand{\punitsphereone}[1]{\mathcal{N}_{+,\|\cdot\|_1}^{#1}}

\newcommand{\prnspace}[1]{\mathbb{R}_+^{#1}}

\allowdisplaybreaks
\usepackage{fancyhdr} % This should be set AFTER setting up the page geometry
\hypersetup{
  pdftitle={Validity and Power of Heavy-Tailed Combination Tests under Asymptotic Dependence},
  pdfauthor={Lin Gui; Tiantian Mao; Jingshu Wang; Ruodu Wang}
}

\begin{document}

\title{\bfseries Validity and Power of Heavy-Tailed Combination Tests under Asymptotic Dependence}
%\title{A sample article title with some additional note\thanksref{T1}}
% Running title: Heavy-Tailed Combination Tests under Asymptotic Dependence
%\thankstext{T1}{A sample of additional note to the title.}

\author[1]{Lin Gui}
\author[2]{Tiantian Mao}
\author[1]{Jingshu Wang}
\author[3]{Ruodu Wang}
\affil[1]{Department of Statistics, The University of Chicago\\
\texttt{glin6@uchicago.edu}, \texttt{jingshuw@uchicago.edu}}
\affil[2]{Department of Statistics and Finance, The University of Science and Technology of China\\
\texttt{tmao@ustc.edu.cn}}
\affil[3]{Department of Statistics and Actuarial Science, University of Waterloo\\
\texttt{wang@uwaterloo.ca}}
\renewcommand{\Authfont}{\normalsize}
\renewcommand{\Affilfont}{\small}
\date{}
\maketitle

\begin{abstract}
    Heavy-tailed combination tests, such as the Cauchy combination test and harmonic mean p-value method, are widely used for testing global null hypotheses by aggregating dependent p-values. Existing theoretical guarantees, however, are largely restricted to the case of asymptotically independent 
p-values, leaving the behavior of these tests under broader dependence structures poorly understood.
 We develop a unified framework based on multivariate regularly varying copulas, a flexible class defined by a mild regularity condition on the joint behavior of p-values near zero, that accommodates a wide range of dependence structures. Within this framework, heavy-tailed combination tests are asymptotically valid when the transformation distribution has tail index $\gamma \leq 1$, with $\gamma = 1$ maximizing power while preserving validity. We further show that combination tests with $\gamma = 1$ achieve strictly greater asymptotic power than Bonferroni's method if and only if the 
p-values are not asymptotically independent and signals are not extremely sparse, with the power advantage growing as dependence strengthens. Bonferroni emerges as the $\gamma \to 0$ limit and becomes overly conservative under asymptotic dependence. These results provide theoretical support for using truncated Cauchy or Pareto combination tests, offering a principled approach to enhance power while controlling false positives under complex dependence.
\end{abstract}

% \begin{keyword}[class=MSC]
%     \kwd[Primary ]{00X00}
%     \kwd{00X00}
%     \kwd[; secondary ]{00X00}
%     \end{keyword}
    

%%%%%%%%%%%%%%%%%%%%%%%%%%%%%%%%%%%%%%%%%%%%%%
%% Please use \tableofcontents for articles %%
%% with 50 pages and more                   %%
%%%%%%%%%%%%%%%%%%%%%%%%%%%%%%%%%%%%%%%%%%%%%%
%\tableofcontents

%%%%%%%%%%%%%%%%%%%%%%%%%%%%%%%%%%%%%%%%%%%%%%
%%%% Main text entry area:
\section{Introduction}

\subsection{Background}
Heavy-tailed combination tests, such as the Cauchy combination test \citep{liu2020cauchy}, the harmonic mean p-value method \citep{wilson2019harmonic}, and combined p-values  based on  generalized means \citep{vovk2020combining}, have gained increasing popularity for testing global null hypotheses by aggregating p-values under complex and unknown dependence structures. These methods are widely used in empirical studies, including combining p-values from different testing procedures on the same dataset \citep{sun2020statistical}, across genetic variants within a genome region \citep{liu2019acat}, or from repeated data splits \citep{cai2022model}. The core idea is to transform p-values via quantile functions of a transformation distribution into heavy-tailed variables and approximate the tail probability of their weighted sum, by treating the components as independent in the approximation. Compared to Bonferroni’s test, which remains valid under arbitrary dependence but is often conservative, heavy-tailed combination tests are viewed as similarly valid in practice but more powerful, especially when signals are dense. 

Despite their growing use, the theoretical foundations of these methods remain incomplete.  Existing results establish their asymptotic validity—meaning that the limiting ratio of the test’s type-I error to the nominal level $\alpha$ becomes at most one as $\alpha \to 0$---only under asymptotic independence of p-values in the lower tail \citep{liu2020cauchy,fang2023heavy,gui2023aggregating,liu2025heavily}, as when test statistics are pairwise Gaussian and not perfectly correlated. Intuitively, asymptotic independence requires that p-values become nearly independent as they approach zero, a restrictive assumption given that the validity of the combination tests is guaranteed only as $\alpha$ vanishes.
While prior work has also shown asymptotic validity in extreme dependence of perfect correlation, little is known theoretically about the tests’ behavior under broader, more realistic dependence structures despite their widespread application in such contexts. 
On the other hand, to get full validity for arbitrary dependence, the combination tests need substantial corrections that are often very conservative \citep{vovk2020combining, vovk2022admissible}, and typical dependence can improve the efficiency of these tests greatly \citep{chen2023trade, gasparin2025combining}.
Procedures that are adaptive to dependence assumptions are also popular in more general contexts of multiple testing \citep{fithian2022conditional,marandon2024adaptive}.

Furthermore, as shown by \citep{gui2023aggregating}, under asymptotic independence, heavy-tailed combination tests are asymptotically equivalent to Bonferroni’s test, offering no real power gain. 
In contrast, when all base p-values are identical, Bonferroni correction penalizes the number of p-values linearly, whereas the Cauchy and harmonic mean methods maintain the nominal type-I error without a penalty of dimension. This contrast raises key questions: what dependence structures beyond asymptotic independence allow combination tests to offer asymptotic power gains over Bonferroni? How does the extent of dependence among p-values influence the magnitude of such gains? And how can one select the transformation to maximize power while maintaining validity? Addressing these questions is essential for both the theoretical understanding and practical use of heavy-tailed combination tests.

\subsection{Our contributions} 
This paper develops a broad theoretical framework for understanding the validity and power of heavy-tailed combination tests under complex dependence among p-values. We introduce the multivariate regularly varying (MRV) copula, a flexible class from extreme value theory that allows for a wide range of dependence structures among p-values near zero. This framework imposes only a mild assumption: that a first-order scaling limit exists for the joint probability of small 
p-values as they approach zero, without restricting the shape of the dependence. The class includes a range of commonly used copula models and accommodates both weak and strong forms of tail dependence.

Within this framework, we show that heavy-tailed combination tests are asymptotically valid provided that the transformation distribution is sufficiently heavy-tailed, specifically when the tail index $\gamma \leq 1$, where a larger $\gamma$ corresponds to lighter tails. Moreover, both asymptotic type-I error and power are non-decreasing in $\gamma$. The increase is strict if and only if the 
p-values are not asymptotically independent near zero, with the additional requirement that signals are not extremely sparse for power.
 The choice $\gamma = 1$ achieves the highest power while preserving asymptotic validity across all MRV copulas, making it a natural default choice in practice.

 We also provide a sharp comparison with Bonferroni's method: combination tests achieve strictly greater asymptotic power than Bonferroni if and only if the 
p-values are not asymptotically independent and signals are not extremely sparse.
 In fact, Bonferroni can be viewed as a limiting case of combination tests with $\gamma \to 0$. As dependence among p-values increases, the power advantage of combination tests with $\gamma = 1$ over Bonferroni becomes larger, reaching its maximum under complete lower-tail dependence.

In summary, our results establish that heavy-tailed combination tests with $\gamma = 1$ provide a robust and effective choice that is asymptotically valid under a broad class of dependence structures and can substantially outperform Bonferroni in power as dependence strengthens. Our theory further requires that the transformation distribution has bounded lower support, a mild condition satisfied by transformations such as the truncated Cauchy or Pareto combination tests, which have also been recommended in practice \citep{gui2023aggregating,fang2023heavy,liu2025heavily}. These results provide theoretical support for the use of $(\gamma=1)$ heavy-tailed combination tests in a wide range of applications.

\subsection{Notations}

All random variables are represented in capital letters, and nonrandom values and vectors are represented in small letters. 
We denote by $\mathbf{e}_i=(0,\dots,1,\dots,0)$   the unit vector where the $i$th component is $1$, and let $x^+=\max(0,x)$. We use $\mathbf{0}$ and    $\mathbf{1}$ for the all-zero and all-one vectors. The non-negative part of the $n$-dimensional space without the origin is denoted by $\Xi=[\mathbf{0},\bm{\infty})\setminus\{\mathbf{0}\}$ and the positive part of $n$-dimensional real space is denoted by $\prnspace{n}$. The cube in $[\mathbf{0},\bm{\infty})$ and its complementation are denoted by
\begin{align*}
     [\mathbf{0},\mathbf{c}]=\{\mathbf{x}=(x_1,\dots,x_n)\in\R^n:0\le x_i\le c_i,~\forall~i\},\quad
     [\mathbf{0},\mathbf{c}]^c=[\mathbf{0},\bm{\infty})\setminus[\mathbf{0},\mathbf{c}].
\end{align*} The unit sphere in $\Xi$ with respect to the $L_1$ norm is 
\begin{align*}
    % \unitsphere{n}=\{\mathbf{x}=(x_1,\dots,x_n)\in\rnspace{n}:\|\mathbf{x}\|=1\}~,\\
    \punitsphereone{n}=
    \left\{\mathbf{x}=(x_1,\dots,x_n)\in\Xi:\sum_{i=1}^n x_i=1\right\}~.
\end{align*}
For a distribution with the cumulative distribution function (CDF) $F$, we use  $\overline{F}(\cdot)=1-F(\cdot)$ to denote the survival function and $Q_F$ to represent
its (left) quantile function, defined as
\[
Q_F(t) = \inf\left\{x\in\mathbb{R}:\  F(x) \ge t \right\}, ~~t\in (0,1).
\]

\section{Heavy-tailed combination test}

We study the global null testing problem based on multiple dependent p-values. Specifically, let $H_{0,1}, \dots, H_{0,n}$ denote $n$ base null hypotheses with corresponding p-values $P_1, \dots, P_n$. Our goal is to test the global null hypothesis
\begin{equation*}
H_0^\mathrm{global}: H_{0,1}\cap \dots \cap H_{0,n}.
\end{equation*}

To construct the heavy-tailed combination test, we first introduce a class of distributions that form a subset of the regularly varying class:

\begin{defn}Let $\mathscr{R}_{-\gamma}^*$ denote the class of CDFs $F$ satisfying:
\begin{enumerate}[label=(\roman*)]
\item $F$ belongs to the \emph{regularly-varying tailed class} $\mathscr{R}_{-\gamma}$ with tail index $\gamma>0$, i.e.,
\begin{equation*}
\lim_{x\to\infty}\frac{\overline{F}(xy)}{\overline{F}(x)}=y^{-\gamma},~~\forall~y>0.
\end{equation*}
%where $\overline{F}(\cdot)=1-F(\cdot)$.
\item $F$ is strictly increasing and continuous on its support.
\item The support of $F$ is left-bounded, i.e., $\operatorname{supp}(F) = [c, \infty)$ for some $c \in \mathbb{R}$.
\end{enumerate}
\end{defn}
The class $\mathscr{R}_{-\gamma}^*$ includes many well-known distributions, such as left-truncated $t$ distributions, Pareto distributions, and inverse Gamma distributions (see Table 1 of \citet{gui2023aggregating} for additional examples). Suppose that there is a pre-specified weight $\omega_i$ for each p-value which satisfies that $\omega_i>0$ and $\sum_{i=1}^n\omega_i=n$. Given a distribution $F \in \mathscr{R}_{-\gamma}^*$, we transform each p-value via its quantile function along with this weight $\omega_i$ to obtain heavy-tailed statistics: $X_{i,\omega_i} = Q_F\left((1 - P_i/\omega_i)^+\right)$. Define the average as follows:
\begin{equation}
\label{eq:average_statistics}
\bar X_{n,\bm{\omega}}=\frac{1}{n}\sum_{i=1}^nX_{i,\omega_i}=\frac{1}{n}\sum_{i=1}^nQ_F\left((1 - P_i/\omega_i)^+\right),
\end{equation}
where $\bm{\omega}=(\omega_1,\dots,\omega_n)$ stands for the vector of weights. 
The heavy-tailed combination tests approximate the tail probability $P(\bar X_{n,\bm{\omega}} > x)$ by $n^{1-\gamma} \overline F(x)$ for large $x$: 
\begin{defn}[Weighted heavy-tailed combination test] 
\label{def:weighted-combination-test}
Given p-values $P_1,\dots,$ $P_n$, a transformation function $F\in\mathscr{R}_{-\gamma}^*$ and a non-random vector of positive weights $\bm{\omega} = (\omega_1, \dots, \omega_n)$ satisfying $\sum_{i=1}^n\omega_i=n$, the combined p-value of the heavy-tailed combination test is defined as $$P_{\mathrm{comb}}^{F,\bm\omega}=\min\left(1, n^{1-\gamma} \overline F(\bar X_{n,\bm{\omega}})\right).$$ 
At significance level $\alpha$, the corresponding decision function is
\begin{equation*}
%\label{eq:define-combination-test}
\phi_\mathrm{comb}^{F, \bm \omega} = 1_{\left\{\bar X_{n,\bm{\omega}}> Q_F\left(1-\alpha/n^{1-\gamma}\right)\right\}}.
\end{equation*}
The tail index $\gamma$ associated with $F$ is referred to as the \emph{transformation parameter} of the test.
\end{defn}
% We assume all weights $\omega_i > 0$, as any zero weight would exclude the corresponding hypothesis from both $S_{n,\bm{\omega}}$ and $\kappa$, effectively omitting it from the testing procedure. This assumption entails no loss of generality in the theoretical developments that follow. 

\begin{rem}
Our formulation of weighting differs slightly from some of the previous work \citep{liu2020cauchy,fang2023heavy,gui2023aggregating} in that we apply weights to the base p-values rather than the test statistics, but consistent with \cite{wilson2019harmonic, vovk2020combining}, where p-values are the primitive objects. 
The two formulations are very similar in practice, and our choice is natural because it allows for consistency and comparison across different transformations.  
%For heavy-tailed combinations, weighting the p-values and weighting the test statistics are asymptotically equivalent. 
Additionally, instead of using the average statistic in \eqref{eq:average_statistics}, one may define a general combination statistic as
$S=a\sum_{i=1}^nX_{i,\omega_i}$ for any fixed $a >0$, with the corresponding combined p-value $\min\left(1, na^{\gamma} \overline F(S)\right)$. %, which is the same as our $P_\mathrm{comb}^{F, \bm \omega}$. 
Setting $a = 1/n$ corresponds to the average statistics as in \Cref{def:weighted-combination-test}, while $a= 1$ corresponds to the sum. Empirically, the choice of $a$ has a negligible effect on the power of the combination test.
\end{rem}

\begin{exmp}[Truncated Cauchy combination test]
\label{exmp:truncated_cauchy}
The Cauchy combination test was originally proposed by \citet{liu2020cauchy}. It constructs a combined test statistic
$$T = \frac{1}{n}\sum_{i=1}^n \tan\left\{(0.5-P_i)\pi\right\}.$$
% where the weights satisfy $\sum_i \omega_i = 1$.  
The corresponding p-value is then computed by approximating the distribution of $T$ with a standard Cauchy distribution:
$$P_{\mathrm{comb}}^{\mathrm{CCT}}=\frac{1}{2}-\frac{\arctan T}{\pi}.$$
Here, the transformation function $F$ is the cumulative distribution function of the standard Cauchy distribution, whose survival function is $\overline{F}_\mathrm{C}(x) = 1/2 - (\arctan x)/\pi$ and quantile function is $Q_{F_\mathrm{C}}(x)=\tan[\pi(x-1/2)]$. The transformation parameter is $\gamma = 1$ and the weights $\omega_i = 1$.

Since the Cauchy distribution is unbounded below, the test can be inefficient when some base p-values equal 1, leading to negatively infinite transformed values. To address this, variants based on truncated or half-Cauchy distributions have been proposed \citep{fang2023heavy, gui2023aggregating, liu2025heavily}. The truncated Cauchy distribution with truncation level $q \in (0,1)$ has CDF
$F_{\mathrm{tC}}(x) = (F_{\mathrm{C}}(x) - q)/(1-q)$ for $x \geq Q_{F_{\mathrm{C}}}(q)$
and belongs to $\mathscr{R}_{-1}^*$. Combination tests based on this transformation were studied by \citet{gui2023aggregating}, while adjusted truncated/half-Cauchy tests were proposed by \citet{fang2023heavy, liu2025heavily} to ensure validity under mutual independence of the p-values.
    
\end{exmp}

\begin{exmp}[Harmonic mean p-values] 
\label{exmp:harmonic_mean}
The harmonic mean p-value (HMP, \citet{wilson2019harmonic}) is defined as 
$$P_{\mathrm{comb}}^{\mathrm{HMP}} = \frac{n}{\sum_i 1/P_i}.$$
This is a special case of the heavy-tailed combination test, where the transformation function $F$ is the CDF of the standard Pareto distribution with survival function $\overline{F}(x) = x^{-\gamma}$ for $\gamma = 1$. The support of this distribution is $[1, \infty)$. In this case, equal weights $\omega_i = 1$ are assigned to all $i$ and the transformed statistics are $X_i = 1/P_i$.
An adjusted version of the HMP was also proposed by \citet{wilson2019harmonic} to ensure validity under the assumption of mutual independence among p-values and large $n$.
\end{exmp}

\section{Multivariate regularly varying copula}
Copulas provide a convenient framework for separating marginal behavior from dependence structure. In the context of statistical inference, this is particularly useful for modeling dependence among p-values without relying on a full joint model for the underlying test statistics.

Building on this perspective, we focus on a broad class of copulas, named \emph{multivariate regularly varying (MRV) copulas}, capturing a wide range of tail dependence behaviors, from weak to strong dependence. Many of the concepts and results in this section are adapted from Chapter 8 of \citet{beirlant2006statistics}, and are restated in a form tailored to our setting and notation. Additional background is provided in Section S2.

\subsection{Definition of MRV copulas}
A \emph{copula} is a multivariate distribution function with standard uniform marginals. For a random vector $\mathbf{Z} = (Z_1, \dots, Z_n)$ with joint distribution function $F_{\mathbf{Z}}$ and marginal distributions $F_1, \dots, F_n$, an associated copula $C$ satisfies
$$
C(F_1(z_1) ,\dots,F_n(z_n))= F_{\mathbf{Z}}(z_1,\dots, z_n),~~~~(z_1,\dots,z_n)\in \R^n.
$$
By Sklar's theorem \citep{sklar1959fonctions}, such a copula always exists and is unique when the marginal distributions are continuous. The copulas of a random vector remain invariant under strictly increasing marginal transformations of the random vector. % 

In hypothesis testing, type-I error control typically focuses on the behavior of p-values near zero. Therefore, the joint lower tail behavior of p-values plays a central role.   
To be consistent with the convention in extreme value theory (where results and models are often formulated for the upper tail), we use the transformation $(1 - P_1, \dots, 1 - P_n)$ of the p-value vector $(P_1, \dots, P_n)$, which maps lower tail dependence into upper tail dependence. We then characterize the dependence structure by analyzing the upper tail behavior of the associated copula. In particular, we focus on the class of MRV copulas, which are defined via the limiting behavior of the copula in the upper tail.

\begin{defn}[Tail dependence function and MRV copula] 
\label{def:mrv-copula}
The upper \emph{tail dependence function} of a copula $C$ is
\[
D(u_1,\dots,u_n): = 1 - C(1 - u_1, \dots, 1 - u_n),~\left(u_1,\dots,u_n\right)\in[\mathbf{0},\mathbf{1}]\subset\R^n. 
\]
A copula $C$ is an \emph{MRV copula} if its tail dependence function $D$ satisfies
\begin{equation*}
%\label{eq:define-mrv-copula}
    \lim_{s\downarrow0} s^{-1} D\left(s v_1,\dots,sv_n\right) = \ell(v_1,\dots,v_n),~~ \forall v_1,\dots,v_n\geq 0  
\end{equation*}
for some function $\ell$. 
\end{defn}

The definition of MRV copulas requires that there exists a first-order scaling limit as the tail dependence function approaches the origin. The limiting function $\ell$ is called the \emph{stable tail dependence function} of the MRV copula $C$ and has the homogeneity property $\ell(sv_1, \dots, sv_n)=s\ell(v_1,\dots, v_n)$ for any $s>0$ and any $v_1, \dots, v_n \geq 0$. It also has the following integral representation:
\begin{equation}
\label{eq:define-l}
\ell(v_1,\dots,v_n) = \int_{\punitsphereone{n}} \max(v_1 \theta_1, \dots, v_n \theta_n) \,  H^*(\intd\bm{\theta}),
\end{equation}
where $H^*$ is a measure on the positive unit sphere $\punitsphereone{n}$, known as the \emph{spectral measure} of $C$ (see Section S2.2 for a formal definition of $H^*$ and the proof of \eqref{eq:define-l}). 

Intuitively, the stable tail dependence function $\ell$ captures the joint dependence of the multivariate distribution at the upper tail. For instance, when $n = 2$, 
$$\ell(1, 1) = 2 - \lim_{s\downarrow0}\Pr[F_1(Z_1)>1-s\mid F_2(Z_2)>1-s]$$
where $(Z_1, Z_2)$ has copula $C$ and marginal CDFs $F_1$ and $F_2$. This quantity reflects the limiting conditional tail probability and summarizes the strength of tail dependence.

An equivalent characterization of MRV copulas (Chapter~8.3.2 of \cite{beirlant2006statistics}) %\citep{10.1007/978-3-642-12465-5_6} 
is through the notion of a limiting copula. Specifically, $C$ is an \emph{MRV copula} if and only if there exists a copula $C^*$ such that,  
\begin{equation}
\label{eq:limiting-copula}
    \lim_{t\to+\infty}C(u_1^{1/t},\dots,u_n^{1/t})^t=C^*(u_1,\dots,u_n), \mbox{~for all $u_1,\dots,u_n\in [\mathbf0,\mathbf1]\subseteq\R^n$.}
\end{equation}
    The limiting copula $C^*$ in \eqref{eq:limiting-copula} is also known as the \emph{extreme value copula}. Intuitively, $C^*$ is the copula of the limiting distribution of the properly normalized component-wise maximum
of independent samples from a multivariate distribution $F$ with corresponding copula $C$ (see details in Section S2.1.1). The connection between $C^*$ and the stable tail dependence function $\ell$ is given by (Eq.~(8.52) of \cite{beirlant2006statistics}):
\begin{equation}
\label{eq:c-star-and-l}
    C^*(u_1,\dots, u_n)=\exp{\left\{-\ell\left(-\log u_1,\dots, -\log u_n\right)\right\}}.
\end{equation}
    
It has also been established that $C^*$, $\ell$, and the spectral measure $H^*$ are in one-to-one correspondence, with each uniquely determining the others. These objects offer complementary representations of the upper-tail dependence structure.

\subsection{Asymptotic dependence structure}
\label{subsec:asymptotic-dependence}
% {
% \color{red}
% We now discuss how the asymptotic dependence of p-values (or test statistics) can be described through MRV copulas and their limiting behavior. 
% Noting that the asymptotic regime considered in this paper corresponds to the error level $\alpha\to0$, which is driven by p-values near $\mathbf{0}$ or, equivalently, by large test statistics ($\to\infty$).
% Since MRV copula limits are formulated near the upper corner $\mathbf{1}$, we recast these regimes as upper-tail dependence within the MRV copula framework. In particular, for p-values this amounts to studying the survival copula of $(P_1,\dots,P_n)$.
% This representation characterizes the joint occurrence of small p-values or large test statistics through the limiting objects associated with MRV copulas.
% }

We now study the strength of asymptotic dependence in the upper tail for random variables $(1 - P_1, \dots, 1 - P_n)$, the transformed p-values, under the assumption that their copula is an MRV copula. This dependence is characterized by the stable tail dependence function $\ell$, or equivalently by the limiting copula $C^*$ or the spectral measure $H^*$ introduced in the previous subsection. These representations describe how joint tail events occur across components, and thus quantify the strength of dependence among p-values when they are close to zero.

% We now discuss the asymptotic dependence structure in the upper tail for random variables $(1-P_1, \dots, 1-P_n)$ with an MRV copula, which is characterized by the stable tail dependence function $\ell$, or equivalently, by the limiting copula $C^*$ or the spectral measure $H^*$. These presentations capture how extreme events co-occur across components of a multivariate distribution, and thus can characterize the dependence structure of p-values when they are close to zero.

Within the MRV copula framework, a special case of weak dependence is \emph{asymptotic independence}. Intuitively, it requires that the joint occurrence of extreme events across components becomes negligible in the upper tail. Formally, it is defined as follows:

\begin{defn}[Asymptotic independence]
\label{def:asymptotic-independence}
A random vector $\mathbf{Z}=(Z_1,\dots,Z_n)$ with continuous marginals is called \emph{asymptotically independent} 
if its copula is MRV and its limiting copula $C^*$ satisfies
\begin{equation}
\label{eq:define-indep-copula}
  C^* (u_1,\dots,u_n) =   C^*_{\mathrm{indep}}(u_1,\dots,u_n):=u_1\dots u_n.
\end{equation}
In this case, the corresponding stable tail dependence function is $\ell(v_1, \dots, v_n) = v_1 + \dots +v_n$, and the spectral measure $H^*$ is concentrated on the basis vectors $\mathbf{e}_i$ for $i = 1, \dots, n$, with unit mass at each point. 
\end{defn}

% When translating into p-values, asymptotic independence means the survival copula of $(1-P_1,...1-P_n)$ is an MRV copula with the limiting copula \eqref{eq:define-indep-copula}.

\begin{rem}
It has been shown that $\mathbf Z$ with continuous marginals is asymptotically independent if it is pairwise quasi-asymptotically independent. That is, for any pair $(i,j)$, the following condition holds (Eq. (8.100) of \cite{beirlant2006statistics})
\begin{equation}
\label{eqn:pairwise}
    \lim_{s\downarrow0}\Pr[F_i(Z_i)>1-s\mid F_j(Z_j)>1-s]=0.
\end{equation}
This result indicates that if, for each pair of components, simultaneous extreme events become negligible in the upper tail, as in \eqref{eqn:pairwise}, then $\mathbf Z$ has an MRV copula and is asymptotically independent. Such pairwise quasi-asymptotic independence is commonly assumed in existing literature on heavy-tailed combination tests \citep{liu2020cauchy,fang2023heavy,gui2023aggregating}, and is satisfied, for example, when test statistics are pairwise Gaussian with no perfect correlations \citep{gui2023aggregating}.
\end{rem}

At the opposite extreme within the MRV copula class is \emph{asymptotic complete dependence}, which represents the strongest form of dependence, where all components tend to exhibit extreme behavior simultaneously.

\begin{defn}[Asymptotic complete dependence]
\label{def:asymptotic-complete-dependence}
A random vector $\mathbf{Z}=(Z_1,\dots,$ $Z_n)$ with continuous marginals
is called \emph{asymptotically completely dependent} if its copula is MRV and its limiting copula $C^*$ satisfies:
\begin{equation}
\label{eq:define-asymp-fully-dep}
   C^*(u_1,\dots,u_n) = C^*_{\mathrm{cdep}}(u_1,\dots,u_n):=\min(u_1,\dots,u_n).
\end{equation}
Equivalently, the stable tail dependence function is $\ell(v_1, \dots,v_n)=\max(v_1,\dots,v_n)$ and the spectral measure is concentrated on $\mathbf{1}/\|\mathbf{1}\|_1$, with a total mass equal to $n$.
\end{defn}

\begin{rem}
A sufficient condition for asymptotic complete dependence is pairwise asymptotic dependence. 
That is, if for any pair $(i,j)$, 
\[
\lim_{s\downarrow0}
\Pr\{F_i(Z_i)>1-s \mid F_j(Z_j)>1-s\}=1,
\]
then the copula of $\mathbf Z$ is an MRV copula with the complete-dependence limiting copula $C^*_{\mathrm{cdep}}$, and hence $\mathbf Z$ is asymptotically completely dependent. 
\end{rem}

We now introduce a framework for comparing the strength of asymptotic dependence in the upper tail of MRV copulas, corresponding to the regime where p-values approach zero. We begin by defining a partial ordering on copulas.

\begin{defn}[Pointwise order]
\label{def:concord-order-copula}
Given two copulas $C_1$ and $C_2$, we say that $C_2$ dominates $C_1$ in the pointwise order, denoted $C_1 \preceq_{\mathrm{pw}} C_2$,  if
\[
C_1(u_1,\dots,u_n)\le C_2(u_1,\dots,u_n),\quad \mathrm{for~all~}(u_1,\dots,u_n)\in[\mathbf{0},\mathbf{1}].
\]
\end{defn}
Intuitively, a higher position in the pointwise order reflects stronger positive dependence. 
The  order is weaker than the commonly used concordance order, which is defined as both $C_1 \preceq_{\mathrm{pw}} C_2$
and $\overline{C}_1 \preceq_{\mathrm{pw}} \overline{C}_2$, where $\overline{C}$ is the survival copula of a copula $C$. For many commonly used copula families, this pointwise ordering is well-understood and can be inferred directly from parameter values.

\begin{exmp}[Dependence ordering of $t$ copulas]
\label{exmp:t-dependence}
Let $T = T_{n, \nu, \Sigma}$ and $T' = T_{n, \nu, \Sigma'}$ be two $n$-dimensional $t$ distributions with the same degrees of freedom  $\nu$ and different scale matrices $\Sigma = (\rho_{ij})$ and $\Sigma' = (\rho_{ij}')$, where $\rho_{ii}=\rho'_{ii}$ for all $i=1,\dots,n$.  If $\rho_{ij} \le \rho_{ij}'$ for all $i \ne j$, then the corresponding $t$ copulas satisfy
\[
C_{t,\Sigma}\preceq_{\mathrm{pw}} C_{t,\Sigma'}.
\]
This follows from Theorem~1 of \citet{ANSARI2021104709}, since the supermodular order implies the pointwise order \citep{shaked2007stochastic}. This result aligns with the intuition that higher correlation leads to stronger positive dependence.
\end{exmp}

\begin{exmp}[Dependence ordering of survival Clayton copulas]
\label{exmp:clayton-dependence}
The Clayton copula with parameter $\theta > 0$ is defined as
\[
C_\theta(u_1,\dots,u_n)
= \left( \sum_{i=1}^n u_i^{-\theta} - (n-1) \right)^{-1/\theta}, \quad u_i \in [0,1].
\]
Its survival copula is given by
$
\overline{C}_\theta(u_1,\dots,u_n)
= \Pr(1-U_1 \le u_1,\dots,1-U_n \le u_n),
$
where $(U_1,\dots,U_n)$ follows $C_\theta$.

The dependence level of the survival Clayton copula is determined by the generator index $\theta$. 
According to Lemma 4.3 of \citet{embrechts2009additivity} and (9.A.17) of \citet{shaked2007stochastic}, for any $0 < \theta_1 \le \theta_2$, 
\[
C_{\theta_1}\preceq_{\mathrm{pw}} C_{\theta_2},
\]
where $C_{\theta_1}$ and $C_{\theta_2}$ are survival Clayton copulas with generator index $\theta_1$ and $\theta_2$.
\end{exmp}

For MRV copulas, this ordering is preserved under the limiting operation, as formalized below.
\begin{restatable}{prop}{Limitorder}
\label{prop:limit-preserve-concordance-order}
Let $C_1$ and $C_2$ be MRV copulas with limiting copulas $C_1^*$ and $C_2^*$, respectively. If $C_1 \preceq_{\mathrm{pw}} C_2$, then $C_1^* \preceq_{\mathrm{pw}} C_2^*$.
\end{restatable}

Furthermore, all limiting copulas of MRV copulas are naturally bracketed, in the sense of the pointwise order, by the extremal cases of asymptotic independence and asymptotic complete dependence.
\begin{restatable}{prop}{ConcordanceOrder}
\label{prop:concordance-order-upper-and-lower-bound}
Let $C^*$ be the limiting copula of an MRV copula, as defined in \eqref{eq:limiting-copula}. Then the following ordering holds:
\begin{equation}
\label{eq:concordance-order-upper-and-lower-bound}
C^*_{\mathrm{indep}}\preceq_{\mathrm{pw}}C^*\preceq_{\mathrm{pw}}C^*_{\mathrm{cdep}},
\end{equation}
where $C^*_{\mathrm{indep}}$ and $C^*_{\mathrm{cdep}}$ are the limiting copulas corresponding to asymptotic independence and asymptotic complete dependence, as defined in \eqref{eq:define-indep-copula} and \eqref{eq:define-asymp-fully-dep}.
%[\hyperref[subsec:proof-concordance-order-upper-and-lower-bound]{Proof}]  
\end{restatable}

\Cref{prop:concordance-order-upper-and-lower-bound} implies that the limiting upper-tail dependence of MRV copulas is always non-negative in the pointwise sense, as it is bounded below by the case of asymptotic independence.

%{\color{blue} By Theorem 2.2 of \cite{ansari2019ordering}, for two dispersion matrices $\Sigma = (\rho_{ij})_{n\times n}$ and $\Sigma' = (\rho_{ij}')_{n\times n}$, if $\rho_{ij}\le \rho_{ij}'$, then $C_{t,\Sigma}\le C_{t,\Sigma'}$ and thus, the corresponding stable tail dependence functions satisfy $\ell_{t,\Sigma'}\ge \ell_{t,\Sigma'}$. This can be moved to Section 4.4.}

\subsection{Examples of MRV copulas}
\label{subsec:mrv-copula-examples}
The class of MRV copulas is quite broad, as its definition involves only the limiting behavior in the tail. As a result, many commonly used dependence structures fall within this class. We present several examples below, illustrating models that can plausibly describe the dependence of p-values.

\begin{exmp}[Multivariate Gaussian test statistics]
\label{exmp:Gaussian}
Consider test statistics $(T_1,\dots,T_n)$ following a multivariate Gaussian distribution with mean $\bm\mu$ and nondegenerate correlation matrix $R$. We study the dependence of the resulting p-values. 

For one-sided tests, the p-values are given by $P_i = 1 - \Phi(T_i)$, while for two-sided tests they take the form $P_i = 2(1 - \Phi(|T_i|))$. Since copulas are invariant under strictly monotone marginal transformations, the dependence structure of $(P_1,\dots,P_n)$ is determined by that of $(T_1,\dots,T_n)$ in the one-sided case, and by that of $(|T_1|,\dots,|T_n|)$ in the two-sided case.

It is well known that when $R$ is nondegenerate, the Gaussian copula has limiting copula
$
C_R^*(u_1,\dots,u_n)=u_1\cdots u_n
$,
indicating asymptotic independence (see, e.g., \cite{mcneil2015quantitative}). Moreover, the asymptotic independence is preserved under taking absolute values of the multivariate Gaussian test statistics (see a proof in Proposition~\ref{prop:abs-preserve-independence}). Therefore, in this case, both one-sided and two-sided p-values are asymptotically independent (in terms of their survival copulas), and hence follow MRV copulas with the independence limiting copula,  under both the global null and mean shift alternatives. 
\end{exmp}

\begin{exmp}[Multivariate $t$ test statistics]
Consider test statistics $(T_1,\dots,T_n)$ following a multivariate $t$ distribution with degrees of freedom $\nu > 0$, location parameter $\bm\mu$, and scale matrix $\Sigma$. As in the Gaussian case, copula invariance implies that the dependence structure of the p-values is determined by that of $(T_1,\dots,T_n)$ in the one-sided case and by that of $(|T_1|,\dots,|T_n|)$ in the two-sided case.

It is known that multivariate $t$ distributions induce $t$ copulas, which exhibit nontrivial upper-tail dependence and belong to the class of MRV copulas (see Theorem 2.3 of \cite{nikoloulopoulos2009extreme}).
This property is preserved under taking absolute values, that is, $(|T_1|,\dots,|T_n|)$ also admits an MRV copula (see a proof in Proposition \ref{prop:absolute-t-mrv-copula}).

Therefore, when test statistics follow a multivariate $t$ distribution, the survival copulas of both one-sided and two-sided p-values are MRV copulas, under both
the global null and mean shift alternatives, reflecting nontrivial upper-tail dependence.
\end{exmp}

Beyond dependence structures induced by test statistics, one can also model the dependence of $(1-P_1,\dots,1-P_n)$ directly using parametric copula models that are capable of capturing joint tail behavior corresponding to small p-values. Important examples include survival Archimedean and extreme value copulas, both of which fall within the class of MRV copulas under suitable conditions.

\begin{exmp}[Survival Archimedean copulas]
Archimedean copulas provide a flexible class of parametric models for dependence. An Archimedean copula with generator $\phi$ is defined as (\cite{nelsen2006introduction}):
\[
C^\phi(u_1,\dots,u_n)
= \phi^{-1}\Big(\phi(u_1)+\cdots+\phi(u_n)\Big), \quad u_i \in [0,1].
\]
The associated survival copula $\overline{C}^\phi$ describes the joint upper-tail behavior.

If the generator $\phi$ is regularly varying at zero with index $\theta>0$, that is,
\begin{equation}
\label{eq:generator-rv}
\lim_{t\downarrow0}\frac{\phi(ty)}{\phi(t)}=y^{-\theta}, \quad \text{for all } y>0,
\end{equation}
then the survival copula $\overline{C}^\phi$ is an MRV copula (see \cite{embrechts2009additivity}). Many commonly used Archimedean copulas satisfy this condition \citep{weng2012characterization}. 
A notable example is the survival Clayton copula, whose generator $\phi_\theta(t)=t^{-\theta}-1$ satisfies \eqref{eq:generator-rv}, and hence the survival Clayton copula is an MRV copula.
\end{exmp}

\begin{exmp}[Extreme value copulas]
Extreme value copulas correspond precisely to the limiting copulas of MRV copulas. By definition, extreme value copulas satisfy the stability property
$
C(u_1^{1/t},\dots,u_n^{1/t})^t = C(u_1,\dots,u_n),
$
and hence belong to the class of MRV copulas (see Eq.~(8.53) of \cite{beirlant2006statistics}).

A simple example is the logistic model with
\[
\ell(v_1,\dots,v_n)= (v_1^{1/\alpha}+\cdots+v_n^{1/\alpha})^\alpha, \quad 0<\alpha\le1,
\]
which induces the copula
$
C(u_1,\dots,u_n)
= \exp\left\{-\ell(-\log u_1,\dots,-\log u_n)\right\}.
$

Other well-known examples include the Galambos and Gumbel copulas.
\end{exmp}

% \begin{exmp}[Extreme value copula] The extreme value copula itself is an MRV copula as it has the stability property $C(u_1^{1/t},\dots,u_n^{1/t})^t = C(u_1\dots,u_n)$ (Eq.~(8.53) of \cite{beirlant2006statistics}). 
% Some well-known examples include:
% \begin{itemize}
% \item A simple example: $\ell(v_1,\dots,v_n)= (v_1^{1/\alpha}+\dots+v_n^{1/\alpha})^\alpha$, $0<\alpha\le 1$, and in this case, 
% \[
%  C(u_1,\dots,u_n) = \exp \left\{-\ell \left( - \log u_1 , \dots, -\log u_n  \right) \right\}= \exp \left\{\left(\sum_{k=1}^n (\log u_k)^{1/\alpha}   \right)^\alpha \right\},
% \]
% \item The \emph{Galambos copula} or \emph{negative logistic model} has the definition
% \begin{equation*}
%     C_{\theta, \alpha, \beta}^{\text{Gal}}(u_1, u_2) = u_1 u_2 \exp \left[ \left\{ (-\alpha \log u_1)^{-\theta} + (-\beta \log u_2)^{-\theta} \right\}^{-1/\theta} \right].
% \end{equation*}

% \item The \emph{Gumbel copula} is defined as
% \[
% C_{\theta, \alpha, \beta}^{\text{Gu}}(u_1, u_2) = u_1^{1-\alpha} u_2^{1-\beta} 
%     \exp \left\{ - \left( (-\alpha \log u_1)^\theta + (-\beta \log u_2)^\theta \right)^{1/\theta} \right\},
% \]
% which is a general version of the Galambos copula.
% \end{itemize}
% \end{exmp}

\begin{exmp}[Mixtures of MRV copulas]
 MRV copulas can also arise under mixture models of p-values. Suppose $(P_1,\dots,P_n)$ follows a finite mixture distribution: with probability $\pi_k > 0$, the vector is drawn from component $k$, in which $(1-P_1,\dots,1-P_n)$ has an MRV copula $C_k$ with tail dependence function $D_k$ and stable tail dependence function $\ell_k$, for $k = 1, \dots , d$ with $\sum_{k=1}^d\pi_k = 1$. If the components share identical marginals, then the joint copula of $(1-P_1,\dots,1-P_n)$ is the mixture
$$C(u_1, \dots, u_n) = \sum_{k=1}^d \pi_k C_k(u_1, \dots, u_n),$$ with tail dependence $D(u_1,\dots,u_n)
= \sum_{k=1}^d \pi_k\, D_k(u_1,\dots,u_n)$. Since each $C_k$ is MRV, 
$$\lim_{s \downarrow 0} s^{-1} D(s v_1, \dots, s v_n) = \sum_{k=1}^d \pi_k \, \ell_k(v_1, \dots, v_n)$$
for any $v_1,\dots,v_n \ge 0$. Therefore, $C$ is again an MRV copula.
 
 %{\color{red}TODO: I'm not sure if generally we do require ``each component has identical marginals''? ChatGPT suggests the result still holds if we require this condition:
% \[
% \lim_{t \downarrow 0}
% \frac{\mathbb{P}(P_i < t \mid \text{component } k)}
% {\sum_{m=1}^d \pi_m \,\mathbb{P}(P_i < t \mid \text{component } m)}
% = a_{k,i}, \forall i, k.
% \]}

% Formally, let $C_1,\dots,C_d$ be MRV copulas with weights $p_1,\dots,p_d \ge 0$ satisfying $\sum_{k=1}^d p_k$ $= 1$, and consider the mixture copula
% \[
% C(u_1,\dots,u_n)
% = \sum_{k=1}^d p_k\, C_k(u_1,\dots,u_n).
% \]
% Then $C$ is again a copula. Its tail dependence function is given by
% \[
% D(u_1,\dots,u_n)
% = \sum_{k=1}^d p_k\, D_k(u_1,\dots,u_n),
% \]
% where $D_k$ is the tail dependence function of $C_k$. Since each $C_k$ is MRV, we have
% \[
% \lim_{s\downarrow0} s^{-1} D(sv_1,\dots,sv_n)
% = \sum_{k=1}^d p_k\, \ell_k(v_1,\dots,v_n),
% \]
% for any $v_1,\dots,v_n \ge 0$, where $\ell_k$ is the stable tail dependence function of $C_k$. Therefore, the mixture copula $C$ is also an MRV copula.

Such mixture structures can arise from latent factors that affect dependence while preserving the marginal p-value distributions. For example, suppose an unobserved event causes all test statistics to take a common extreme value with probability $\rho$, while with probability $1-\rho$, the statistics are mutually independent. The resulting copula of $(1-P_1, \dots, 1-P_n)$ is then a mixture
$C = \rho \, C^*_{\mathrm{cdep}} + (1-\rho) \, C^*_{\mathrm{indep}}$, which is an MRV copula.

% {\color{red}Such mixture structures can naturally arise in applications. For example, based on \citet{bates2023testing}, we can show that conformal inference p-values across either $m =2$ or $m = 3$ test samples can be written as a mixture of independent / complete dependent components.

% a mixture model with components corresponding to block independence with perfect within-block dependence. One special case is when $m=2$, such dependence is a mixture of independence and complete dependence. (need to read the paper to verify)}
\end{exmp}

% \begin{exmp}[Mixture of MRV copulas] 
% {\color{red} Check conformal paper? mention}
% The class of MRV copulas is also convex in the sense that finite mixtures of MRV copulas remain within the class. Specifically, suppose a copula $C$ is a finite mixture of MRV copulas $C_1, \dots, C_d$, with non-negative weights $p_1, \dots, p_d$ satisfying $\sum_{k=1}^dp_k=1$, then 
% \[
% C(u_1,\dots,u_n) = p_1 C_1(u_1,\dots,u_n )+\dots+p_d C_d(u_1,\dots,u_n)
% \]
% is still a copula and its tail dependence function is the corresponding convex combination:
% \[
% D(u_1,\dots,u_n)
% =\sum_{k=1}^dp_k\left(1-C_k(1-u_1,\dots,1-u_n)\right)
% =\sum_{k=1}^dp_k D_k(u_1,\dots,u_n),
% \]
% where $D_k$ is the tail dependence function of $C_k$. Since each $C_k$ is MRV, the limit exists:
% \[
% \lim_{s\downarrow0}s^{-1}D(sv_1,\dots,cv_n)=p_1 \ell_1(v_1,\dots,v_n)+\dots+p_d \ell_d(v_1,\dots,v_n),
% \]
% for any $v_1,\dots,v_n\geq 0$,
% where $\ell_k$ is the stable tail dependence function of $C_k$. Thus, a finite mixture of MRV copulas remains an MRV copula.
% \end{exmp}

The preceding examples show several general distributional models that lead to MRV copulas. To further illustrate how MRV dependence can naturally arise in statistical applications, we consider an example involving tests for overlapping gene sets.

\begin{exmp}[Overlapping gene-set tests]
Gene-set or pathway analysis is widely used to study associations between sets of functionally related genes and a treatment or phenotype. A common approach combines gene-level test statistics to construct a gene-set statistic and its associated $p$-value for testing whether any gene in the set is associated with the treatment or phenotype \citep{goeman2004global}.
Several functionally related gene sets may describe different aspects of a broader biological process. Their $p$-values can then be further combined to test whether any of the selected sets for the biological process is associated with the treatment or phenotype. Such gene sets often overlap or are nested, particularly those defined by common annotation systems such as Gene Ontology \citep{ashburner2000gene}.

Let $A_1,\ldots,A_n\subseteq\{1,\ldots,m\}$ denote the selected gene sets, where $m$ is the number of individual genes involved. Suppose that all the $m$ gene-level test statistics are jointly Gaussian under the global null, with a non-degenerate correlation matrix. For each gene set $A_i$, let $P_i$ denote the $p$-value obtained by combining the corresponding gene-level test statistics for its global null. Common choices include
minimum-$p$ and Simes methods,
as well as heavy-tailed combinations such
as the Cauchy and harmonic mean. We can show that, for each of these methods, the survival copula of the gene-set $p$-values $(P_1,\ldots,P_n)$ is MRV (formally stated in Proposition~\ref{prop:gaussian-groups}). The resulting gene-set \(p\)-values can exhibit nontrivial asymptotic dependence. For all these methods, 
\[
\lim_{s\downarrow0}\Pr(P_i\le s\mid P_k\le s)
=\frac{|A_i\cap A_k|}{\max\{|A_i|,|A_k|\}}.
\]
The same conclusions hold for direct Bonferroni and Simes combinations, without
exact null calibration.
Thus, overlapping, nonidentical sets can yield nontrivial asymptotic dependence, even though distinct gene-level tests are asymptotically independent. 
\end{exmp}

\subsection{Connection between MRV copula and MRV distributions}
\label{subsec:mrv-copula-and distribution}
%One tantalizing property of an MRV copula is that marginal heavy-tailed transformations can reconstruct an multivariate regularly varying distribution (MRV distribution), whose tail properties are mathematical tractable.

Finally, we discuss the connection between MRV copulas and MRV distributions. 
Intuitively, we will show that if the vector $(1 - P_1, \dots, 1 - P_n)$ follows an MRV copula, then the heavy-tailed transformed statistics $\mathbf{X} = (X_1, \dots, X_n)$, where each $X_i = Q_{F_{i}}(1 - P_i)$ for non-negative distribution $F_i$ with a regularly varying tail, follow an MRV distribution. We will also analytically represent the tail probability of the average $\bar X_n = \sum_{i=1}^n X_i/n$.

There are many equivalent definitions of MRV distributions (see, for instance, Theorem 6.1 of \cite{resnick2007heavy}). We present an intuitive definition as follows:
\begin{defn}[MRV distributions (Chapter 8.4 of \cite{beirlant2006statistics})]
\label{def:mrv}
An $n$-dimensional random vector $\mathbf{X}=(X_1,\dots,X_n)$ with support $[\mathbf{0},\bm{\infty})$ has 
an \emph{MRV distribution}
if there exists a function $\lambda:(\mathbf{0}, \bm{\infty}) \rightarrow (0,\infty)$ such that for any $\mathbf{x}\in(\mathbf{0}, \bm{\infty})$:
\begin{equation}
\label{eq:define-mrv}
\lim_{t\to+\infty}\frac{\Pr(\mathbf{X}\in t[\mathbf{0},\mathbf{x}]^c)}{\Pr(\mathbf{X}\in t[\mathbf{0},\mathbf{1}]^c)}=\lambda(\mathbf{x}).
\end{equation}
It can be shown that the limiting function $\lambda$ must satisfy $\lambda(s\mathbf{x})=s^{-\gamma}\lambda(\mathbf{x})$ for some $\gamma >0$ and any $s>0$ (See Section S2.2). 
The class of MRV distributions with the parameter $\gamma$ is denoted by $\mathbf{X}\in\mathrm{MRV}_{-\gamma}$.
\end{defn}
Given an MRV copula $C$ and marginal distributions $F_1, \dots, F_n$, we can construct an MRV distribution under mild conditions:
\begin{restatable}{prop}{CopulaToMRV}
\label{prop:copula-to-mrv-distribution}
Let $C$ be an MRV copula, and let $F_1, \dots, F_n$ be marginal distributions supported on $\R_+$ such that
\begin{equation}
\label{eq:define-ci}
\lim_{t\to\infty}\overline{F}_i(t)/\overline{F}_0(t)=c_i\in [0,\infty),~~i=1,\dots,n,
\end{equation}
for some reference distribution $F_0\in\mathscr{R}_{-\gamma}$ with $\gamma>0$. If at least one $c_i > 0$, then the joint distribution with CDF
 \begin{equation*} 
 \label{eq:MRV-250218-1}
 F_{\mathbf{X}}(x_1,\dots,x_n) = C\left(F_1(x_1), \dots,F_n(x_n)\right)
\end{equation*} 
is an MRV distribution with parameter $\gamma$.
%[\hyperref[subsec:proof-copula-to-mrv-distribution]{Proof}].
\end{restatable}

\begin{rem}
$c_i$'s defined in \eqref{eq:define-ci} are unique up to a constant factor, and the specific constant is determined by the choice of reference distribution $F_0$. For $\mathbf{X}$  with an MRV distribution, we fix the reference distribution $F_0(t)=1-\Pr(\mathbf{X}\in t[\mathbf{0},\mathbf{1}]^c)$ in the following main text and supplementary materials.
\end{rem}

\Cref{prop:copula-to-mrv-distribution} implies that an MRV distribution can arise from an MRV copula even if only one marginal distribution has a regularly varying tail. 
%However, if only one tail dominates, the resulting limit distribution is effectively univariate. In practice, we often require at least two marginals, say $F_i$ and $F_j$, to have asymptotically equivalent tails:
% \Cref{prop:copula-to-mrv-distribution} suggests that only one of marginal transformations $F_i$ being regularly varying tailed can construct an MRV distribution from an MRV copula. Although if there is only one dominant tail, the joint distribution would converge to an one-dimensional regularly varying tailed distribution. For the sake of practical use, we often want at least two $F_i$ and $F_j$ have asymptotically equivalent tails. That is, 
%\[
%\lim_{t\to+\infty}\frac{1-F_i(t)}{1-F_j(t)}=c\in(0,\infty).
%\]
%to ensure that the multivariate dependence is preserved in the tail behavior (?).
Conversely, suppose that \({\bf X}=(X_1,\ldots,X_n)\in \mathrm{MRV}_{-\gamma}\) has continuous marginal distributions and non-negligible marginal tails, in the sense that $\lim_{t\to\infty}{\mathbb P(X_i>t)}/
{\mathbb P(\mathbf{X}\in t[0,\mathbf 1]^c)}>0$, $i=1,\ldots,n .$ 
Then the copula of $\mathbf{X}$ is an MRV copula (see Corollary 5.18 in \citet{resnick2008extreme} and Eq.~(8.79) of \citet{beirlant2006statistics}).  
We can then characterize tractable approximations for the tail probability of $\bar X_n = \sum_{i=1}^n X_i/n$ in terms of the associated spectral measure $H^*$:
% As such, an MRV distribution inherits key elements from its copula, including the stable tail dependence function $\ell$, the limiting copula
% $C^*$, and the spectral measure $H^*$ as defined in \eqref{eq:define-mrv-copula}-\eqref{eq:define-l}. Then, we can
%  derive tractable approximations for tail probabilities of $S_n = \sum_{i=1}^n X_i$:
\begin{restatable}{prop}{SumTailGeneralNorm}
\label{prop:sum-tail-general-norm}
% Let $\mathbf{X}=(X_1,\dots,X_n)\in{\rm MRV}_{-\gamma}$, $\gamma>0$, be an $\R_+^n$-valued random vector. 
% Then
% \begin{equation}
% \label{eq:mrv-tail-general-norm}
% h(\gamma,H^*):=\lim_{t\to\infty}\frac{\Pr(\bar X_n>t)}{\frac{1}{n}\sum_{i=1}^n\Pr(X_i>t)}
% =n^{1-\gamma}\int_{\punitsphereone{n}}\frac{\left(\sum_{i=1}^n(c_i\theta_i)^{1/\gamma}\right)^\gamma}{\sum_{i=1}^nc_i} H^*(\mathrm{d}\bm{\theta}),
% \end{equation}
% where $H^*$ is the spectral measure of $\mathbf{X}$'s copula, and 
% \[
% c_i=\lim_{t\to\infty}\frac{\Pr(X_i>t)}{\Pr(\mathbf{X}\in t[\mathbf{0},\mathbf{1}]^c)}\in[0,\infty),~~i=1,\dots,n.
% \]
Let $\mathbf{X}=(X_1,\dots,X_n)\in{\rm MRV}_{-\gamma}$, $\gamma>0$, be an $\R_+^n$-valued random vector and have an MRV copula. Denote $H^*$ the spectral measure of $\mathbf{X}$'s copula. 
Then
\begin{equation}
\label{eq:mrv-tail-general-norm}
h(\gamma,H^*):=\lim_{t\to\infty}\frac{\Pr(\bar X_n>t)}{\frac{1}{n}\sum_{i=1}^n\Pr(X_i>t)}
=n^{1-\gamma}\int_{\punitsphereone{n}}\frac{\left(\sum_{i=1}^n(c_i\theta_i)^{1/\gamma}\right)^\gamma}{\sum_{i=1}^nc_i} H^*(\mathrm{d}\bm{\theta}),
\end{equation}
where  
\[
c_i=\lim_{t\to\infty}\frac{\Pr(X_i>t)}{\Pr(\mathbf{X}\in t[\mathbf{0},\mathbf{1}]^c)}\in[0,\infty),~~i=1,\dots,n.
\]
%[\hyperref[subsec:proof-sum-tail-general-norm]{Proof}]  
\end{restatable}

\section{Combination test under MRV copula}
\label{sec:comb-test-under-mrv}
In this section, we investigate the asymptotic validity and power of the weighted heavy-tailed combination test defined in \Cref{def:weighted-combination-test}, under the setting where the dependence among p-values in the lower tail (their joint behavior as they approach zero) is characterized by an MRV copula. Our analysis focuses on the asymptotic setting where the nominal significance level $\alpha \to 0$ and the number of hypotheses $n$ is fixed.

\subsection{Asymptotic validity}

To build intuition,  let us consider the combination test with weights $\omega_i = 1$ for each $i$. If each p-value $P_i$ is uniformly distributed on $[0,1]$ under the global null, then the CDF of the joint distribution of $(1-P_1, \dots, 1-P_n)$ is a copula. If this copula is an MRV copula, then following \Cref{prop:copula-to-mrv-distribution}, the heavy-tail-transformed statistics $(X_1, \dots, X_n)$, where $X_i = Q_{F_{i}}(1 - P_i)$, follows an MRV distribution. 
As a result, the tail probability of their average can be characterized using \Cref{prop:sum-tail-general-norm}, enabling us to analyze type-I error control of the test in the limit $\alpha \to 0$.
More generally, 
for a weighted heavy-tailed combination test defined in \Cref{def:weighted-combination-test}, define the \emph{limiting scaled type-I error} as
\begin{equation}
\label{eq:limiting_error}
q(\gamma)=\lim_{\alpha\downarrow0}\frac{\Pr_{H_0^\mathrm{global}}\left(P_\mathrm{comb}^{F, \bm \omega}\leq \alpha\right)}{\alpha}.
\end{equation}
We have the following result:
\begin{restatable}{thm}{AsymptoticValidity}
\label{thm:combination-test-valid}
%Consider the weighted heavy-tailed combination test defined in \Cref{def:weighted-combination-test}.
Assume that under the global null, the joint distribution of $(1-P_1, \dots, 1-P_n)$ follows an MRV copula and has standard
uniform marginals. Then the limiting scaled type-I error \( q(\gamma) \) of the weighted heavy-tailed combination test depends on the transformation function $F$ only through its tail index $\gamma$, and is non-decreasing in $\gamma$; that is, 
\[
q(\gamma_1) \geq q(\gamma_2), \quad \text{for all } \gamma_1 > \gamma_2 > 0.
\]
Moreover, equality holds if and only if $(1 - P_1, \dots, 1 - P_n)$ is asymptotically independent. 
In particular, $q(1) =  1$, and hence the test is asymptotically valid for all $\gamma \leq 1$, i.e.,
\[
q(\gamma) \leq 1, \quad \text{for all } \gamma \leq 1.
\]
%In other words, the weighted heavy-tailed combination test is asymptotically valid when $\gamma \leq 1$.
%[\hyperref[subsec:proof-heavy-tailed-test-validity]{Proof}].
\end{restatable}

\Cref{thm:combination-test-valid} 
establishes the asymptotic validity of the weighted heavy-tailed combination test for all $\gamma \leq 1$ under a broad class of dependence structures, including those that violate asymptotic independence assumptions typically required in earlier literature. In particular, asymptotic validity is guaranteed whenever the joint distribution of the p-values exhibits lower-tail dependence in the limit. Moreover, the limiting type-I error reaches the nominal level, i.e., $q(1)=1$, when $\gamma = 1$, regardless of the specific asymptotic dependence structure among the p-values. In contrast, the test becomes asymptotically conservative for $\gamma < 1$ and invalid for $\gamma > 1$ when the p-values are not asymptotically independent.

%This approach remains valid even when p-values are only approximately uniform, such as those arising in permutation tests. 

% \begin{rem}
% When the p-values are conservative, i.e., each base p-value is stochastically larger than a standard uniform variable under the null, the asymptotic validity of the weighted heavy-tailed combination test still holds. 

% Specifically, Let $C$ denote the MRV copula of $(1 - P_1, \dots, 1 - P_n)$, and define uniform random variables $(U_1, \dots, U_n)$ with joint distribution following the same copula $C$. Let $F_i$ denote the CDF of $P_i$ under the null, so that
% \[
% (P_1, \dots, P_n) \overset{d}{=} (Q_{F_1}(U_1), \dots, Q_{F_n}(U_n)).
% \]
% Let \( P_{\mathrm{comb}}^{F, \bm{\omega}} \) and \( \tilde{P}_{\mathrm{comb}}^{F, \bm{\omega}} \) be the combined p-values computed from \( (Q_{F_1}(U_1), \dots, Q_{F_n}(U_n)) \) and \( (U_1, \dots, U_n) \), respectively. Since each \( Q_{F_i}(U_i) \geq U_i \) under the global null (as $F_i(x)\leq x$), it follows that
% \[
% P_{\mathrm{comb}}^{F, \bm{\omega}} \geq \tilde{P}_{\mathrm{comb}}^{F, \bm{\omega}}.
% \]
% Therefore, the asymptotic validity of \( \tilde{P}_{\mathrm{comb}}^{F, \bm{\omega}} \) established in \Cref{thm:combination-test-valid} implies the asymptotic validity of \( P_{\mathrm{comb}}^{F, \bm{\omega}} \) as well.
% \end{rem}

\begin{rem}
When the p-values are conservative, including discrete p-values such as those arising from permutation tests, that is, each base p-value $P_i$ is stochastically larger than a standard uniform distribution under the null, the asymptotic validity of the weighted heavy-tailed combination test still holds.  
Specifically, by Theorems 3.3.5 and 3.3.8 of \cite{muller2002comparison},
there exists a random vector $(\tilde P_1,\dots,\tilde P_n)$ with  standard uniform marginals and the same copula as $(P_1,\dots,P_n)$ such that 
$P_i\ge \tilde P_i$ for all $i$. 
Since the combined p-value of the combination test is non-decreasing in each base p-value, 
the type-I error of the combination test applied to $(P_1,\dots,P_n)$ is no larger than that applied to   $(\tilde  P_1,\dots,\tilde  P_n)$, whose asymptotic validity is guaranteed in \Cref{thm:combination-test-valid}. 
\end{rem}

\begin{restatable}{cor}
{ValidCauchy}
\label{cor:CCT-validity}
    Under the same assumptions as in \Cref{thm:combination-test-valid}, the combined 
p-value $P_{\mathrm{comb}}^{\mathrm{CCT}}$
 of the standard Cauchy combination test (\Cref{exmp:truncated_cauchy}) satisfies
$$\lim_{\alpha \downarrow 0} \frac{P_{H_0^{\mathrm{global}}}(P_{\mathrm{comb}}^{\mathrm{CCT}} \leq \alpha)}{\alpha} \leq 1.$$
\end{restatable}

Although the standard Cauchy distribution does not belong to $F\in\mathscr{R}_{-1}^*$, as its support is unbounded below, the corollary follows from \Cref{thm:combination-test-valid} because the Cauchy and truncated Cauchy transformations share the same right-tail behavior, so the validity of the truncated Cauchy combination test transfers to the unbounded case. The standard Cauchy combination test has been studied in detail in the recent parallel work \cite{chakraborty2025universal}, which shows that it is asymptotically conservative under many tail dependence structures, including the multivariate $t$ copula.

In addition, we characterize the bounds on the limiting scaled type-I error for heavy-tailed combination tests. This limit is determined by the lower-tail dependence structure among the p-values, with the bounds achieved under asymptotic complete dependence and asymptotic independence.

\begin{restatable}{thm}{ErrorCompleteDependence}
\label{thm:error-complete-dependence}
Under the same assumptions as in \Cref{thm:combination-test-valid}, the limiting scaled type-I error $q(\gamma)$ as defined in \eqref{eq:limiting_error} satisfies
\begin{equation*}
q(\gamma)\in%\lim_{\alpha\downarrow0}\frac{\Pr_{H_0^\mathrm{global}}\left(P_\mathrm{comb}^{F, \bm \omega}\leq \alpha\right)}{\alpha}
\begin{cases}
\left[\frac{1}{n}\left(\sum_{i=1}^n\omega_i^{1/\gamma}\right)^{\gamma},1\right],~~\gamma\leq 1,\\[5pt]\left[1,\frac{1}{n}\left(\sum_{i=1}^n\omega_i^{1/\gamma}\right)^{\gamma}\right],~~\gamma \geq 1.
\end{cases} 
\end{equation*}
Moreover, when $\gamma \neq 1$, the lower (or upper) bound $\frac{1}{n}\left(\sum_{i=1}^n \omega_i^{1/\gamma}\right)^{\gamma}$ is achieved if and only if $(1 - P_1, \dots, 1 - P_n)$ is asymptotically completely dependent, while the bound $1$ is achieved if and only if $(1 - P_1, \dots, 1 - P_n)$ is asymptotically independent.
%[\hyperref[subsec:error-complete-dependence]{Proof}].
\end{restatable}

\Cref{thm:error-complete-dependence} is closely related to \Cref{prop:concordance-order-upper-and-lower-bound}, 
 which shows that the limiting lower-tail dependence structure among p-values lies between asymptotic independence and asymptotic complete dependence under the MRV copula. For $\gamma \leq 1$, where the heavy-tailed combination tests are asymptotically valid, these two extreme cases correspond to the least and most conservative scenarios, respectively. 
As a special case, when all the weights are equal ($\omega_i = 1$), the bound under complete dependence simplifies to $n^{\gamma-1}$.

%when $\gamma \leq 1$ type-I error in the asymptotic regime exhibits a monotonic relationship with the tail index $\gamma$, reinforcing the practical observation that smaller values of $\gamma$ ensure validity while potentially reducing power.

\subsection{Asymptotic power}
\label{subsec:asymptotic-power}
Next, we investigate the power difference between heavy-tailed combination tests using different transformation parameters $\gamma$. A key determinant of this power is the strength and distribution of the signals under the alternative.
To formalize this, we define the set of dominating signals as follows:
%The power of a global test is highly dependent on the strength of signals, especially those stronger signals might have higher impact. We define the set of dominating signals as follows:
\begin{defn}[Dominating signals]
Consider the p-value vector $\mathbf{P} = (P_1, \dots, P_n)$.  
For two p-values $P_i$ and $P_j$, we write $P_i \succ P_j$ if
\[
\lim_{t\downarrow0}\frac{\Pr(P_i < t)}{\Pr(P_j < t)}= 0.
\]
The set of dominating signals for $\mathbf{P}$ is then defined as
\begin{equation*}
I_{\mathbf P}=  \left\{i\in \{1,\dots, n\}: 
\nexists j \in \{1,\dots, n\}  \text{ such that } P_i \succ P_j\right\}.
\end{equation*}
\end{defn}
Intuitively, $I_{\mathbf P}$ consists of the indices of p-values that have the heaviest lower tails under the alternative, which will most likely dominate the combined test statistic. 
We also impose the following condition on the joint distribution of the base p-values:
\begin{assumption}
\label{asmp:p-alternative}
The base p-values $(P_1, \dots, P_n)$ under the true configuration of hypotheses satisfy:
%Under the alternative, the p-vector sequence $\mathbf{P}^{(k)}=(P_1^{(k)},\dots,P_n^{(k)})$ converges in distribution to a p-vector $\mathbf{P}$ satisfying that
\begin{enumerate}[label=(\roman*)]
    \item Each $P_i$ is continuously distributed;
    \item The vector $(1-P_1, \dots, 1-P_n)$ has an MRV copula;
    \item For any distinct $i,j\in I_{\mathbf{P}}$,  the relative tail probabilities are comparable: 
    \[\lim_{t\downarrow0}\frac{\Pr(P_i<t)}{\Pr(P_j< t)}\in(0,\infty);\]
    and for at least one $i\in I_{\mathbf{P}}$, there exists an $F\in\mathscr{R}_{-1}$ such that $Q_{F}(1-P_i)\in\mathscr{R}_{-\beta}$ for some $\beta\le1$. 
\end{enumerate}
\end{assumption}
\Cref{asmp:p-alternative} is mild and aligns with typical alternatives in hypothesis testing. Part~(i) is a routine continuity assumption. Part~(ii) assumes that the limiting lower-tail dependence structure among p-values remains well-defined under any true hypotheses configuration. Part~(iii) ensures that the transformed statistics $X_{i,\omega_i}$ corresponding to dominating signals retain regularly varying tails under the alternative and with any choice of transformation $F$ (See \Cref{lemma:assmp-1-iii-ensure-rv} for a proof). This is reasonable, as p-values under the alternative are often more concentrated near zero, which results in heavier-tailed transformed statistics than those under the null.
%in the sense that (i) the dependence structure among p-values remain the same under the global null and alternative, and (ii) the distribution of individual p-values under the alternative are more concentrated around $0$ compared to those under the global null, which makes the marginal distributions of $\mathbf{X}$ more ``heavy-tailed'' than regularly varying distributions, i.e., marginal distributions of $\mathbf{X}$ under the global null. 

We now present two types of alternatives under which \Cref{asmp:p-alternative} holds:
%In particular, \Cref{asmp:p-alternative} includes some common scenarios applied in practice:
\begin{exmp}[Alternative types]\label{exmp:alternative_types}
Define $(U_1, \dots, U_n)$ to follow an MRV copula with standard uniform marginals. Let $\overset{d}{=}$ denote equality in distribution. We consider two types of alternative hypotheses: 
\label{def:alternative-types}
$\ $
\begin{enumerate}[label=(\alph*)]
\item \emph{Type-A alternative}:
\begin{equation*}
    (1-P_1,\dots,1-P_n) \overset{d}{=}(G_1(G^{-1}_1(U_1)+\mu_1),\dots, G_n(G^{-1}_n(U_n)+\mu_n)),
\end{equation*}
where $\mu_1,\dots,\mu_n\ge0$ and $G_1,\dots,G_n$ are continuous CDFs satisfying either (i) a tail-shift invariance condition, $\lim_{x \to \infty} \overline{G}_i(x - \mu)/\overline{G}_i(x) \in (0,\infty)$ for all $\mu > 0$ and all $i$, or (ii) that each $G_i$ is the standard normal distribution. These conditions are satisfied by many common distributions, including the normal, exponential, and Student's $t$ distributions. %This type of alternative represents settings where signals are introduced via location shifts, such as in one-sided testing problems:
% one of the two conditions:
% \begin{enumerate}[label=(\roman*)]
% \item $\lim_{x\to\infty}\frac{\overline{G}_i(x-\mu)}{\overline{G}_i(x)}\in(0,\infty),~i=1,\dots,n$.
% \item $G_i$s are the standard normal distribution.
% \end{enumerate}
% Common families such as standard normal, exponential, or Student's t are included by these conditions.
% (see \Cref{prop:alternative-a} for the exact class of distributions included). 
Intuitively, the Type-A alternative presents signals through location shifts, such as the common one-sided location test: 
\[
H_0:\mu_i\leq 0~\text{for all }i=1,\dots,n,~~H_1:\mu_i>0~\text{for some }i=1,\dots,n.
\]
Under condition (i), all coordinates are equally dominant, and $I_{\mathbf{P}} = \{1, \dots, n\}$. Under condition (ii), only those indices with the largest $\mu_i$ values dominate, so $I_{\mathbf{P}} = \{i: \mu_i = \max_j \mu_j\}$.
% The set of the dominating signals $I_{\mathbf{P}}$ for type-A alternative is as follows:
% \begin{itemize}
%     \item For (i), $I_{\mathbf{P}}=\{1,\dots, n\}$;
%     \item For (ii), $I_{\mathbf{P}}=\left\{i\in\{1,\dots, n\}:\mu_i=\max_{j\in\{1,\dots, n\}}\mu_j\right\}$.
% \end{itemize}
\item \emph{Type-B alternative}:
\begin{equation*}
    (P_1,\dots,P_n) \overset{d}{=} ((1-U_1)^{\beta_1},\dots,(1-U_n)^{\beta_n}),
\end{equation*}
where $\beta_1,\dots,\beta_n \ge 1$. 
Each marginal $P_i$ then follows a $\mathrm{Beta}(1/\beta_i, 1)$ distribution. This model, originally proposed in \citet{Sellke01022001}, serves as a simple parametric form for p-values under the alternative. In this setting, the set of dominating signals is given by $I_{\mathbf{P}} = \{i: \beta_i = \max_j \beta_j\}$.
\end{enumerate}    
Both types of alternatives satisfy the regularity conditions in \Cref{asmp:p-alternative}. We provide a detailed discussion in \Cref{sec:p-alternative}.
\end{exmp}

With \Cref{asmp:p-alternative}, we can apply \Cref{prop:sum-tail-general-norm} to analyze the asymptotic power of the heavy-tailed combination test under alternative hypotheses.
%We can show that the asymptotic power of the heavy-tailed combination test increases with the transformation parameter $\gamma$. The power is strictly increasing in $\gamma$ when the set of dominating signals includes at least two indices and the corresponding p-values are asymptotically dependent in the lower tail. 
Because the power of any test vanishes as $\alpha \to 0$, we compare different tests through a \emph{limiting scaled power} that factors out marginal effects:
\begin{equation}
\label{eq:limiting-power-ratio}
\tilde{q}(\gamma)=\lim_{\alpha\downarrow0}\frac{\Pr\left(P_{\mathrm{comb}}^{F,\bm\omega}\le\alpha\right)}{\sum_{i=1}^n\Pr\left(P_i\le{\omega_i\alpha}/{n}\right)}.
\end{equation}
The denominator, which is the sum of marginal probabilities, is not affected by the choice of $F$ and thus $\gamma$. Hence, any change in $\tilde q(\gamma)$ with $\gamma$ directly reflects a change in power.

We have the following result on the power comparison. 
%It is noted that \eqref{eq:limiting-power-ratio} is a function of $\gamma$, which means $F$ only affects this value through its tail index $\gamma$. Moreover, the denominator of \eqref{eq:limiting-power-ratio} is not affected by the choice of $F$ and hence $\gamma$. That is, the monotonicity of $\tilde{q}(\gamma)$ directly reflects the power of the test. \wang{Polish}
\begin{restatable}{thm}{PowerIncrease}
\label{thm:power-increasing}
%Consider the weighted heavy-tailed combination test defined in \Cref{def:weighted-combination-test} and 
Suppose that \Cref{asmp:p-alternative} holds. 
The limiting scaled power $\tilde{q}(\gamma)$  depends on the transformation function $F$ only through its tail index $\gamma$, and is non-decreasing in $\gamma$;  that is,
\[
\tilde{q}(\gamma_1) \geq \tilde{q}(\gamma_2), \quad \text{for all } \gamma_1 > \gamma_2 > 0.
\]
%Assume the assumptions in \Cref{thm:combination-test-valid} and \cref{asmp:p-alternative}. 
% Let $F_1\in\mathscr{R}_{-\gamma_1}^*$ and $F_2\in\mathscr{R}_{-\gamma_2}^*$ with $0<\gamma_1<\gamma_2$. Then the asymptotic power of the heavy-tailed combination test is non-decreasing in the tail index $\gamma$; specifically
% \begin{equation*}
% %\label{eq:power-increase}
% \lim_{\alpha\downarrow0}\frac{\Pr\left(P_\mathrm{comb}^{F_1, \bm \omega}\leq \alpha\right)}{\Pr\left(P_\mathrm{comb}^{F_2, \bm \omega}\leq \alpha\right)}\leq1.
% \end{equation*}
Moreover, the inequality is strict if and only if 
the dominating-signal set $I_{\mathbf P} = \{i_1,\dots, i_S\}$ has at least two elements, and
the subvector $(1-P_{i_1}, \dots, 1- P_{i_S})$ is not asymptotically independent. 
%[\hyperref[subsec:proof-power-increasing]{Proof}].
\end{restatable}

Together, \Cref{thm:combination-test-valid,thm:power-increasing} explain a key empirical observation \citep{gui2023aggregating}: both the type-I error and power of the combination test rise with the transformation index $\gamma$. %These results also imply that combination tests using different transformation functions $F$ with the same tail index $\gamma$ are asymptotically equivalent.
Based on this insight, we recommend using a transformation with parameter $\gamma = 1$ in practice, as it offers a favorable trade-off between maximizing power and ensuring asymptotic validity.

\subsection{Comparison with the Bonferroni test}
\label{subsec:comp-with-bonferroni}
In this subsection, we compare the heavy-tailed combination tests
and the Bonferroni test, which is widely recognized for its validity under arbitrary dependence structures. 
Under the MRV framework, we aim to characterize the dependence structures where the combination tests have a real advantage over the Bonferroni test.  
%
% 
%The Bonferroni test is widely recognized for its robustness in handling arbitrary dependence structures. However, it is also well known for its conservativeness, particularly under conditions of strong positive dependence. In contrast, the heavy-tailed combination test has demonstrated substantial power advantages in empirical studies. Despite these advantages, theoretical results by \citet{gui2023aggregating} establish that, under the previously assumed validity conditions, the heavy-tailed combination test is asymptotically equivalent to the Bonferroni test when p-values are derived from test statistics following a pairwise bivariate normal distribution with no perfect correlations \citep{liu2020cauchy,fang2023heavy,gui2023aggregating}. Since we have shown in earlier sections that the heavy-tailed combination test remains asymptotically valid under the broader class: multivariate regularly varying (MRV) copula, we now undertake a comparative power analysis of these two methods under this dependence structure.
%
To facilitate our analysis, we first define the weighted Bonferroni test with pre-specified non-random weights \citep{genovese2006false}.
\begin{defn}[Weighted Bonferroni test]
\label{def:bonferroni-test}
Let $P_1,\dots,P_n$ be the p-values and $\bm{\omega}=\left(\omega_1,\dots,\omega_n\right) \in\mathbb{R}_+^n$ be a non-random weight vector satisfying $\sum_{i = 1}^n \omega_i = n$. Then the combined p-value of the weighted Bonferroni test is
$$P_{\mathrm{bon}}^{\bm{\omega}} =\min \left(1, n\min_{i=1,\dots,n}\frac{P_i}{\omega_i}\right).$$
% at the significance level $\alpha$ has the decision function
% \begin{equation}
% \label{eq:bonferroni-test}
% \phi_{\mathrm{bon}}^{\bm{\omega}} = 1_{\left\{\min_{i=1,\dots,n}P_i/\omega_i<\alpha\right\}}.
% \end{equation}
\end{defn}

% \begin{lem}
% If $C$ is an MRV copula,
% and $\mathbf X=(X_1,\dots,X_n)$ follows from $C$, then 
% $\mathbf Y=(\phi_1(X_1),\dots, \phi_n(X_n))$ 
% has an MRV distribution,
% where $\phi_1,\dots,\phi_n$
% are asymptotically equivalent RV increasing functions.
% \end{lem}

% \begin{restatable}{lem}{TransformedMRV}
% \label{lemma:transformed_statistics_mrv}
% Suppose $\mathbf{P}$ has a MRV copula. Further, suppose {\color{blue}at least one $w_i$ is strictly greater than $0$ and} $F\in\mathcal{R}_{-\gamma}$ with $\gamma>0$ and $\mathrm{supp}(F)\subset\R_+$. Then the transformed random vector $\mathbf X$ defined in \eqref{eq:def-transformed-random-vector} is regularly varying on $[\mathbf{0},\mathbf{\infty})\setminus\{\mathbf{0}\}$. [\hyperref[subsec:proof_transformed_stats_mrv]{Proof}].
% \end{restatable}

We can show that the combination tests are at least as powerful as Bonferroni in the following result.
\begin{restatable}{thm}{BonferroniConservative}
\label{thm:bonferroni-more-conservative}
%Consider the weighted heavy-tailed combination test defined in \Cref{def:weighted-combination-test} and
Suppose that \Cref{asmp:p-alternative} holds.
The weighted heavy-tailed combination test is asymptotically at least as powerful as the weighted Bonferroni test:
\begin{equation*}
%\label{eq:asymp-comb-test-more-powerful}
\lim_{\alpha\downarrow0}\frac{\Pr\left(P_\mathrm{comb}^{F, \bm \omega}\leq \alpha\right)}{\Pr\left(P_{\mathrm{bon}}^{\bm{\omega}}\leq \alpha\right)}\ge1.
\end{equation*}
Moreover, the inequality is strict if and only if the dominating-signal set $I_{\mathbf P} = \{i_1,\dots, i_S\}$ has at least two elements, and
the subvector $(1-P_{i_1}, \dots, 1- P_{i_S})$ is not asymptotically independent. 
%$(\mathbf{1}-\mathbf{P})_{I_{\mathbf{P}}}$ is not asymptotically independent with $|I_{\mathbf P}|>1$.
%[\hyperref[subsec:proof-bonferroni-conservative]{Proof}].
\end{restatable}

\Cref{thm:bonferroni-more-conservative} confirms that while both the heavy-tailed combination test (with $\gamma \leq 1$) and the Bonferroni test are asymptotically valid under the MRV copula framework, the combination test can be strictly more powerful when the lower tails of the p-values exhibit asymptotic dependence and the signal is not extremely sparse, meaning that the set of dominant signals includes more than one component. %This is consistent with empirical findings showing that the power improvement of combination tests with $\gamma = 1$ is more substantial when the p-values are more positively dependent and the signals are less sparse. 
Intuitively, when there is only a single strong signal, there is no benefit from aggregation, and the combination test provides no advantage over the Bonferroni method.

We can further show that the Bonferroni test can be viewed as the limiting case of the combination test as $\gamma \to 0$:
\begin{restatable}{thm}{BonferroniGammaZero}
\label{thm:bonferroni-gamma-zero}
% Consider the weighted heavy-tailed combination test defined in \Cref{def:weighted-combination-test} and 
Suppose that \Cref{asmp:p-alternative} holds.
We have
\begin{equation}
\label{eq:asymp-bonferroni-gamma-zero}
\lim_{\gamma\downarrow0}\lim_{\alpha\downarrow0} \frac{\Pr\left(P_\mathrm{comb}^{F, \bm \omega}\leq \alpha\right)}{\Pr\left(P_{\mathrm{bon}}^{\bm{\omega}}\leq \alpha\right)}
%\frac{\Pr\left(\phi_\mathrm{bon}^{\bm{\omega}/n}=1\right)}{\Pr\left(\phi_\mathrm{comb}^{F, \bm \omega}=1\right)}
=1.
\end{equation}
%where $\bm{\omega}=(\omega_1,\dots,\omega_n)\in\R_+^n$ and $\bm{\omega}/n=(\omega_1/n,\dots,\omega_n/n)$. 
%[\hyperref[subsec:proof-bonferroni-gamma-zero]{Proof}].
\end{restatable}

Finally, we compare the asymptotic type-I error and power ratio between the proposed heavy-tailed combination test with $\gamma = 1$ and the Bonferroni test under different dependence structures of the p-values. We demonstrate that the relative advantage of the combination test over Bonferroni increases as the asymptotic lower-tail dependence among the p-values increases. 

\begin{restatable}{thm}{BonferroniDependence}
\label{thm:bonferroni-more-conservative-with-dependence}
Assume the assumptions as in \Cref{thm:combination-test-valid}. Let $C^*$ denote the limiting copula of $(1 - P_1, \dots, 1 - P_n)$. Then the asymptotic type-I error ratio between the weighted combination test with $\gamma = 1$  and the weighted Bonferroni test, 
\begin{equation}
\label{eq:error-increase-vs-dependence}
r(C^*)=\lim_{\alpha\downarrow0}\frac{\Pr_{H_0^\mathrm{global}}\left(P^{F,\bm{\omega}}_\mathrm{comb}\leq \alpha\right)}{\Pr_{H_0^\mathrm{global}}\left(P^{\bm{\omega}}_{\mathrm{bon}}\leq \alpha\right)},
\end{equation}
is non-decreasing in $C^*$ under the pointwise order. That is, if $C_1^*\preceq_{\mathrm{pw}}C_2^*$, then $r(C_1^*)\le r(C_2^*)$.
%[\hyperref[subsec:proof-bonferroni-dependence]{Proof}].
\end{restatable}

This result also extends to the power comparison under alternatives where the tail heaviness of the transformed statistics matches that under the null, i.e., $\beta = 1$ in \Cref{asmp:p-alternative},  as is the case for the Type-A alternatives in \Cref{exmp:alternative_types} (\Cref{prop:alternative-a} for the proof).

\begin{restatable}{cor}{PowerGainDependence}
\label{thm:bonferroni-more-conservative-with-dependence-power}
Assume the assumptions as in \Cref{thm:bonferroni-more-conservative} and further $\beta=1$ in \Cref{asmp:p-alternative}. Let $C^*$ denote the limiting copula of $(1 - P_1, \dots, 1 - P_n)$. Then the asymptotic power ratio between the weighted combination test with $\gamma=1$ and the weighted Bonferroni test 
\begin{equation}
\label{eq:power-gain-vs-dependence}
\tilde{r}(C^*)=\lim_{\alpha\downarrow0}\frac{\Pr\left(P^{F,\bm{\omega}}_\mathrm{comb}\leq \alpha\right)}{\Pr\left(P^{\bm{\omega}}_{\mathrm{bon}}\leq \alpha\right)},
\end{equation}
is non-decreasing in $C^*$ under the pointwise order. That is, if $C_1^*\preceq_{\mathrm{pw}}C_2^*$, then $\tilde r(C_1^*)\le \tilde r(C_2^*)$.
%[\hyperref[subsec:proof-power-gain-dependence]{Proof}].
\end{restatable}

\Cref{exmp:t-dependence,exmp:clayton-dependence} illustrate how the pointwise order among $t$ and Clayton copulas is governed by their respective parameters, such as correlation or generator index. Since the pointwise order of copulas can be inherited by their limiting copulas, in view of \Cref{thm:bonferroni-more-conservative-with-dependence} and \Cref{thm:bonferroni-more-conservative-with-dependence-power}, this implies that the Bonferroni test becomes increasingly conservative relative to the heavy-tailed combination test with $\gamma=1$. 
While our theoretical results focus on $\gamma = 1$, which we recommend in practice, simulations in \Cref{sec:simulations} will demonstrate that the increasing trend in the power ratio between the combination test and the Bonferroni test with respect to tail dependence persists more broadly for other choices of $\gamma > 0$.

\begin{rem}[One signal]
\label{rem:sparse-signal}
Suppose that we observe that one of the p-values, say $P_1=:p$, is close to $0$, and all other p-values $P_2,\dots,P_n$ are of moderate size in $(0,1)$. This may happen in the case where only one hypothesis has a significant signal and all other null hypotheses are true, that is, an extreme case of sparse signal.
In this case, if $p$ is very small, then  the  p-value  $P^{F,\bm{\omega}}_\mathrm{comb}$ with any $\gamma>0$ and $\bm{\omega}=(1,\dots,1)$ is similar, because, as $p\downarrow 0$ and keeping the other p-values constant,
$$ 
 \bar X_{n,\bm{\omega}}=
\frac{1}{n}\sum_{i=1}^nQ_F (1 - P_i  ) 
\simeq \frac 1nQ_F (1 -p ) ,
$$
and hence
$$
P_{\mathrm{comb}}^{F,\bm\omega}=  n^{1-\gamma} \overline F
(\bar X_{n,\bm{\omega}})=
 n^{1-\gamma} \overline F\left(
\frac{1}{n} Q_F\left(1  -p\right)\right)
\simeq  n  \overline F (
 Q_F (1  -p ) )
=np.
$$
Therefore, in this special case, we expect all heavy-tailed combination tests to perform similarly to the Bonferroni test. 
\end{rem}

\section{Simulations}
\label{sec:simulations}

In this section, we use simulations to assess the validity and power of the combination tests at fixed significance levels and compare their power with that of the Bonferroni test. We consider three significance levels $\alpha$: $5\times 10^{-2}$, $5\times 10^{-3}$ and $5\times10^{-4}$. 
The tests are expected to more closely reflect their theoretical properties in asymptotic validity and power at smaller significance levels, which are particularly relevant in many applications where multiple testing adjustments lead to stringent thresholds for individual p-values. We simulate two classes of MRV copulas: the survival Clayton copula and the multivariate 
t copula, both with varying levels of asymptotic dependence. The total number of base hypotheses $n$ is set to either $5$ or $100$.

\subsection{Empirical validity}
\label{subsec:empirical-validity}
We first investigate how the type-I error of the heavy-tailed combination test varies with both the transformation parameter $\gamma$ and with the dependence structure among the p-values near $0$. Specifically, we generate $(1-P_1,\dots, 1-P_n)$ under the global null from two families of copulas: the survival Clayton copula and the multivariate $t$ copula. 
For the survival Clayton copula, we vary the generator index $\theta$ to induce different levels of dependence, such that the pairwise Kendall’s $\tau = \theta / (2 + \theta)$ ranges from $0.05$ to $0.95$ in increments of $0.05$, following the relationship given in \citet[Table 7.5]{mcneil2015quantitative}. For the multivariate $t$ copula, we fix the degrees of freedom at $\nu = 5$ and specify the scale matrix $\Sigma$ with unit diagonal entries and a common off-diagonal correlation $\rho$. To vary the pairwise Kendall’s $\tau$ from $0.05$ to $0.95$, we use the transformation $\rho = \sin(\pi \tau / 2)$, as given in \citet[Section 3.1]{demarta2005t}.
In all settings, we fix the combination test weights at $\omega_i = 1$. 
We consider two classes of transformation function $F$, both with varying $\gamma$: 
the truncated $t$ distribution with $\gamma =\nu$ degrees of freedom (defined analogously to the truncated Cauchy in \Cref{exmp:truncated_cauchy}, truncated below the $0.1$ percentile) and the Pareto distribution with shape parameter $\gamma$. 
We compare four values of $\gamma$: $0.3$, $0.6$, $1.0$, and $1.2$, where $\gamma = 1$ corresponds to the truncated Cauchy combination test and the harmonic mean p-value introduced in \Cref{exmp:truncated_cauchy} and \Cref{exmp:harmonic_mean}. For benchmarking, we also include the Bonferroni test and the Half-Cauchy combination test (HCCT) proposed by \citet{liu2025heavily}. Each configuration is replicated $10^7$ times to estimate the empirical scaled Type-I error, defined as the empirical type-I error divided by the nominal level  $\alpha$.

\begin{figure}[ht!]
\centering
\includegraphics[width=0.8\linewidth]{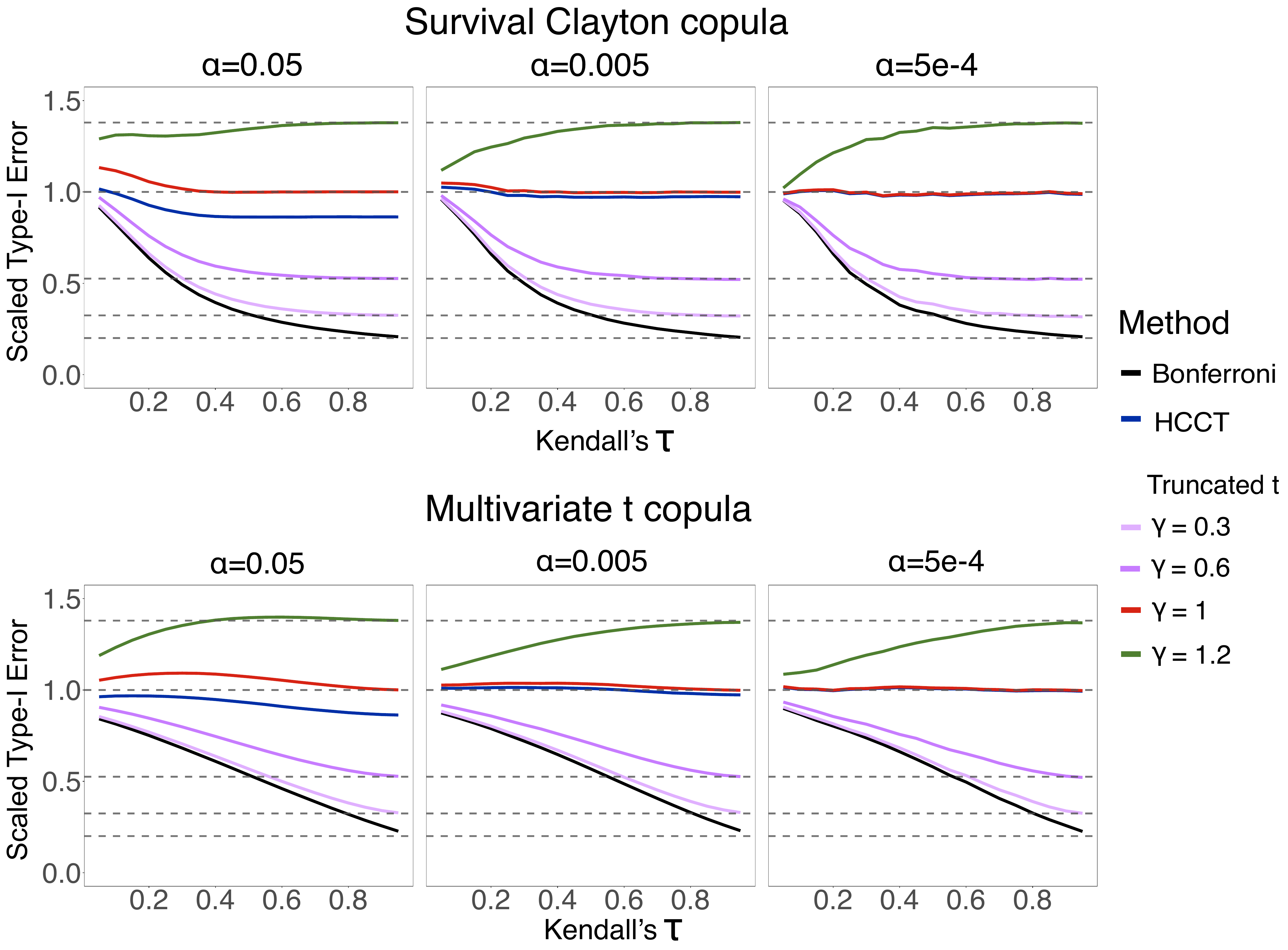}
\caption{Empirical scaled Type-I error of the combination test using truncated t
distributions and $n = 5$. The scaled Type-I error is defined as the empirical type-I error divided by the nominal level $\alpha$. The $5$ dashed horizontal lines, from bottom to top, indicate the bound $n^{\gamma -1}$ for $\gamma = 0, 0.3, 0.6, 1$ and $1.2$. }
\label{fig:validity-tt-5}
\end{figure}

\Cref{fig:validity-tt-5} compares the empirical type-I errors of the combination tests using truncated $t$ distributions with $n = 5$. Similar results for Pareto distributions and $n = 100$ are shown in \Cref{fig:validity-pareto-5,fig:validity-tt-100,fig:validity-pareto-100}. The scaled type-I error increases with the transformation parameter $\gamma$. Empirical type-I errors are well controlled for $\gamma \leq 1$, especially at smaller $\alpha$. When $\gamma = 1$, the test has a stable type-I error rate, whereas for $\gamma < 1$, the tests become more conservative as dependence among p-values increases. All tests approach the theoretical bound $n^{\gamma - 1}$ as the dependence of p-values approaches complete dependence. Bonferroni’s method behaves like the combination test with $\gamma = 0$, and is the most conservative, especially under high dependence of p-values. HCCT performs similarly to the truncated $t$ test with $\gamma = 1$, while offering better control at $\alpha = 0.05$ by correcting for inflation of the test under independence.

\subsection{Empirical power}
\label{subsec:empirical-power}
Next, we investigate the empirical power of the heavy-tailed combination tests by varying the type of alternatives, signal density, signal strength, and dependence structure among p-values.

To generate p-values under the alternative hypotheses, we first simulate $(\tilde P_1,\dots, \tilde P_n)$ under the global null as described in \Cref{subsec:empirical-validity}. For the type-A alternative, we transform each $\tilde P_i$ to a t-statistic $T_i=t_5^{-1}(1-\tilde P_i)$, where $t_5^{-1}$ is the quantile function of the Student's $t$ distribution with $5$ degrees of freedom. We then define each p-value $P_i =1-t_5(T_i+\mu_i)$, where $\bm{\mu}=(\mu_1,\dots,\mu_n)$ is the pre-determined signal vector with $\mu_i\ge0$. For sparse signals, we use $\bm{\mu}=(\mu,\mu,\mathbf{0})$, and for dense signals, $\bm{\mu}=(\mu,\dots,\mu)$. For the Type-B alternative, we directly generate $P_i=\tilde P_i^{\beta_i}$, with $\bm{\beta}=(\beta_1,\dots,\beta_n)$ where each $\beta_i\ge1$. Sparse and dense signal configurations are created using $(\beta_s, \beta_s, \beta_w,\dots, \beta_w)$ and $(\beta_s, \dots, \beta_s)$ respectively, where $\beta_s > \beta_w \geq 1$. Each simulation setting is replicated $10^6$ times to estimate empirical power.

\begin{figure}[ht!]
\centering
\includegraphics[width=0.9\textwidth]{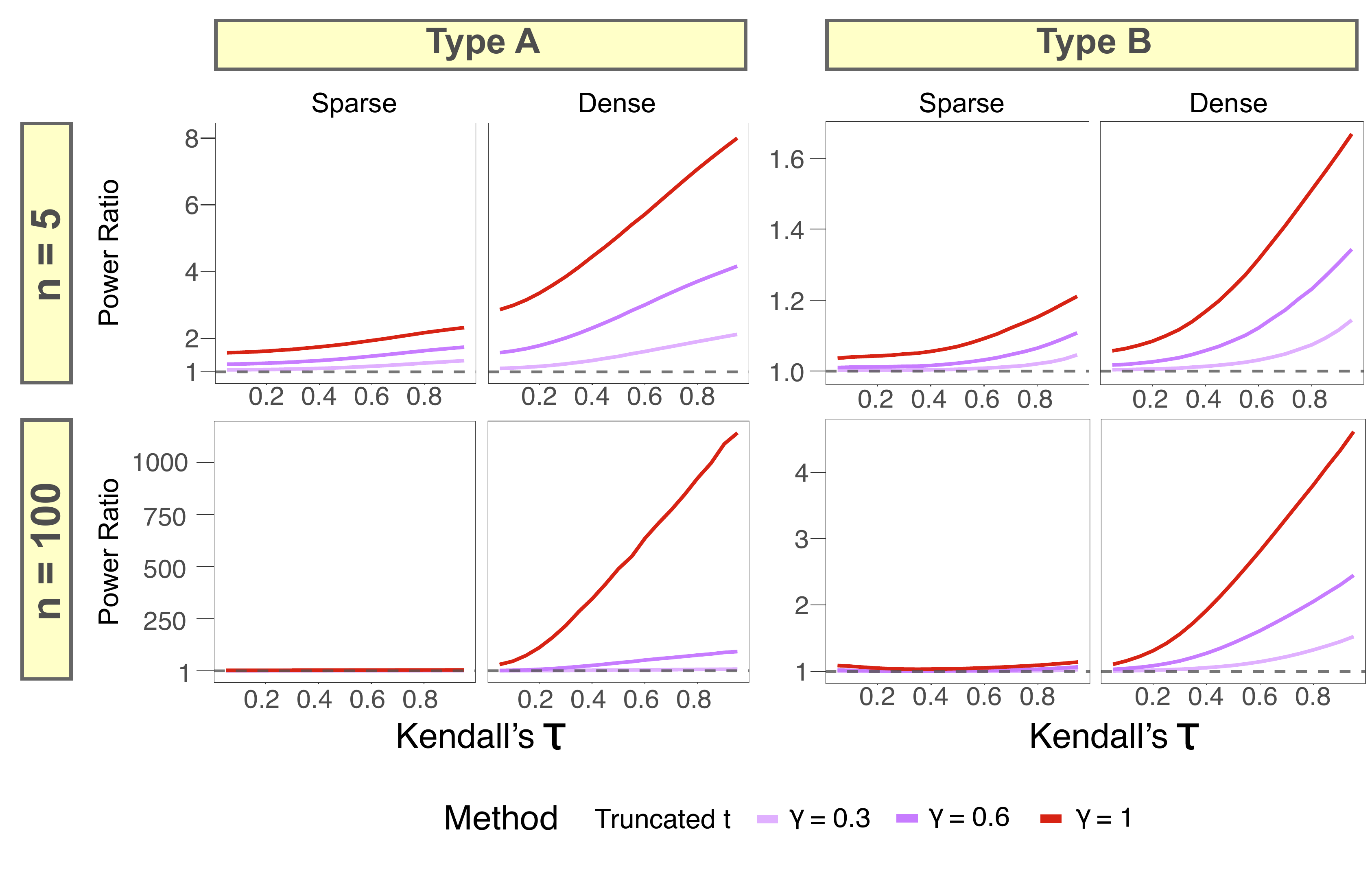}
\caption{Power ratio of the combination test to the Bonferroni test versus Kendall's $\tau$, under different alternative types at significance level $\alpha = 5 \times 10^{-3}$. The underlying dependence among base p-values is modeled using the Clayton copula. }
\label{fig:power-dep-clayton-5e-3}
\end{figure}

We first study how power depends on the strength of dependence among p-values by varying Kendall's $\tau$ from $0.05$ to $0.95$, while keeping marginal signal strengths fixed. 
For the Type-A alternative, we set the signal level such that the combination test with the truncated $t$ distribution (truncated below $0.1$ percentile) can achieve the power of $0.5$ when Kendall's $\tau$ is $0.5$. For the Type-B alternative, we set the signal level so that the same test achieves the power of $0.2$ when Kendall's $\tau$ is $0.5$. This avoids scenarios where all tests have too small power or all have power near $1$. Moreover, we fix $\beta_w = 1.5$ across all settings for Type-B alternatives.

\Cref{fig:power-dep-clayton-5e-3} displays the empirical power ratio of the combination test to the Bonferroni test at significance level $\alpha = 0.005$, with p-values generated from a survival Clayton copula. As dependence increases, the power advantage of the combination test becomes larger, especially in dense signal scenarios, for both Type-A and Type-B alternatives. The power gain is more pronounced for Type-A alternatives. For sparse signals, combination tests still exhibit a power gain over Bonferroni, consistent with our theoretical results, although the improvement is modest due to the presence of only two signals. Furthermore, as the transformation index $\gamma$ increases, the power of the combination test improves, which also aligns with our asymptotic theory.  Similar patterns hold for other significance levels and copulas, and when using heavy-tailed combination tests with Pareto distributions, as shown in Figures~\ref{fig:power-dep-clayton-5e-2}-\ref{fig:pareto-power-dep-t-5e-4}.

\begin{figure}[ht!]
\centering
\includegraphics[width=0.9\textwidth]{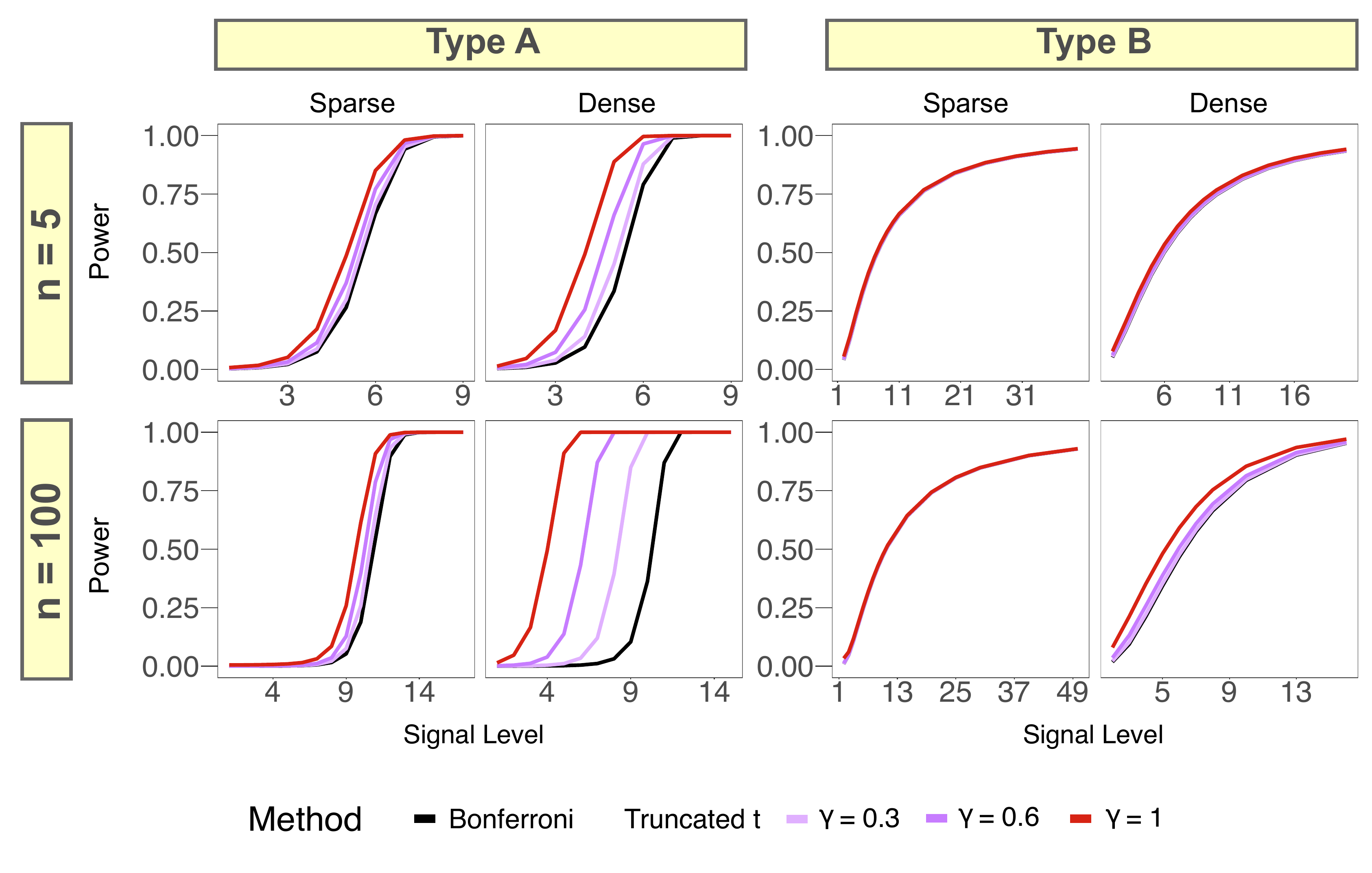}
\caption{The power of the combination test and the Bonferroni test versus signal levels under different alternative types at significance level $\alpha=5\times10^{-3}$. The underlying
dependence among base p-values is modeled using the Clayton copula.}
\label{fig:power-signal-clayton-5e-3}
\end{figure}

To further assess the power gains, we examine the absolute power of the combination tests and Bonferroni's method as signal strength varies. For Type-A alternatives, we vary $\mu$; for Type-B, we vary $\beta_s$, increasing signal strength until at least one of the tests reaches power near one, or until power differences among tests become negligible. In these scenarios, Kendall's $\tau$ is fixed at $0.5$. \Cref{fig:power-signal-clayton-5e-3} shows the results at $\alpha = 0.005$ with base p-values generated from a survival Clayton copula. The figure illustrates how the combination test consistently outperforms Bonferroni as signals strengthen, especially for dense signals. For example, for the Type-A dense signal when $n = 100$, the combination tests with $\gamma = 1$ achieve power close to 1, whereas the Bonferroni test remains near zero power. Additional results for other significance levels and dependence structures are provided in Figures~\ref{fig:power-signal-clayton-5e-2}-\ref{fig:pareto-power-signal-t-5e-4}.

\subsection{Comparison with light-tailed combination tests}
\label{subsec:comp-with-light-tail-test}
In addition, we compare with light-tailed alternatives based on Fisher's combination statistic $\sum_{i=1}^n \log P_i$, where $-\log P_i$ is a light-tailed (exponential) transformation of the p-value. We consider two methods that use this statistic but calibrate it differently to handle dependence.

The first is the adjusted Fisher test of \cite{kost2002combining}, which targets parametric dependence (typically Gaussian or $t$). It computes the mean and variance of Fisher's statistic  $-2\sum_i \log P_i$
under the assumed joint distribution of the test statistics, with the covariance matrix specified or estimated, and approximates its null distribution by a scaled chi-square with matching first two moments. This is a moment-matching approximation without formal type-I error guarantees. The second is the corrected geometric-mean p-value of \cite{vovk2020combining}, which targets arbitrary dependence. It defines the combined p-value as
$$P_{\mathrm{geo}} = \min\left\{1, \; e \cdot \exp\left(\tfrac{1}{n}\sum_{i=1}^n \log P_i\right)\right\}$$
which is a monotone transformation of Fisher's statistic calibrated to be provably valid under any joint distribution of the base 
p-values. As baselines, we also include Bonferroni's test and HCCT \citep{liu2025heavily} as in previous subsections.
Because the adjusted Fisher test requires moments of $\log P_i$
pairs under the joint distribution of the test statistics, which is tractable for Gaussian and $t$ but not for other MRV copulas such as the Clayton, we restrict this comparison to the multivariate $t$ copula. 
We use the settings of \Cref{subsec:empirical-validity} and \Cref{subsec:empirical-power} restricted to the multivariate $t$ copula, evaluating empirical validity at $\alpha = 0.05$ and 
$5 \times 10^{-3}$ and empirical power under the Type-A alternative. The adjusted Fisher test is implemented with the true correlation matrix, giving it the most favorable setting.

% We additionally compare the heavy-tailed combination test with the light-tailed combination test.
% For empirical validity, we use the same settings as in \Cref{subsec:empirical-validity}, with the specific choices of a
% multivariate t copula and significance levels $\alpha=0.05$ and $5\times 10^{-3}$. 
% For power, we use the same settings as in \Cref{subsec:empirical-power}, with the specific choices of a multivariate t copula, the type-A alternative, and significance level $\alpha=0.05$.

% {\color{red} 
% For light-tailed combination tests, we consider Fisher's combination test with several dependence-adjusted corrections~\citep{kost2002combining,vovk2020combining}. The adjusted Fisher test of \citet{kost2002combining} approximates the null distribution of Fisher's statistic by a scaled chi-square distribution, with parameters calibrated using the known dependence structure among the underlying test statistics. The corrected geometric-mean p-value in \citet{vovk2020combining} uses $e$ times the geometric mean of the base p-values as the aggregated p-value, and is theoretically valid under arbitrary dependence.}
% We also include Bonferroni's test and the Half-Cauchy combination test \citep{liu2025heavily} as baselines.

\cref{fig:comp-light-tailed-test} shows that the adjusted Fisher and geometric-mean tests occupy opposite ends of a validity-power trade-off, and neither is competitive across the full range of settings. The adjusted Fisher test fails to control type-I error at $\alpha = 5 \times 10^{-3}$, reflecting the limits of moment-matching approximations in the tail. Since meaningful power comparisons require validity, we restrict the power comparison to $\alpha = 0.05$. At this level, the adjusted Fisher test aggregates dense signals effectively but performs poorly under sparse signals, sometimes even worse than Bonferroni. The geometric-mean test maintains validity across all dependence levels but is overly conservative, with consistently low power. In contrast, heavy-tailed combination tests such as the truncated $t_1$ test and HCCT control type-I error across the full range of $\alpha$
and retain substantial power against both sparse and dense alternatives,  without requiring the dependence structure to be specified or estimated.

% As shown in \cref{fig:comp-light-tailed-test}, the adjusted Fisher cannot does not always control the type-I error, especially at smaller significance levels. When validity holds, (e.g., $\alpha=0.05$,) it demonstrates a strong ability to aggregate signals when the signals are dense. However, when the signals are sparse, its performance deteriorates substantially and can even be worse than Bonferroni’s test. 
% The corrected geometric-mean p-value, on the other hand, always controls the type-I error regardless of the strength of dependence. Correspondingly, the test is overly conservative, and its power is generally low. In comparison, heavy-tailed combination tests such as truncated $t_1$ test and HCCT generally achieve a better trade-off between type-I error control and power.

\begin{figure}[ht!]
\centering
\includegraphics[width=0.9\textwidth]{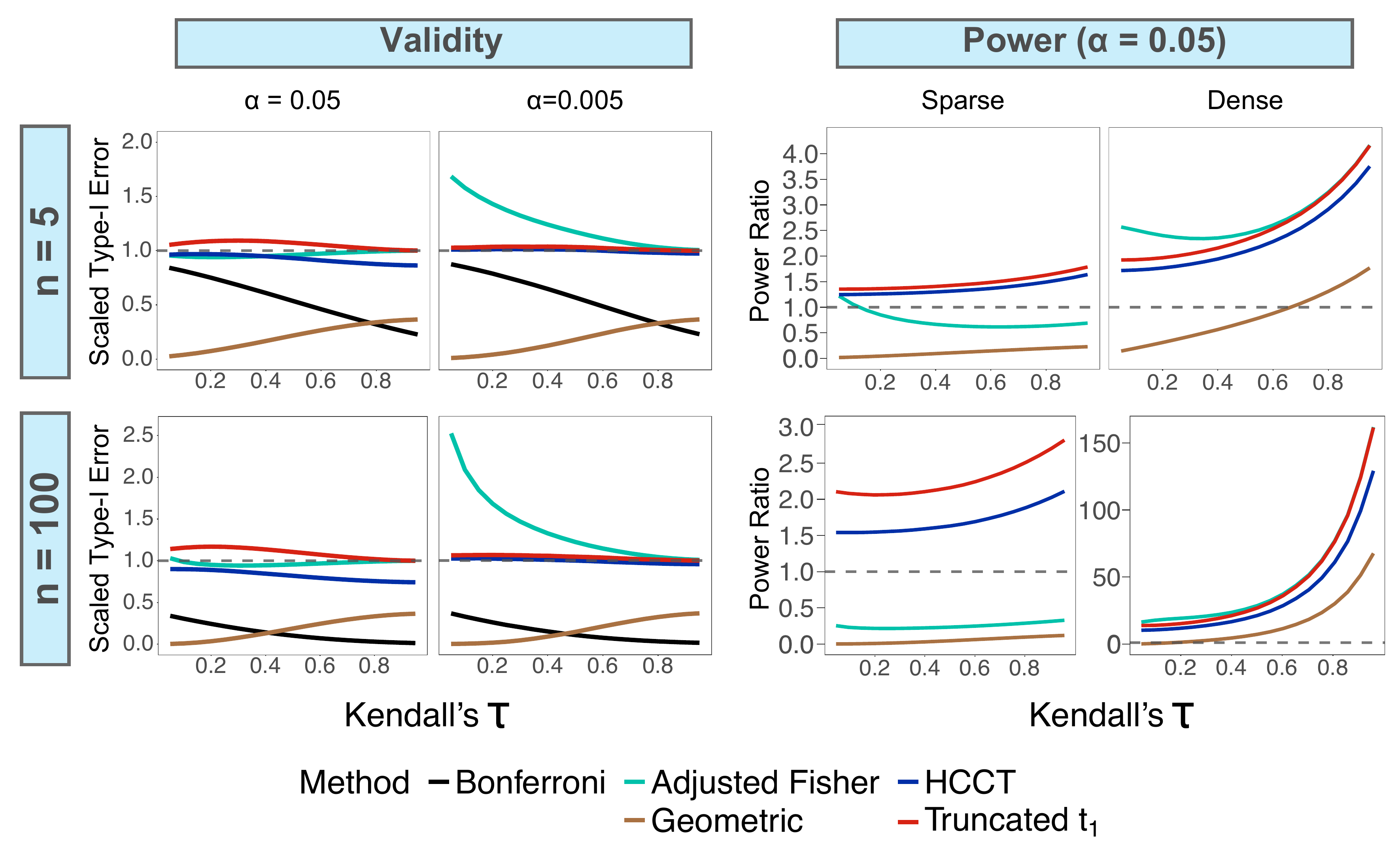}
\caption{Validity and power comparison between heavy-tailed and light-tailed combination tests, with Bonferroni's test as a baseline, as functions of Kendall's $\tau$. The dependence among base p-values is modeled by the multivariate $t$ copula. }
\label{fig:comp-light-tailed-test}
\end{figure}

\section{Real data analysis}
\label{sec:real-data}
Spatial transcriptomics \citep{rao2021exploring} is an emerging technology that enables high-resolution measurement of gene expression while preserving spatial information within tissues. One fundamental step in the spatial transcriptomics analysis is identifying spatially variable genes (SVGs), which are genes whose expression levels vary with spatial locations. Detecting SVGs is critical for uncovering biologically informative signals \citep{asp2019spatiotemporal} and is essential for integrating spatial transcriptomics with single-cell RNA-seq data \citep{satija2015spatial}.

Due to the high noise and complex spatial structures in such data, numerous computational methods have been developed to detect SVGs. However, benchmarking studies suggest no single method consistently outperforms others across datasets \citep{chen2024evaluating, chen2025benchmarking}. To improve detection power, we use heavy-tailed combination tests to aggregate p-values from multiple SVG detection methods.

Specifically, we use the five methods benchmarked in \citet{chen2025benchmarking} that both produce p-values and are shown to control type-I error:
SpatialDE \citep{svensson2018spatialde}, SPARK \citep{sun2020statistical}, SOMDE \citep{hao2021somde}, SPARK-X \citep{zhu2021spark}, and direct computation of RV-coefficient \citep{escoufier1973traitement}. P-values for each gene from all five methods were obtained directly from the authors of \citet{chen2025benchmarking}. We restrict our analysis to the 67 datasets in which all five methods were benchmarked (see \Cref{tab:selected_datasets}), each containing on average about 14,243 genes.  We compare combination strategies via the Bonferroni test and the heavy-tailed combination test with truncated $t$ distribution (truncated below $0.1$ percentile) and $\gamma = 1$.
To identify SVGs, we apply the Benjamini-Hochberg (BH) procedure \citep{benjamini1995controlling} to the gene-level p-values for each method and combination strategy, with the nominal false discovery rate set at $\alpha = 0.05$.

\Cref{fig:real_data}a displays the proportion of discoveries across 67 spatial transcriptomics datasets using five individual SVG detection methods and their combined results. Across datasets, combining p-values consistently improves power,  and the heavy-tailed combination test slightly outperforms the Bonferroni method. However, the observed power gains over Bonferroni are small, despite substantial dependence among methods applied to the same data (reported in \cref{fig:p-vals-corr}). A closer look reveals that for most genes, only a single method typically yields a notably smaller p-value than others, resulting in a sparse signal configuration, for which all heavy-tailed combination p-values are similar to the Bonferroni p-value (see Remark \ref{rem:sparse-signal}).

To better isolate settings where combination methods should excel, we restrict attention to genes where the second-smallest p-value is no more than 10 times the smallest, indicating at least two methods provide comparable evidence. Within this subset, heavy-tailed combination tests demonstrate clear advantages over Bonferroni, aligning with our simulation findings (\Cref{fig:real_data}b). These results reveal that while p-value combination is generally beneficial, its efficiency gain depends on concordance among methods: when multiple methods have similar strength of p-values, heavy-tailed combination tests offer stronger improvements in detection power.

\begin{figure}[t]
    \centering
    \includegraphics[width=0.9\linewidth]{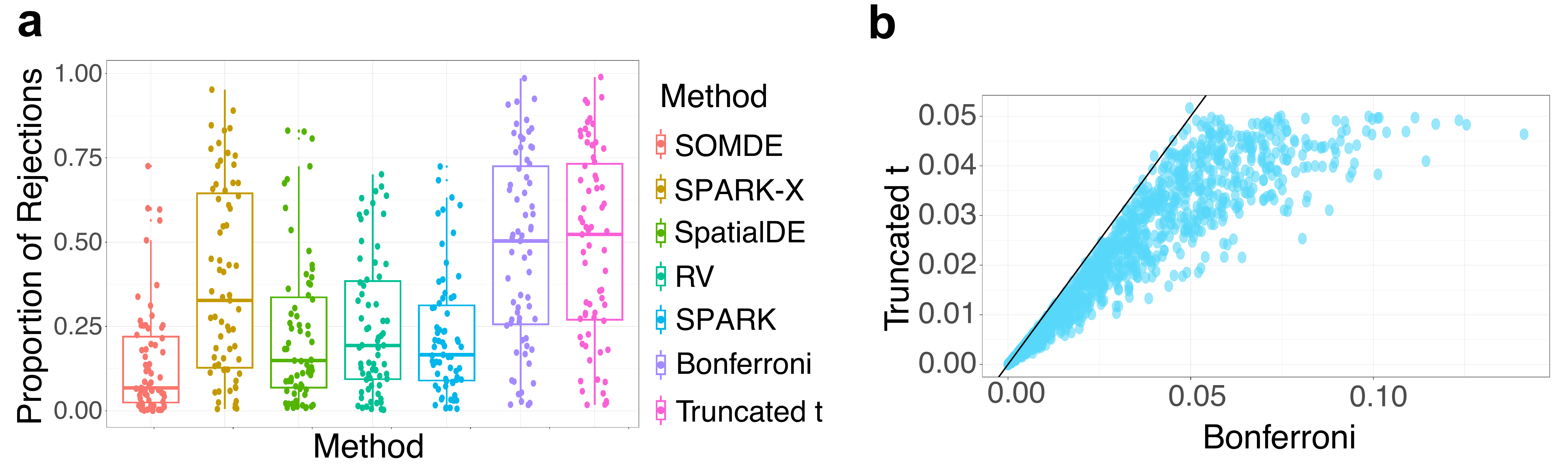}
\caption{Combined p-values improve detection of SVGs. \textbf{a}) Proportion of rejections using original p-values of each method versus combined p-values with FDR controlled at $\alpha = 0.05$. Each dot represents one dataset. 
\textbf{b}) Comparison of combined p-values obtained using the heavy-tailed combination test with truncated $t$ distribution against the Bonferroni test, for genes with the second-smallest p-value no more than 10 times the smallest. Each dot represents a gene, and only genes with at least one resulting p-value below $0.05$ are shown.}
\label{fig:real_data}
\end{figure}

\section{Discussion}
\label{sec:discussion}
The MRV framework provides a unified understanding of the validity and power of heavy-tailed combination tests under a broad range of dependence structures. One limitation is that MRV characterizes the limiting joint tail behavior of the $p$-values, which is difficult to verify from observed data alone, especially without an explicit model for their joint distribution. However, this difficulty is only relevant to the theoretical guarantees, and does not affect the implementation of the combination tests, which do not require specifying or estimating the dependence structure.

Our theoretical analysis focuses on asymptotic validity and power as the significance level $\alpha \to 0$. At fixed $\alpha$, convergence may be slow, and obtaining general non-asymptotic bounds and rates remains an important direction for future work. In our simulations, type-I error inflation is modest even at $\alpha=0.05$ (Figure~\ref{fig:validity-tt-5}); explicit bounds are also available in special cases, including the Clayton copula with Pareto transformation studied by \citet{chen2025subuniformity} (Figure~\ref{fig:nonasymptotic_typeI}). Adjustments such as the HCCT \citep{liu2025heavily}, which adjust for type-I error inflation under independence, can further improve control when $\alpha$ is not small and dependence is weak, although such adjustments become less critical for smaller $\alpha$ or under moderate to strong dependence.

Our theory also assumes a fixed number of hypotheses $n$. As empirical results suggest effectiveness for moderate $n$ (e.g., a few hundred), we conjecture that the theory can extend to scenarios where $n$ increases slowly as $\alpha$ decreases, similar to results shown in \citet{liu2020cauchy} for special cases.

Finally,  \Cref{cor:CCT-validity} establishes asymptotic validity of the original Cauchy combination test within our framework, but our other results, including monotonicity in $\gamma$ and validity and power comparisons with Bonferroni, require transformations with bounded lower support. These results therefore do not directly extend to the original Cauchy test. Previous studies \citep{fang2023heavy,gui2023aggregating} show that unbounded lower support can cause inefficiency with conservative or negatively correlated $p$-values, while the recent parallel work of \citet{chakraborty2025universal} establishes asymptotic conservatism of the original Cauchy combination test under many tail-dependence structures. These findings support bounded-below transformations, such as truncated Cauchy and Pareto, in applications.

\section*{Acknowledgments}
 The authors would like to thank the University of Chicago’s Research Computing Center for its support of this work.

%%%%%%%%%%%%%%%%%%%%%%%%%%%%%%%%%%%%%%%%%%%%%%
%% Funding information, if any,             %%
%% should be provided in the                %%
%% funding section.                         %%
%%%%%%%%%%%%%%%%%%%%%%%%%%%%%%%%%%%%%%%%%%%%%%
\section*{Funding}
 J.W. is partly supported by the National Science Foundation (DMS-2238656) and the National Institute of General Medical Sciences (R35GM162500).  
T.M. is partly supported by the National Natural Science Foundation of China (12371476). R.W. is partly supported by the Natural Sciences and Engineering Research Council of Canada (CRC-2022-00141, RGPIN-2024-03728).

\bibliographystyle{plainnat}

\bibliography{paper/ref}

%%%%%%%%%%%%%%%%%%%%%%%%%%%%%%%%%%%%%%%%%%%%%%
%% Supplementary Material, including data   %%
%% sets and code, should be provided in     %%
%% {supplement} environment with title      %%
%% and short description. It cannot be      %%
%% available exclusively as external link.  %%
%% All Supplementary Material must be       %%
%% available to the reader on Project       %%
%% Euclid with the published article.       %%
%%%%%%%%%%%%%%%%%%%%%%%%%%%%%%%%%%%%%%%%%%%%%%
\clearpage

\renewcommand{\thesection}{S\arabic{section}}   
\renewcommand{\thetable}{S\arabic{table}}   
\renewcommand{\thefigure}{S\arabic{figure}}

\renewcommand{\theHsection}{S\arabic{section}}
\setcounter{section}{0}
\renewcommand{\theHfigure}{S\arabic{figure}}
\setcounter{figure}{0}
\renewcommand{\theHtable}{S\arabic{table}}
\setcounter{table}{0}

\renewcommand {\thepage} {S\arabic{page}}
\setcounter{page}{1}

\begin{center}
\Large\bfseries
Supplementary Information for  \\``Validity and Power of Heavy-Tailed Combination Tests under Asymptotic Dependence''
\end{center}

\section{More notations}
For simplicity, we use $\Pr_0$ to denote $\Pr_{H_0^\mathrm{global}}$ and $[n]=\{1,\dots,n\}$. Further, $\odot$ stands for the Hadamard product and $\circ$ denotes function composition. Let $\Phi$ denote the CDF of the standard normal. %$\overline{F}$ stands for the survival function of a distribution represented by its CDF $F$. 
The power of a vector is taken element-wise, i.e.,
\[
\mathbf{v}^{\beta} = (v_1^{\beta},\dots,v_n^\beta),~\mathbf{v}\in\R^n,~\beta\in\R ~~~\mbox{(whenever suitable)}.
\]

\section{Additional background on MRV copulas and MRV distributions}
In this section, we provide additional background on MRV copulas and  distributions from extreme value theory, focusing on the key concepts needed for Proposition 3.4. This material complements the main text and does not involve the combination test framework.

We begin by introducing additional concepts from extreme value theory that complement the discussion of MRV copulas in the main text and clarify their connections to related objects. We then introduce equivalent definitions of multivariate regularly varying distributions and their associated Radon measures, along with their polar (radial–angular) decomposition. Finally, we discuss tail standardization and establish the relationship between the limiting measures and spectral measures of MRV distributions and their associated copulas.

\subsection{Additional concepts and structure of MRV Copulas}
In this subsection, we develop additional concepts from extreme value theory for MRV copulas, including the domain of attraction, exponent measure, and polar (radial–angular) decomposition, and clarify their connections with concepts like the stable tail dependence function, the limiting copula and spectral measure of MRV copulas introduced in the main text.

\subsubsection{Domain of attraction and multivariate extreme value distribution}
The domain of attraction and multivariate extreme value distributions provide a fundamental framework for characterizing the limiting behavior of extremes under general dependence structures. Specifically, for an independent
sample $\mathbf{X}_1,\ldots,\mathbf{X}_k$ from a distribution function $F$, the distribution of the component-wise maximum, $\mathbf{M}_k$,  is given by
\[
\mathbb{P}\bigl(\mathbf{M}_k \le \mathbf{x}\bigr)
= \mathbb{P}\bigl(\mathbf{X}_1 \le \mathbf{x}, \ldots, \mathbf{X}_k \le \mathbf{x}\bigr)
= F^k(\mathbf{x}), \qquad \mathbf{x} \in \mathbb{R}^n .
\]

Then the limiting distribution is called a multivariate extreme value distribution, and $F$ is in its domain of attraction. 
\begin{defn}
  If there exist $(\mathbf{a}_k)_k$ and $(\mathbf{b}_k)_k$, where $\mathbf{a}_k > \mathbf{0}$, and an $n$-variate distribution function $G$ with non-degenerate
margins such that
\begin{equation*} %\label{eq:MRV-1}
F^k(\mathbf{a}_k \mathbf{x} + \mathbf{b}_k)
\xrightarrow{d} G(\mathbf{x}),
\qquad k \to \infty ,
\end{equation*}
then  we say that $F$ is in the {\it domain of attraction} of $G$,
denoted by $F \in D(G)$. Moreover, $G$ is called a {\it multivariate extreme value (MEV)} function.  
\end{defn}

One important property of multivariate domains of attraction is their decomposition into marginal and dependence components via copulas. More specifically, let $F$ be an $n$-variate distribution function with marginal distributions $F_j$ and the copula $C$. Also, let $G$ be an $n$-variate MEV function with margins $G_j$ and copula $C^*$.  Then, by ({8.79}) of \citesupp{beirlant2006statistics}, we have that $F\in D(G)$ if and only if $F_j \in D(G_j)$, $\forall i\in [n]$, together with
\begin{equation} \label{eq:MEV}
\lim_{k\to\infty}
C^{\,k}(u_1^{1/k},\ldots,u_n^{1/k})
=
C^*(\mathbf{u}),
\qquad
\mathbf{u}\in[\mathbf{0},\mathbf{1}]\subset\R^n .
\end{equation} 
Given the definition of the MRV copula in Equation (3) of the main text Section 3.1, we can also immediately see that $C_F$ satisfying \eqref{eq:MEV} should be an MRV copula, and any MRV copula can be written as a copula of some $F\in D(G)$. 

\subsubsection{Exponent measure $\mu^*$ and stable tail dependence function $\ell$}
Given an MRV copula $C$, its limiting copula $C^*$ is associated with an exponent measure $\mu^*$.
\begin{defn} \label{def:0127-1}
Let $C^*$ be a limiting copula of an MRV copula. Its \emph{exponent measure} $\mu^{*}$ can be defined by
\begin{equation}\label{eq:0127-2}
\mu^{*}\left([\mathbf{0},\mathbf{z}]^c\right)=-\log C^{*}\left(e^{-1/z_1},\dots,e^{-1/z_n}\right), 
~~\mathbf{z}\in[\mathbf{0},\bm{\infty})\subset\R^n.
\end{equation}
\end{defn}
An important property of the exponent measure $\mu^*$ is that it satisfies the homogeneity (see (8.11) of \citesupp{beirlant2006statistics}): for any $0<s<\infty$,
\begin{equation}
\label{eq:mu-homogeneity}
\mu^{*}(s\,\cdot)=s^{-1}\mu^{*}(\cdot).
\end{equation}

Notice that \eqref{eq:MEV} is equivalent to 
\begin{equation*} 
\lim_{t\to\infty}
t\Bigl\{1-C\bigl(1-v_1/t ,\ldots, 1- v_n/t \bigr)\Bigr\}
=-\log C^* (e^{-v_1},\ldots,e^{-v_n}),
~~
\mathbf{v}\in[\mathbf{0},\bm{\infty})\subset\R^n,
\end{equation*}
then we have the stable tail dependence function $\ell$ of the MRV copula satisfying
\begin{align} \label{eq:260127-10}
     \ell(v_1,\dots,v_n)
     & = \lim_{s\downarrow0} s^{-1} \left(1 - C(1 - sv_1, \dots, 1 - sv_n)\right)\notag\\
    &=   -\log C^* (e^{-v_1},\ldots,e^{-v_n})\notag\\
    % & =  -\log {G_*}(1/ v_1,\ldots,1/v_n) \notag\\
& =\mu^{*}\bigl([\mathbf{0},(1/v_1,\ldots,1/v_n)]^c\bigr).
\end{align}

For completeness, we note that the stable tail dependence function $\ell$ and the exponent measure $\mu^*$ can also be defined with an MEV distribution function $G$. For example, for any $n$-variate MEV function $G$, its exponent measure is defined by
\[
\mu^{*}\bigl([\mathbf{0},\mathbf{z}]^c\bigr)=-\log G_{*}(\mathbf{z}),
~~\mathbf{z}\in[\mathbf{0},\bm{\infty})\subset\R^n.
\]
where
\[
G_{*}(\mathbf{z})=G\!\left(
Q_{G_1}(e^{-1/z_1}),\ldots,
Q_{G_n}(e^{-1/z_n})
\right).
\]
The above formulation shows that both objects are intrinsic to the MRV copula and do not depend on a particular choice of marginal distributions.

\subsubsection{The polar decomposition and the spectral measure $H^*$}
The exponent measure $\mu^*$ admits a polar (radial–angular) decomposition, which yields a spectral measure $H^*$ associated with MRV copulas, as introduced in Section 3.1 of the main text. Specifically, $\mu^*$ can be decomposed into a radial component and an angular component, where the angular component is captured by $H^*$ defined on the unit simplex. This representation provides a geometric characterization of dependence, with $H^*$ describing how mass is distributed across different directions.

Specifically, define the mapping $T:~\mathbb{R}^{n}\setminus\{\mathbf{0}\}\to(0,\infty)\times\punitsphereone{n}$ by
\begin{equation}\label{eq:polar}
T(\mathbf{z})=(r,\boldsymbol{\theta}),
\qquad
\text{where } r=\|\mathbf{z}\|_1
\text{ and }
\boldsymbol{\theta}=\mathbf{z}/\|\mathbf{z}\|_1.
\end{equation}
That is, $r$ is the radial part and $\bm{\theta}$ is the angular part of $\mathbf{z}$. Note that $T$ is a one-to-one and onto mapping (a bijection).
We now formally define the spectral measure $H^*$:
\begin{defn}
The \emph{spectral measure} $H^*$ of an exponent measure $\mu^*$ on
$\punitsphereone{n}$ 
is defined by
\begin{equation}
H^*(B)
=
\mu^{*}\bigl(\{\mathbf{z}\in[\mathbf{0},\bm{\infty}):\|\mathbf{z}\|_1\ge1,\,
\mathbf{z}/\|\mathbf{z}\|_1\in B\}\bigr),
\end{equation}
for Borel subsets $B$ of $\punitsphereone{n}$. We denote as $B\in \mathcal{B}(\punitsphereone{n})$.
\end{defn}
The homogeneity of $\mu^{*}$ in \eqref{eq:mu-homogeneity} implies
\[
\mu^{*}\bigl(\{\mathbf{z}\in[\mathbf{0},\bm{\infty}):\|\mathbf{z}\|_1\ge r,\,
\mathbf{z}/\|\mathbf{z}\|_1\in B\}\bigr)
=
r^{-1}  H^*(B),~~0<r<\infty,~B\in \mathcal{B}(\punitsphereone{n}).
\]
That is, in polar coordinates $(r,\bm{\theta})$, the exponent measure $\mu^*$ can factor into a measure in the radial coordinate and the spectral measure in the angular coordinate. In other words, it has the \emph{polar decomposition}:
\begin{equation*}
\mu^{*}\circ T^{-1}({\rm d}r,{\rm d}\boldsymbol{\theta})
=
r^{-2}\,{\rm d}r\,H^*({\rm d}\boldsymbol{\theta}).
\end{equation*}
Accordingly, for any function $g:[\mathbf{0},\bm{\infty})\setminus\{\mathbf{0}\}\to\R$, it holds that 
\begin{align*}
\int_{[\mathbf{0},\bm{\infty})\setminus\{\mathbf{0}\}} g(\mathbf{z})\,\mu^{*}({\rm d}\mathbf{z})
&=
\int_{\punitsphereone{n}}\int_{0}^{\infty}
g\!\left(r\boldsymbol{\theta}\right)
r^{-2}\,{\rm d}r\,H^*({\rm d}\boldsymbol{\theta}).
\end{align*}
One special case is that for any $\mathbf{z}\in[\mathbf{0},\bm{\infty})$,
\begin{align}
\label{eq:v2-260127-4}
%V_{*}(\mathbf{z})
\mu^{*}\bigl([\mathbf{0},\mathbf{z}]^c\bigr)
&=\int_{[\mathbf{0},\bm{\infty})\setminus\{\mathbf{0}\}}
\mathbb{I}\!\left(
\max_{j\in[n]}(y_j/z_j)>1
\right)\mu^*(\intd\mathbf{y})\notag\\
& =\int_{\punitsphereone{n}}\int_{0}^{\infty}
\mathbb{I}\!\left(
r\max_{j\in[n]}(\theta_j/z_j)>1
\right)r\intd r^{-2} H^*(\intd\bm{\theta})\notag\\
&=\int_{\punitsphereone{n}}\max_{j\in[n]}(\theta_j/z_j)H^*({\rm d}\boldsymbol{\theta}),
\end{align}
where $\mathbb{I}(\cdot)$ is the indicator function.
% and the requirement that the margins of $G_{*}$ are standard Fr\'echet is equivalent to
% \begin{equation*}
% \int_{\Xi}
%  {\omega_j} 
% \, H^*({\rm d}\boldsymbol{\omega})
% =
% 1,
% \qquad
% j=1,\ldots,d.
% \end{equation*}
Combining~\eqref{eq:260127-10} with~(\ref{eq:v2-260127-4}) yields
\[
\ell(\mathbf{v}) 
=\int_{\punitsphereone{n}}
\max_{j\in[n]}
\left(
 {\theta_j}  \, v_j
\right)
\, H^*(\mathrm{d}\boldsymbol{\theta}),
\qquad
\mathbf{v}\in[\mathbf{0},\bm{\infty}),
\]
which is Equation (2) of the main text.

\subsection{Additional background on MRV distributions}
\label{subsec:equi-def-mrv}
According to Chapter 6 of \cite{resnick2007heavy}, if ${\bf X}\in {\rm MRV}_{-\gamma}$, (that is, it satisfies \eqref{eq:define-mrv} in the main text with  limiting function $\lambda$ in Definition \ref{def:mrv},) then     there exists a Radon measure $\nu$ on $[\mathbf{0},\bm{\infty})\setminus\{\mathbf{0}\}$ such that
\begin{equation} \label{eq:260205-1}
\lim_{t\to\infty}
\frac{1-F(t\mathbf{x})}{1-F(t\mathbf{1})}
=
\lim_{t\to\infty}
\frac{\mathbb{P}\!\left[\mathbf{X}/t \in [0,\mathbf{x}]^{c}\right]}
{\mathbb{P}\!\left[\mathbf{X}/t \in [0,\mathbf{1}]^{c}\right]}
=
\nu\!\left([0,\mathbf{x}]^{c}\right),
\end{equation}
for all points $\mathbf{x}\in[\mathbf{0},\bm{\infty})\setminus\{\mathbf{0}\}$ which are continuity points of the function $\nu([0,\cdot]^{c})$. 
It holds that  $\lambda({\bf x})=\nu\!\left([0,\mathbf{x}]^{c}\right)$. 
Following Theorem 6.1 of \citesupp{resnick2007heavy}, for this Radon measure $\nu$, we also have
\begin{equation}
\label{eq:nu-borel-set}
\lim_{t\to\infty}\frac{\Pr\left(\mathbf X\in tA\right)}{\Pr\left(\mathbf X\in t[\mathbf 0,\mathbf 1]^c\right)}
=\nu(A), 
\end{equation}
for any Borel set $A\in[\mathbf{0},\bm{\infty}) \setminus \{\mathbf{0}\}$.

According to Section 2.5 of \cite{resnick2008multivariate}, similar to the exponent measure, the Radon measure $\nu$ also has a homogeneity property
\[
\nu(t\cdot)=t^{-\gamma}\nu(\cdot).
\]
Notice that given the relationship $\lambda({\bf x})=\nu\!\left([0,\mathbf{x}]^{c}\right)$, we also have 
\[
 \lambda(s\mathbf{x})=\nu\!\left(s[0,\mathbf{x}]^{c}\right) = s^{-\gamma}\nu\!\left([0,\mathbf{x}]^{c}\right) = s^{-\gamma}\lambda(\mathbf{x}), 
 \]
as claimed in Section 3.4 of the main text. Additionally, we have the following polar decomposition
\[
\nu\circ T^{-1}({\rm d}r,{\rm d}\boldsymbol{\theta})
=
r^{-\gamma-1}\,{\rm d}r\,H({\rm d}\boldsymbol{\theta}),
\]
where $T$ is the mapping given by \eqref{eq:polar}, and $H$ is called the spectral measure of $\nu$ (or $X$) with the definition that
\[
H(B)
=
\nu\bigl(\{\mathbf{z}\in[\mathbf{0},\bm{\infty}):\|\mathbf{z}\|_1\ge1,\,
\mathbf{z}/\|\mathbf{z}\|_1\in B\}\bigr).
\]

\subsection{MRV distributions vs MRV copulas}
It is worth noting that the spectral measure $H$ of the MRV distribution is not the same as the spectral measure $H^*$ of its (MRV) copula. In this section, we will connect the radon measure $\nu$ with the exponent measure $\mu^*$ and spectral measure $H^*$ of its MRV copula.

\begin{restatable}{prop}{Intermediate}
\label{prop:measure-equation}
Let ${\bf X}\in {\rm MRV}_{-\gamma}$ with $\gamma>0$ and $\nu(\cdot)$ is its Radon measure satisfying \eqref{eq:260205-1}. $\mu^*$ and $H^*$ are the exponent measure and spectral measure of its MRV copula. Then, for any Borel set $B\in[\mathbf{0},\bm{\infty})\setminus\{\mathbf{0}\}$, 
\begin{equation}
\label{eq:measure-equation}
\nu(B) = \mu^*\left(\left\{\mathbf{x}:~\left(\mathbf{c}\odot\mathbf{x}\right)^{1/\gamma}\in B\right\}\right).
\end{equation} 
For specific sets,  we have
\begin{align}
\nu([\mathbf 0,{\bf x}]^c)
=&\mu^*\left([\mathbf 0,(c_1^{-1}x_1^\gamma,\dots,c_n^{-1}x_n^\gamma)]^c\right)\notag\\
=&{  \int_{\punitsphereone{n}} \max_i \{ c_ix_i ^{-\gamma} \theta_i\} H^*(\mathrm{d}{\bm\theta})},\label{eq:cube}\\
\nu(\{\mathbf{x}: \|\mathbf{x}\|_1>n \})
=&\mu^*( \{ \mathbf{x}: \|( {\mathbf c}\odot \mathbf{x})^{^{1/\gamma}}\|_1>n\})\notag\\
=& n^{-\gamma}\int_{\punitsphereone{n}}
\left(\sum_{i=1}^n(c_i{\theta}_i)^{1/\gamma}\right)^\gamma H^*(\mathrm{d}\bm{\theta}),\label{eq:ball}
\end{align}
where  $c_i := \lim_{t\to\infty}\frac{\Pr\left(X_i>t\right)}{\Pr\left(\mathbf X\in t[\mathbf 0,\mathbf 1]^c\right)}
=\nu(B_i)$ with $B_i=\{\mathbf{x}\in\R_+^n:x_i>1\}$, $i\in [n]$. Here, when some \(c_i=0\), we use the convention \(c_i^{-1}=\infty\) and   we interpret
\[
    [\mathbf 0,(c_1^{-1}x_1^\gamma,\dots,c_n^{-1}x_n^\gamma)]^c
    =
    \{{\bf z}\in[{\bf 0},\boldsymbol{\infty}): z_i>c_i^{-1}x_i^\gamma
      \text{ for some } i \text{ with } c_i>0\}.
\]
\end{restatable}

To further differentiate between $H$ and $H^*$, we see the following comparison. For the spectral measure $H$ of $\nu$, by definition,
\[
\nu([{\bf 0},{\bf x}]^c)= \int_{\punitsphereone{n}}
\max_{i\in[n]} \left(\theta_i^\gamma x_i^{-\gamma}\right)\,H(\mathrm{d}{\boldsymbol{\theta}}).
\]
On the other hand, we have
\[
\nu([{\bf 0},{\bf x}]^c)
=\ell(c_1x_1^{-\gamma},\dots,c_nx_n ^{-\gamma})
= {\int_{\punitsphereone{n}}
\max_{i\in[n]}
\left(
 {\theta_i}  \, c_ix_i^{-\gamma}
\right)
\, H^*(\mathrm{d}\boldsymbol{\theta})},
\]
where $c_i = \nu(B_i)$ with $B_i=\{\mathbf{x}\in\R_+^n:x_i>1\}$, $i\in [n]$.
Hence, it follows that 
\[
\int_{\punitsphereone{n}}
\max_{i\in[n]} \left(\theta_i^\gamma x_i^{-\gamma}\right)\,H(\mathrm{d}{\boldsymbol{\theta}})
=
{\int_{\punitsphereone{n}}
\max_{i\in[n]}
\left(
 {\theta_i}  \, c_ix_i^{-\gamma}
\right)
\, H^*(\mathrm{d}\boldsymbol{\theta})},
~~\forall \, {\bf x}>{\bf 0}.
\]

\section{Another partial ordering on MRV copulas}
\label{sec:dependence-vs-combination-test}
In this section, we introduce a new partial ordering on MRV copulas to compare the strength of asymptotic dependence, which enables stronger theoretical results on how dependence impacts the conservativeness and power of heavy-tailed combination tests when $\gamma < 1$. Specifically, we define a partial order based on the convex order of the spectral measures associated with MRV copulas:

% To characterize the order of dependence level, we first introduce a partial order to compare the degree of asymptotic positive dependence among MRV copulas, defined via the convex order on their associated spectral measures.

\begin{defn}[Convex order]
    For any two measures $H_1$ and $H_2$ on the same space $\Theta$, \emph{$H_1$ dominates $H_2$ in convex order}, denoted by $H_1\ge_{\rm cx} H_2 $, holds if $\int_\Theta   \phi({\mathbf \theta}) {\rm d} H_1({\mathbf \theta}) \ge \int_\Theta \phi({\mathbf \theta}) {\rm d} H_2({\mathbf \theta})$ for any convex function $\phi$.
\end{defn}
 
\begin{defn}[Partial order between MRV copulas]
Let $C$ and $C'$ be MRV copulas with spectral measures $H^*$
and $H^{**}$, respectively. We say that $C'$ \emph{dominates} $C$, denoted by $C\preceq C'$, if
$$ H^*\ge_{\rm cx} H^{**}.$$
\end{defn}
Intuitively, $C\preceq C'$ corresponds to $C$ being less asymptotically positively dependent than $C'$. Similar to \Cref{prop:concordance-order-upper-and-lower-bound} in the main text, we can prove that all MRV copulas are between asymptotic independence and asymptotic complete dependence in terms of this new partial order:
\begin{restatable}{prop}{PartialOrder}
\label{prop:partial-order}
Let $C_\mathrm{indep}$ and $C_{\mathrm{cdep}}$ be the MRV copulas under asymptotic independence and asymptotic complete dependence, respectively. Then for any MRV copula $C$, we have
$$C_\mathrm{indep} \preceq C \preceq C_{\mathrm{cdep}}.$$
%[\hyperref[subsec:other-theory-results]{Proof}].
\end{restatable}
As established in \Cref{prop:sum-tail-general-norm}, the asymptotic dependence structure characterized by the MRV copula governs the tail behavior of the $L_1$ norm of the MRV random vector. Furthermore, the convex ordering of the associated spectral measure and the convexity of  $\bm{\theta}\mapsto {{\|(\mathbf{c}\odot\bm{\theta})^{1/\gamma}\|_1}^\gamma}$ for $\gamma\in (0,1]$ directly induce  this monotonicity:
\begin{restatable}{prop}{QVSH}
%\begin{prop}
\label{prop:q-vs-H}
Suppose that the MRV copula $C$ has spectral measure $H^*$
and the MRV copula $C'$ has spectral measure $H^{**}$.
If $C\succeq C'$, then $h(\gamma,H^*) \le h(\gamma,H^{**}) $ for $\gamma\in (0,1].$
\end{restatable}
According to the proof of \Cref{thm:combination-test-valid},
\[
h(\gamma,H^*)=n^{1-\gamma}\lim_{\alpha\downarrow0}\frac{\Pr_{H_0^{\mathrm{global}}}\left(P^{F,\bm{\omega}}_{\mathrm{comb}}\le\alpha\right)}{\alpha}
\]
%\wang{This is unclear. I think the notation $q(\gamma, H^\star)$ does not depend on weights, while the right-hand side depends on weights.}
characterizes the asymptotic type-I error of the heavy-tailed combination test, with $\boldsymbol{\omega}=(1,\dots,1)$. Consequently, a larger partial order of the MRV copula leads to a smaller type-I error, indicating a more conservative test. We can also obtain the same monotonicity property for power, similar to how we establish \Cref{thm:bonferroni-more-conservative-with-dependence-power} after \Cref{thm:bonferroni-more-conservative-with-dependence}.

Unfortunately, establishing the partial order $C\succeq C'$ for specific forms of MRV copulas is more challenging than the pointwise order. In low-dimensional cases or certain high-dimensional cases, the order is tractable. For example, when $n = 2$, the order $C \succeq C'$  corresponds to the pointwise inequality $\int_0^x H^*(t,1){\rm d}t\le \int_0^x H^{**}(t,1) {\rm d}t$ for $x\in [0,1]$.\footnote{Note that when $n=2$, $H^*$ is a measure on $\{(x_1,x_2)\in\R_+^2: x_1+x_2=1\}$. We have $H^*\ge_{\rm cx}  H^{**}$  if and only if  $H_1^*\ge_{\rm cx}  H_1^{**}$, where $H_1^*$ is the first marginal distribution of $H^*$. Therefore, this is again equivalent to $\int_0^x H^*(t,1){\rm d}t\le \int_0^x H^{**}(t,1) {\rm d}t$ for $x\in [0,1]$.} 
 %$H(\mathbf{u}) \ge H'(\mathbf{u})$ for all $\mathbf{u} \in [0,1]^2$ (\cite{a})
For higher dimensions ($n > 2$), similar pointwise ordering can hold when independent pairs of components exist. However, general theoretical results on ordering between spectral measures in dimensions $n > 2$ are limited and technically difficult to obtain (see, e.g., \citesupp{mao2015relations}). For instance, it remains unclear whether increasing the entries of the correlation matrix in a multivariate t-copula leads to a spectral measure that is larger in the convex order.

Some special cases nevertheless exhibit the monotonicity described in \Cref{prop:q-vs-H}. For example, the Clayton copula admits this monotonicity property in general dimensions, although the argument does not directly follow from the convex order of spectral measures.

Empirically, our simulations in \Cref{subsec:empirical-validity} show that $h(\gamma,H^*)$ tends to decrease as dependence increases with a fixed $\gamma < 1$, suggesting that the test becomes more conservative. This observation points to a promising direction for future work: developing a more refined order of MRV copulas and its implications for the tail behavior of the MRV random vector.

\section{Theoretical results in the examples}

\subsection{MRV of absolute multivariate Gaussian and multivariate t}
We show that multivariate Gaussian and multivariate t-distributed random vectors retain MRV copulas after taking absolute values, a property needed for Examples 3.3 and 3.4 in the main text.
Specifically, we first show that the absolute value transformation preserves asymptotic independence for multivariate Gaussian variables. For the multivariate $t$ distribution, we show that the copula of the absolute-valued variables continues to exhibit multivariate regular variation.
\begin{restatable}{prop}{AbsoluteGauss}
\label{prop:abs-preserve-independence}
Let $\mathbf{T}=(T_1,\dots,T_n)$ follow a multivariate Gaussian distribution with a nondegenerate correlation matrix. Then $(|T_1|,\dots,|T_n|)$ is asymptotically independent.
\end{restatable}

\begin{restatable}{prop}{Absolutet}
\label{prop:absolute-t-mrv-copula}
Let $\mathbf{T}=(T_1,\ldots,T_n)$ follow a multivariate
$t$ distribution with degrees of freedom $\nu>0$, location parameter
$\mathbf u\in\mathbb R^n$, and positive definite scale matrix $\Sigma$.
Then the copula of
$(|T_1|,\ldots,|T_n|)$
is an MRV copula.
\end{restatable}

% \begin{restatable}{prop}{OverlappingGeneSet}
% \label{prop:overlapping-gene-set-mrv}
\subsection{MRV for overlapping gene set p-values}
\label{sec:s42}

Let $\mathbf T=(T_1,\ldots,T_m)\sim \mathrm{N}_m(\mathbf0,R)$, where $R$ is a positive definite correlation matrix, and let $A_1,\ldots,A_n\subseteq[m]$, with $m_i:=|A_i|\ge1$.
Take $U_j=\overline{\Phi}(T_j)$ for all $j$ (one-sided) or $U_j=2\overline{\Phi}(|T_j|)$ for all $j$ (two-sided).
For each $A_i$, choose minimum-$p$, Simes, or a heavy-tailed combination, with corresponding statistics
\[
D_i=\begin{cases}
\displaystyle\min_{j\in A_i}U_j, & \text{minimum-}p,\\[2pt]
\displaystyle\min_{1\le r\le m_i}\frac{m_iU_{i,(r)}}{r}, & \text{Simes},
\end{cases}
\qquad S_i=\sum_{j\in A_i}Q_F(1-U_j),
\]
where $U_{i,(r)}$ is the $r$th smallest $U_j$ in $A_i$, and $F\in\mathscr R_{-\gamma}$, $\gamma>0$, is continuous and satisfies $F(-x)\le\overline{F}(x)$ for all sufficiently large $x$. All probabilities below are under the global null.
Let $G_i(x)=\Pr_0(D_i\le x)$ and $H_i(x)=\Pr_0(S_i\le x)$ denote the exact null CDFs. Define
\begin{equation}\label{eq:gene-calibrated-pvalues}
P_i=\begin{cases}
G_i(D_i), & \text{minimum-}p\text{ or Simes},\\
1-H_i(S_i), & \text{heavy-tailed combination}.
\end{cases}
\end{equation}
For minimum-$p$ and Simes, respectively, one may instead use the direct outputs
$P_i=\min\{1,m_iD_i\}$ (Bonferroni) and $P_i=D_i$, without applying $G_i$.

\begin{restatable}{prop}{OverlappingGeneSet}
\label{prop:gaussian-groups}
For any choice of $P_i$ in this section (\ref{sec:s42}), $\Pr_0(P_i\le t)\sim t$ as $t\downarrow0$,
the survival copula of $(P_1,\ldots,P_n)$ is MRV, and for $i\ne k$,
\[
\lim_{s\downarrow0}\Pr_0(P_i\le s\mid P_k\le s)
=\frac{|A_i\cap A_k|}{\max\{m_i,m_k\}}.
\]
\end{restatable}

For direct outputs, the copula statement uses marginal standardization near zero; clipping at $1$ is immaterial. Direct Bonferroni is valid at every level, whereas direct Simes need not be under general $R$. 
The condition on $F$ includes every distribution in $\mathscr R_{-\gamma}^{*}$, as well as the standard Cauchy distribution.

% \subsection{Asymptotic Validity of Cauchy Combination Test}
% \label{sec:validity-vanilla-cauchy}
% The asymptotic validity of the Cauchy combination test~\citep{liu2020cauchy} is a straightforward conclusion from \Cref{thm:combination-test-valid} in the main text.

% \begin{restatable}{prop}{ValidCauchy}
% \label{prop:validity-vanilla-cauchy}
% Assume that under the global null, the joint distribution of $(1-P_1, \dots, 1-P_n)$ follows an MRV copula and has standard
% uniform marginals. Then the Cauchy Combination test is asymptotic valid. That is,
% \[
% \lim_{\alpha\downarrow0}\frac{\Pr_{H_0^{\mathrm{global}}}\left(P_{\mathrm{comb}}^{\mathrm{CCT}}\le\alpha\right)}{\alpha}\le1.
% \]
% \end{restatable}

\subsection{Type-A and type-B alternatives}
\label{sec:p-alternative}
We verify that the Type-A and Type-B alternatives defined in \Cref{exmp:alternative_types} satisfy \Cref{asmp:p-alternative}. We also characterize the set of dominating signals for each alternative type.

We first consider the Type-A alternative, where p-values are generated from continuous CDFs within the following class:
\begin{defn}
\label{def:certain-distributions}
Let $\mathscr{D}$ denote the class of continuous distributions whose survival functions $\overline{G}$ satisfy one of the following conditions:
\begin{enumerate}[label=(\roman*)]
\item $\lim_{x\to\infty}\frac{\overline{G}(x-\mu)}{\overline{G}(x)}\in(0,\infty)$.
% \item there exists a constant $b>0$ such that $\lim_{x\to\infty}\frac{\overline{D}(x)}{e^{-bx}}\in(0,\infty)$.
\item $G$ is the standard normal distribution.
\end{enumerate}
\end{defn}
We now establish that Type-A alternatives with $G \in \mathscr{D}$ satisfy \Cref{asmp:p-alternative} and characterize their dominating signals.

\begin{restatable}{prop}{AlternativeA}
\label{prop:alternative-a}
Suppose $C$ is an MRV copula and $\mathbf{U}\sim C$. The p-vector $\mathbf{P}$ is given by
\[
(1-P_1,\dots,1-P_n) \overset{d}{=}(G_1(G_1^{-1}(U_1)+\mu_1),\dots, G_n(G_n^{-1}(U_n)+\mu_n)),
\]
where $G_1,\dots,G_n$ are in $\mathscr{D}$, and $\mu_1,\dots,\mu_n\ge0$. Then the p-vector $\mathbf{P}$ meets \Cref{asmp:p-alternative} and $\beta=1$. Moreover, the dominating signals 
\begin{enumerate}[label=(\roman*)]
\item $I_{\mathbf{P}}=[n]$ when $G_1,\dots,G_n$ belong to $\mathscr{D}$ (i),
\item $I_{\mathbf{P}}=\{i:\mu_i=\max_j{\mu_j}\}$  when $G_1,\dots,G_n$ belong to $\mathscr{D}$ (ii).
\end{enumerate}
\end{restatable}

We now turn to the Type-B alternative and establish that it satisfies \Cref{asmp:p-alternative}. Additionally, we characterize the set of dominating signals under this alternative. The key properties follow directly from the construction of Type-B alternatives, where p-values are generated through Beta-type transformations of uniform random variables.
\begin{restatable}{prop}{AlternativeB}
\label{prop:alternative-b}
Suppose $C$ is an MRV copula and $\mathbf{U}\sim C$. The p-vector $\mathbf{P}$ satisfies
\[
    (1-P_1,\dots,1-P_n) \overset{d}{=} (1-(1-U_1)^{\beta_1},\dots,1-(1-U_n)^{\beta_n}),
\]
where $\beta_1\ge\dots\ge\beta_n\ge1$. Then the p-vector $\mathbf{P}$ meets \Cref{asmp:p-alternative}, $\beta=1/\beta_1$, and the dominating signals $I_{\mathbf{P}}=\{i\in[n]:\beta_i=\max_j\beta_j\}$.
\end{restatable}

\section{Proofs of theoretical results}
\subsection{Proof of \Cref{{prop:limit-preserve-concordance-order}}}
\Limitorder*
\begin{proof}
    Based on the definition of pointwise order, we have for any $(u_1,\dots, u_n)\in [\bm 0, \bm 1]\subseteq\mathbb{R}^n$,
    $$C_1(u_1, \dots, u_n)\leq C_2(u_1,\dots, u_n).$$
    Thus, based on the definition of the limiting copula, we have 
    $$C_1^\star(u_1, \dots, u_n)=\lim_{t\to+\infty}C_1(u_1^{1/t},\dots, u_n^{1/t})^t\leq \lim_{t\to+\infty}C_2(u_1^{1/t},\dots, u_n^{1/t})^t = C_2^\star(u_1, \dots, u_n),$$
    which finishes the proof.
\end{proof}

\subsection{Proof of \Cref{prop:concordance-order-upper-and-lower-bound}}
\label{subsec:proof-concordance-order-upper-and-lower-bound}
\ConcordanceOrder*
\begin{proof}
Since the limiting copula can be written as
\[
C^*(u_1,\dots,u_n)=\exp\left\{-\ell(-\log u_1,\dots,-\log u_n)\right\},
\]
\eqref{eq:concordance-order-upper-and-lower-bound} is equivalent to
\begin{equation}
\label{eq:concordance-order-by-l}
\max(x_1,\dots,x_n)\le \ell(x_1,\dots,x_n)\le x_1+\cdots+x_n,\quad x_1,\dots,x_n>0,
\end{equation}
which is a well-known property of stable tail dependence function $\ell$; see (L3) on Page 257 of \citesupp{beirlant2006statistics}.
Therefore, Equation \eqref{eq:concordance-order-upper-and-lower-bound} in the proposition of the main text immediately follows.
\end{proof}

\subsection{Proof of \Cref{prop:copula-to-mrv-distribution}}
\label{subsec:proof-copula-to-mrv-distribution}

 \CopulaToMRV*

\begin{proof}
Since there exists a $F_0\in\mathscr{R}_{-\gamma}$ with $\gamma>0$ such that
\[
\lim_{t\to\infty} \frac{\overline{F}_i(t)}{\overline{F}_0(t) }  = c_i,~~i=1,\dots,n
\]
For any $x>0$, it holds that
\begin{align}
     \lim_{t\to\infty} \frac{\overline{F}_i(tx)}{\overline{F}_0(t) }  = c_ix^{-\gamma},~~x>0, \label{eq:250218-3}
\end{align}
and the convergence is uniform for $x\ge x_0>0$.
Also note that by the definition of the stable tail dependence function $\ell$ and the fact that $\ell$ is a (convex and thus) continuous function (page 275 of \citesupp{beirlant2006statistics}), 
it follows that  for any $x_1,\dots,x_n$ where $\min_i x_i > 0$:
\begin{equation}\label{eq:tail_ratio}
    \begin{aligned}
&\lim_{t\to+\infty}\frac{1-F_{\mathbf{X}}(tx_1,\dots,tx_n)}{1-F_{\mathbf{X}}(t,\dots,t)} 
=  \lim_{t\to\infty}\frac{1-C(F_1(tx_1),\dots,F_n(tx_n))}{1-C(F_1(t),\dots,F_n(t))}\\
&= \lim_{t\to\infty}\frac{1-C(1-\overline{F}_1(tx_1),\dots,1-\overline{F}_n(tx_n))}{1-C(1-\overline{F}_1(t),\dots,1-\overline{F}_n(t))}\\
&= \lim_{t\to\infty}\frac{1-C(1-c_1x_1^{-\gamma}\overline{F}_0(t),\dots,1-c_nx_n^{-\gamma}\overline{F}_0(t))}{1-C(1-c_1\overline{F}_0(t),\dots,1-c_n\overline{F}_0(t))}\\ 
&= \lim_{t\to\infty}\frac{1-C(1-c_1x_1^{-\gamma}\overline{F}_0(t),\dots,1-c_nx_n^{-\gamma}\overline{F}_0(t))}{ \overline{F}_0(t)}\\
&\qquad\times\lim_{t\to\infty}\frac{ \overline{F}_0(t)} {1-C(1-c_1\overline{F}_0(t),\dots,1-c_n\overline{F}_0(t))}\\[3pt]
& = \frac{\ell(c_1x_1^{-\gamma},\dots,c_nx_n ^{-\gamma})} {\ell(c_1,\dots,c_n)},
\end{aligned}
\end{equation}
 where the third equality follows from the monotonicity of copula and the uniform convergence of regularly varying  function, i.e., the convergence of (\ref{eq:250218-3}). 
 Letting $\mathbf{X}\sim F_{\mathbf{X}}$, then for $\mathbf{x}=(x_1,\cdots, x_n)\in (\mathbf{0}, \bm{\infty}) $, we have
 \begin{equation*}
\lim_{t\to\infty}\frac{\Pr(\mathbf{X}\in t[\mathbf{0},\mathbf{x}]^c)}{\Pr(\mathbf{X}\in t[\mathbf{0},\mathbf{1}]^c)}=\lim_{t\to+\infty}\frac{1-F_{\mathbf{X}}(tx_1,\dots,tx_n)}{1-F_{\mathbf{X}}(t,\dots,t)} =\frac{\ell(c_1x_1^{-\gamma},\dots,c_nx_n ^{-\gamma})} {\ell(c_1,\dots,c_n)} \overset{\Delta}{=}\lambda(\mathbf{x}).
\end{equation*}
Also, we have for any $s>0$,
$$\lambda(s\mathbf{x})=\frac{\ell(c_1s^{-\gamma}x_1^{-\gamma},\dots,c_ns^{-\gamma}x_n ^{-\gamma})} {\ell(c_1,\dots,c_n)}= \frac{s^{-\gamma}\ell(c_1x_1^{-\gamma},\dots,c_nx_n ^{-\gamma})}{\ell(c_1,\dots,c_n)}= s^{-\gamma}\lambda(\mathbf{x}).$$
 This completes the proof.
\end{proof}

\subsection{Proof of \Cref{prop:sum-tail-general-norm}}
\label{subsec:proof-sum-tail-general-norm}
\SumTailGeneralNorm*
\begin{proof} 
Since $\mathbf{X}\in\mathrm{MRV}_{-\gamma}$, by \eqref{eq:nu-borel-set} in \Cref{subsec:equi-def-mrv}, there exists a Radon measure $\nu$ on $[\mathbf{0},\bm{\infty})\setminus\{\mathbf{0}\}$ such that for any Borel set $A\in[\mathbf{0},\bm{\infty}) \setminus \{\mathbf{0}\}$,
\[
\lim_{t\to\infty}\frac{\Pr\left(\mathbf X\in tA\right)}{\Pr\left(\mathbf X\in t[\mathbf 0,\mathbf 1]^c\right)}
=\nu(A).
\]
Therefore, taken in $A=\{X_i>1\}$ and $A=\{\mathbf{x}: \|x\|_1>n\}$, we have both
\[
c_i := \lim_{t\to+\infty}\frac{\Pr(X_i>t)}{{\Pr\left(\mathbf X\in t[\mathbf 0,\mathbf 1]^c\right)}}\ge0
\]
exists, and
\[
\lim_{t\to+\infty}\frac{\Pr(\sum_{i=1}^n X_i > nt)}{{\Pr\left(\mathbf X\in t[\mathbf 0,\mathbf 1]^c\right)}}
=\nu\left(\{\mathbf{x}: \|\mathbf{x}\|_1>n\}\right)
\]
exist. 
Then the target limit can be rewritten as follows:
\begin{align*}
h(\gamma,H^*) &=\lim_{t\to+\infty}\frac{\Pr(\bar X>t)}{\sum_{i=1}^n\Pr(X_i>t)}\\
&=\lim_{t\to+\infty}\frac{\Pr(\sum_{i=1}^n X_i > nt)}{{\Pr\left(\mathbf X\in t[\mathbf 0,\mathbf 1]^c\right)}}\bigg/ \lim_{t\to+\infty}\sum_{i=1}^n\frac{\Pr(X_i>t)}{{\Pr\left(\mathbf X\in t[\mathbf 0,\mathbf 1]^c\right)}}\\
&  = \nu\left(\{\mathbf{x}: \|\mathbf{x}\|_1>n\}\right) \times \left( \sum_{i=1}^nc_i \right)^{-1}\\
&=\frac{\int_{\punitsphere{n}}{\|(\mathbf{c}\odot\bm{\theta})^{1/\gamma}\|_1}^\gamma H^*(\mathrm{d}\bm{\theta})}{\sum_{i=1}^nc_i},
\end{align*}
where the last equation plugs in Equation \eqref{eq:ball} of \Cref{prop:measure-equation}.
This completes the proof.
\end{proof}

\subsection{A common lemma for all theorems}
\label{subsec:moving-constant-does-not-affect-tail}
We intend to prove that moving a constant does not affect the tail properties of an MRV random vector with the following lemma. With this lemma, without loss of generality, we can assume that the lower bound $c$ of $\mathrm{supp}(F)$ is $0$ for all theorems below.

\begin{lem}
\label{lemma:mrv-robust-to-shift}
Suppose $\mathbf{X}\in\mathrm{MRV}_{-\gamma}$ with $\gamma>0$. For any fixed constant $c\in\R$, it holds that
\begin{align}
&\lim_{t\to\infty}\frac{\Pr\left(\sum_{i=1}^n (X_i+c)>t\right)}{\sum_{i=1}^n\Pr\left(X_i+c>t\right)}
=\lim_{t\to\infty}\frac{\Pr\left(\sum_{i=1}^n X_i>t\right)}{\sum_{i=1}^n\Pr\left(X_i>t\right)}
=\lim_{t\to\infty}\frac{\Pr\left(\frac{1}{n}\sum_{i=1}^n X_i>t\right)}{n^{-\gamma}\sum_{i=1}^n\Pr\left(X_i>t\right)},\\
&\lim_{t\to\infty}\frac{\Pr\left(\max_{i\in[n]} X_i+c>t\right)}{\sum_{i=1}^n\Pr\left(X_i+c>t\right)}
=\lim_{t\to\infty}\frac{\Pr\left(\max_{i\in[n]} X_i>t\right)}{\sum_{i=1}^n\Pr\left(X_i>t\right)}.
\end{align}
\end{lem}

\begin{proof}
We use $I_{\mathbf{X}}$ to denote the set of dominant tails of $\mathbf{X}$. That is,
\[
I_{\mathbf{X}}=\{i\in[n]:~c_i\neq0\},~\mathrm{where}~c_i=\lim_{t\to\infty}\frac{\Pr(X_i>t)}{\Pr(\mathbf{X}\in t[\mathbf{0},\mathbf{1}]^c)},~i=1,\dots,n.
\]
The conclusion for the maximum is straightforward:
\begin{align*}
&\lim_{t\to\infty}\frac{\Pr(\max_{i\in[n]} X_i+c>t)}{\sum_{i=1}^n\Pr(X_i+c>t)}
=\lim_{t\to\infty}\frac{\Pr(\max_{i\in[n]} X_i>t-c)}{\sum_{i=1}^n\Pr(X_i>t-c)}
=\lim_{t\to\infty}\frac{\Pr(\max_{i\in[n]} X_i>t)}{\sum_{i=1}^n\Pr(X_i>t)}.
\end{align*}
For the sum, the first equality is due to 
\begin{align*}
&\lim_{t\to\infty}\frac{\Pr(\sum_{i=1}^n (X_i+c)>t)}{\sum_{i=1}^n\Pr(X_i+c>t)}
=\lim_{t\to\infty}\frac{\Pr(\sum_{i=1}^n X_i>t - n c)}{\sum_{i=1}^n\Pr(X_i>t-c)}\\
&=\lim_{t\to\infty}\frac{\Pr( X_1+\dots+ X_n>t - n c)}{\sum_{i=1}^n\Pr( X_i>t - n c)}\times\lim_{t\to\infty}\frac{\sum_{i=1}^n\Pr(X_i>t - n c)}{\sum_{i=1}^n\Pr(X_i>t-c)}\\
&=\lim_{t\to\infty}\frac{\Pr(X_1 + \dots + X_n > t)}{\sum_{i=1}^n\Pr(X_i>t)}
\times\lim_{t\to\infty}\frac{\sum_{i\in I_{\mathbf{X}}}\Pr(X_i>t - n c)}{\sum_{i\in I_{\mathbf{X}}}\Pr(X_i>t-c)} \\
&=\lim_{t\to\infty}\frac{\Pr(X_1 + \dots + X_n > t)}{\sum_{i=1}^n\Pr(X_i>t)},
\end{align*}
where the last equality is because all $X_i$s with $i\in I_{\mathbf{X}}$ have regularly varying tails. 
The second equality also holds:
\begin{align*}
&\lim_{t\to\infty}\frac{\Pr(\sum_{i=1}^n X_i>t)}{\sum_{i=1}^n\Pr(X_i>t)}
=\lim_{t\to\infty}\frac{\Pr(\sum_{i=1}^n X_i>t )}{\sum_{i\in I_{\mathbf{X}}}\Pr(X_i>t)}\\
&=\lim_{t\to\infty}\frac{\Pr(\frac{1}{n}\sum_{i=1}^n X_i>t)}{\sum_{i\in I_{\mathbf{X}}}\Pr( X_i>t)}\times\lim_{t\to\infty}\frac{\sum_{i\in I_{\mathbf{X}}}\Pr(X_i>t)}{\sum_{i\in I_{\mathbf{X}}}\Pr(X_i>nt)}\\
&=\lim_{t\to\infty}\frac{\Pr(\frac{1}{n}\sum_{i=1}^n X_i > t)}{n^{-\gamma}\sum_{i\in I_{\mathbf{X}}}\Pr(X_i>t)}
=\lim_{t\to\infty}\frac{\Pr(\frac{1}{n}\sum_{i=1}^n X_i > t)}{n^{-\gamma}\sum_{i=1}^n\Pr(X_i>t)},
\end{align*}
where the first, third, and last equalities all come from $X_i\in\mathcal{R}_{-\gamma}$ for all $i\in I_{\mathbf{X}}$.
Then the lemma follows.
\end{proof}

\subsection{Proof of \Cref{thm:combination-test-valid}}
\label{subsec:proof-heavy-tailed-test-validity}
% \begin{lem}
% \label{lemma:mrv-invariant-to-move}
% Suppose a n-dimensional random vector $\mathbf{X}\in\mathrm{MRV}_{-\gamma}$, $\gamma>0$. For any constant vector $\mathbf{c}\in\R^n$ and any $\mathbf{x}\subset\R^n$,
% \[
% \lim_{t\to+\infty}\frac{\Pr(\mathbf{X}+\mathbf{c}\in t[\mathbf{0},\mathbf{x}]^c)}{\Pr(\mathbf{X}\in t[\mathbf{0},\mathbf{1}]^c)}
% =\lim_{t\to+\infty}\frac{\Pr(\mathbf{X}\in t[\mathbf{0},\mathbf{x}]^c)}{\Pr(\mathbf{X}\in t[\mathbf{0},\mathbf{1}]^c)}
% \]
% \end{lem}
\AsymptoticValidity*
\begin{proof}
%For simplicity, we use $\mathbf{P}=(P_1,\dots,P_n)$ to 
Denote by $C$ the copula of  $\mathbf{1}-\mathbf{P}$.
Under the global null, each $P_i$ has a standard uniform distribution. Thus,
\begin{align*}
&\lim_{\alpha\downarrow0}\frac{\sum_{i=1}^n\Pr_0\left(P_i<\omega_i\overline{F}\left(n Q_F(1-\frac{\alpha}{n^{1-\gamma}})\right)\right)}{\alpha}
=\lim_{\alpha\downarrow0}\frac{\sum_{i=1}^n \omega_i\overline{F}\left(n Q_F(1-\frac{\alpha}{n^{1-\gamma}})\right)}{\alpha}\\
&=\lim_{\alpha\downarrow0}\frac{\sum_{i=1}^n \omega_i n^{-\gamma}\overline{F}\left(Q_F(1-\frac{\alpha}{n^{1-\gamma}})\right)}{\alpha}
=\lim_{\alpha\downarrow0}\frac{\sum_{i=1}^n \omega_i\frac{\alpha}{n}}{\alpha}=1.
\end{align*}
Hence, for any $\gamma>0$,
\begin{equation}
\label{eq:rewrite-q}
\begin{split}
q(\gamma)&=\lim_{\alpha\downarrow0}\frac{\Pr_0\left(P_\mathrm{comb}^{F, \bm \omega}\leq \alpha\right)}{\alpha}\\
&=\lim_{\alpha\downarrow0}\frac{\Pr_0\left(\frac{1}{n}\sum_{i=1}^n Q_F\left((1 - P_i/\omega_i)^+\right)> Q_F(1-\frac{\alpha}{n^{1-\gamma}})\right)}{\sum_{i=1}^n\Pr_0\left(Q_F\left((1 - P_i/\omega_i)^+\right)>n Q_F(1-\frac{\alpha}{n^{1-\gamma}})\right)}\\
&=\lim_{t\to\infty}\frac{\Pr_0\left(\sum_{i=1}^n Q_F\left((1 - P_i/\omega_i)^+\right)>t\right)}{\sum_{i=1}^n\Pr_0\left(Q_F\left((1 - P_i/\omega_i)^+\right)>t\right)}.  
\end{split}
\end{equation}
By \Cref{lemma:mrv-robust-to-shift}, without loss of generality, we assume the lower bound of $F$'s support is 0.
% Next, we define shifted transformed test statistics $(\tilde X_1, \dots, \tilde X_n)$ so that their support lies in $[\mathbf{0}, \infty)$, enabling a direct characterization via an MRV distribution. Specifically, recall that the support of the transformation function $F$ is $[c,\infty)$. Define the shifted distribution $\tilde{F}(x)=F(x-c)$ which is supported on $[0, \infty)$ and has quantile function $Q_{\tilde{F}}(p) = Q_F(p) - c$. Define the shifted transformed statistics $\tilde{X}_i=Q_F\left((1-P_i/\omega_i)^+\right)-c=Q_{\tilde{F}}\left((1-P_i/\omega_i)^+\right)$.
By definition, 
\begin{align*}
&\Pr_0\left({X}_1\le x_1,\dots,{X}_n\le x_n\right)\\
&=\Pr_0\left(Q_{{F}}\left((1-P_1/\omega_1)^+\right)\le x_1,\dots,Q_{{F}}\left((1-P_n/\omega_n)^+\right)\le x_n\right)\\
&=\Pr_0\left((1-P_1/\omega_1)^+ \le {F}(x_1),\dots,(1-P_n/\omega_n)^+\le{F}(x_n)\right)\\
&=\Pr_0\left(P_1/\omega_1 \ge 1-{F}(x_1),\dots,P_n/\omega_n \ge 1-{F}(x_n)\right)\\
&=\Pr_0\left(1-P_1 \le 1-\omega_1\left(1-{F}(x_1)\right),\dots,1-P_n \le 1-\omega_n\left(1-{F}(x_n)\right)\right)\\
&=C\left(1-\omega_1\left(1-{F}(x_1)\right),\dots,1-\omega_n\left(1-{F}(x_n)\right)\right).
\end{align*}
Thus, the random vector ${\mathbf{X}}=\left({X}_1,\dots,{X}_n\right)$ has the MRV copula $C$ and each marginal satisfies
\begin{equation}
\label{eq:validity-define-ci}
c_i=\lim_{t\to+\infty}\frac{\omega_i \left(1-{F}(t)\right)}{1-{F}(t)}=\omega_i>0,
\end{equation}
where ${F}\in\mathscr{R}_{-\gamma}^*$ with the support $[\mathbf{0},\mathbf{\infty})$. Then 
following \Cref{prop:copula-to-mrv-distribution}, the random vector ${\mathbf{X}}\in\mathrm{MRV}_{-\gamma}$ and all its marginals belong  to $\mathscr{R}_{-\gamma}$. Following \eqref{eq:rewrite-q},
\begin{equation*}
\label{eq:change_x_to_tilde_x_validity}
\begin{split}
&q(\gamma)
=\lim_{t\to\infty}\frac{\Pr_0(\sum_{i=1}^n X_i>t)}{\sum_{i=1}^n\Pr_0( X_i>t)}
% =\lim_{t\to\infty}\frac{\Pr_0(\sum_{i=1}^n X_i>t-nc)}{\sum_{i=1}^n\Pr_0( X_i>t-c)}\\
% &=\lim_{t\to\infty}\frac{\Pr_0( X_1+\dots+ X_n>t-nc)}{\sum_{i=1}^n\Pr_0( X_i>t-nc)}\times\lim_{t\to\infty}\frac{\sum_{i=1}^n\Pr_0( X_i>t-nc)}{\sum_{i=1}^n\Pr_0( X_i>t-c)}\\
% &=\lim_{t\to\infty}\frac{\Pr_0( X_1+\dots+ X_n>t)}{\sum_{i=1}^n\Pr_0( X_i>t)}
% =\lim_{t\to\infty}\frac{\Pr_0(\frac{1}{n}\sum_{i=1}^n  X_i>t)}{\sum_{i=1}^n\Pr_0( X_i>nt)}\\
=\lim_{t\to\infty}\frac{\Pr_0(\frac{1}{n}\sum_{i=1}^n  X_i>t)}{n^{-\gamma}\sum_{i=1}^n\Pr_0( X_i>t)},
\end{split}
\end{equation*}
where the last equality is because each $ X_i\in\mathscr{R}_{-\gamma}$.
According to \Cref{prop:sum-tail-general-norm},
\begin{equation}
\label{eq:error-limit}
q(\gamma)
=n^{\gamma}\lim_{t\to\infty}\frac{\Pr_0(\frac{1}{n}\sum_{i=1}^n  X_i>t)}{\sum_{i=1}^n\Pr_0( X_i>t)}
=\frac{1}{n}\int_{\punitsphereone{n}}\left(\sum_{i=1}^n(\omega_i\theta_i)^{1/\gamma}\right)^\gamma H^*(\mathrm{d}\bm{\theta}),   
\end{equation}
where $H^*$ is the spectral measure of $C$ and the last equality plugs in $\sum_{i=1}^n{\omega}_i=n$.
Since all $\omega_i>0$, by \citesupp[Theorem 19]{hardy1952inequalities}, $\|(\bm{\omega}\odot\bm{\theta})^{1/\gamma}\|_1^\gamma$ is non-decreasing in $ \gamma >0$. Therefore,
$q(\gamma)$ is non-decreasing in $\gamma$. 
Based on the definition of the stable tail dependence function, for any $i\in[n]$,
\[
\int_{\punitsphereone{n}}\theta_iH^*(\intd\bm{\theta})=\ell(\mathbf{e}_i)=\lim_{s\downarrow0}s^{-1}D(s\mathbf{e}_i)=1
\]
Accordingly,
\begin{align*} 
\frac{1}{n}\int_{\punitsphereone{n}}\sum_{i=1}^n \omega_i \theta_i H^*(\intd\bm{\theta}) = \frac{\sum_{i=1}^n\omega_i}{n}=1.
\end{align*}
Therefore, when $\gamma \ge 1$,
\begin{align*} 
q(\gamma)
=\frac{1}{n}\int_{\punitsphereone{n}}\left(\sum_{i=1}^n(\omega_i\theta_i)^{1/\gamma}\right)^\gamma H^*(\mathrm{d}\bm{\theta})
\ge \frac{1}{n}\int_{\punitsphereone{n}}\sum_{i=1}^n \omega_i \theta_i H^*(\intd\bm{\theta})
= 1 ,
\end{align*}
and for $\gamma\le 1$,
\begin{align*} 
q(\gamma)
=\frac{1}{n}\int_{\punitsphereone{n}}\left(\sum_{i=1}^n(\omega_i\theta_i)^{1/\gamma}\right)^\gamma H^*(\mathrm{d}\bm{\theta})
\le \frac{1}{n}\int_{\punitsphereone{n}}\sum_{i=1}^n \omega_i \theta_i H^*(\intd\bm{\theta})
= 1.
\end{align*}
% Since all $\omega_i>0$, to get the strict inequalities, $H^*$ should have positive mass on $\{\bm{\theta}\in\punitsphereone{n}: \theta_i \theta_j>0  \mbox{~for some~} i\neq j\}$. That is, $H^*$ does not concentrate on $\mathbf{e}_i$s and the p-vector $\mathbf{1}-\mathbf{P}$ is  asymptotically independent.
Since all \(\omega_i>0\), strict inequality holds if and only if
\(H^*\) assigns positive mass to
$\{\bm{\theta}\in\punitsphereone{n}: \theta_i \theta_j>0  \mbox{~for some~} i\neq j\}$.
That is, $H^*$ does not concentrate on $\mathbf{e}_i$s, or equivalently, the p-vector $\mathbf{1}-\mathbf{P}$ is not asymptotically independent.
\end{proof}

\subsection{Proof of \Cref{cor:CCT-validity}}

\ValidCauchy*

\begin{proof}
We use $F_{\rm C}$ and $F_{\rm tC}$ to denote CDFs of standard Cauchy and truncated Cauchy distributions, as defined in \Cref{exmp:truncated_cauchy}. Notice that
\[
F_{\rm tC}(x) = \frac{F_{\rm C}(x)-q}{1-q}.
\]
Therefore, their survival functions and quantile functions satisfy
\begin{align*}
\overline F_{\rm tC}(x) &= \overline F_{\rm C}(x)/(1-q),\\
Q_{F_{\rm tC}}(1-p) &= Q_{F_{\rm C}}(1-(1-q) p),\quad \text{for~all}~p\in (0,1).
\end{align*}
Now, we note that
\[
Q_{F_{\rm C}}(1- p)< \theta Q_{F_{\rm C}}(1-\theta p),
\]
for all $p\in (0,1)$ and $\theta\in (0,1)$. This follows from $Q_{F_{\rm C}}(1-t)=\cot(\pi t)$ for $t\in (0,1)$ and that $t\mapsto t \cot(t)$ is strictly decreasing.

For simplicity, we use $P_{\rm tC}$ to denote p-values of the heavy-tailed combination test using truncated Cauchy. Writing $\theta =1-q$, we have
\begin{align*}
P_{\rm comb}^{\rm CCT}=  \overline F_{\rm C}\left(\frac1 n\sum_{i=1}^ n Q_{F_{\rm C}}(1-P_i) \right)&\ge 
 \overline F_{\rm C}\left(\frac1 n\sum_{i=1}^ n \theta Q_{F_{\rm C}}(1-\theta P_i) \right)
 \\&=
\overline F_{\rm C}\left(\frac1 n\sum_{i=1}^ n \theta Q_{F_{\rm tC}} (1- P_i  ) \right) \\
&= \overline {F}_{\rm C}(\theta {\overline{F}_{\rm tC}^{-1}}(P_{\rm tC}))
=\theta  \overline {F}_{\rm tC}(\theta {\overline{F}_{\rm tC}^{-1}}(P_{\rm tC})).
\end{align*}
The asymptotic validity of $P_{\rm comb}^{\rm CCT}$ follows immediately from that of $P_{\rm tC}$ in \Cref{thm:combination-test-valid}, and the fact that $ {F}_{\rm tC}$ is in $\mathscr{R}_{-1}^*$, which gives $\theta \overline F_{\rm tC}(\theta x)/\overline F_{\rm tC}(x)\to1$ as $x\to \infty$.
\end{proof}

\subsection{Proof of \Cref{thm:error-complete-dependence}}
\label{subsec:error-complete-dependence}
\ErrorCompleteDependence*
\begin{proof}
Following the proof of \Cref{thm:combination-test-valid}, the asymptotic type-I error ratio is
\[
q(\gamma)
=\frac{1}{n}\int_{\punitsphereone{n}}\left(\sum_{i=1}^n(\omega_i\theta_i)^{1/\gamma}\right)^\gamma H^*(\mathrm{d}\bm{\theta}).
\]
Since when $0<\gamma<1$, $\psi(\bm\theta)=\left(\sum_{i=1}^n(\omega_i\theta_i)^{1/\gamma}\right)^\gamma$ is a convex function and $H^*/n$ is a probability measure on $\punitsphereone{n}$, by Jensen's inequality,
\begin{align*}
&q(\gamma)
=\int_{\punitsphereone{n}}\left(\sum_{i=1}^n(\omega_i\theta_i)^{1/\gamma}\right)^\gamma \frac{H^*(\mathrm{d}\bm{\theta})}{n}\\
\ge& \psi\left(\int_{\punitsphereone{n}}\theta_1 \frac{H^*(\mathrm{d}\bm{\theta})}{n},\dots,\int_{\punitsphereone{n}}\theta_n \frac{H^*(\mathrm{d}\bm{\theta})}{n}\right)\\
=&\psi\left(\frac{1}{n},\dots,\frac{1}{n}\right)
=\left(\sum_{i=1}^n(\omega_i/n)^{1/\gamma}\right)^\gamma
=\frac{1}{n}\left(\sum_{i=1}^n \omega_i^{1/\gamma}\right)^\gamma.
\end{align*}
The equality holds if and only if $H^*/n$ is a delta measure at $(\frac{1}{n},\dots,\frac{1}{n})$. That is, $H^*$ is the spectral measure of an asymptotic completely dependent MRV copula. Moreover, since $\psi(\bm\theta)$ is a convex function, for any $\bm\theta\in\punitsphereone{n}$,
\[
\psi(\bm\theta)
\le\sum_{i=1}^n\theta_i\psi(\mathbf{e}_i)=\sum_{i=1}^n\theta_i\omega_i.
\]
Accordingly,
\begin{align*}
&q(\gamma)
=\int_{\punitsphereone{n}}\left(\sum_{i=1}^n(\omega_i\theta_i)^{1/\gamma}\right)^\gamma \frac{H^*(\mathrm{d}\bm{\theta})}{n}
\le \int_{\punitsphereone{n}}\sum_{i=1}^n\theta_i\omega_i \frac{H^*(\mathrm{d}\bm{\theta})}{n}\\
&=\frac{1}{n}\sum_{i=1}^n\omega_i\int_{\punitsphereone{n}}\theta_i H^*(\mathrm{d}\bm{\theta})
=\frac{1}{n}\sum_{i=1}^n\omega_i
=1.
\end{align*}
The equality holds if and only if $H^*$ only has point mass on those points satisfying $\psi(\bm\theta)=\sum_{i=1}^n\theta_i\omega_i$, i.e., $\bm\theta=\mathbf{e}_i$. In other words, $H^*$ is the spectral measure of those asymptotic independent MRV copulas.

When $\gamma>1$, $-\psi(\bm\theta)$ is a convex function. By similar deduction,
\[
1\le q(\gamma)\le\frac{1}{n}\left(\sum_{i=1}^n \omega_i^{1/\gamma}\right)^\gamma,
\]
and the upper bound and lower bound are attained if and only if the MRV copula characterizes the asymptotic complete dependence and asymptotic independence, respectively.
\end{proof}

\subsection{Proof of \Cref{thm:power-increasing}}
\label{subsec:proof-power-increasing}
\begin{lem}
\label{lemma:assmp-1-iii-ensure-rv}
Suppose a random variable $P_0$ with continuous distribution satisfies that there exists a continuous $F\in\mathscr{R}_{-1}$ such that $Q_F(1-P_0)\in\mathscr{R}_{-\beta}$ where $\beta\le1$. Then, the cumulative distribution function of $P_0$ has the form $t^\beta h(t)$, where $h$ is a slowly varying function at $0$: for all $y>0$
\begin{align} \label{eq:slowvary}
  \lim_{t\downarrow0}\frac{h(ty)}{h(t)}=1.  
\end{align}
Accordingly, for any $\tilde{F}\in\mathscr{R}_{-\gamma}$, any positive weight $\omega>0$, and any random variable $P$ with continuous distribution such that
\[
\lim_{t\downarrow0}\frac{\Pr(P\le t)}{\Pr(P_0\le t)}\in(0,\infty),
\]
it holds that $Q_{\tilde{F}}((1-P/\omega)^+)\in\mathscr{R}_{-\beta\gamma}$.
\end{lem}
\begin{proof}
Let $D_0$ denote the cumulative distribution function of $P_0$. The survival function of $Q_F(1-P_0)$ is
\begin{align*}
\Pr(Q_F(1-P_0)>x)= \Pr(P_0<\overline{F}(x))=\Pr(P_0\le\overline{F}(x))=D_0\left(\overline{F}(x)\right), 
\end{align*}
where the second to last equality is because $P_0$ has a continuous distribution.
Since $Q_F(1-P_0)\in\mathscr{R}_{-\beta}$, for all $y>0$
\begin{align*}
&y^{-\beta}
=\lim_{x\to\infty}\frac{\Pr\left(Q_F(1-P_0)>xy\right)}{\Pr\left(Q_F(1-P_0)>x\right)}
=\lim_{x\to\infty}\frac{D_0\left(\overline{F}(xy)\right)}{D_0\left(\overline{F}(x)\right)}\\
=&\lim_{x\to\infty}\frac{D_0\left(y^{-1}\overline{F}(x)\right)}{D_0\left(\overline{F}(x)\right)}
=\lim_{t\downarrow0}\frac{D_0\left(y^{-1}t\right)}{D_0\left(t\right)},  
\end{align*}
where the third equality is due to continuity of $D_0$ and $F\in \mathscr{R}_{-1}$ and the last equality is due to the continuity of $F$. In other words, for any $y>0$,
\begin{equation}
\label{eq:construct-h}
\lim_{t\downarrow0}\frac{D_0\left(yt\right)}{D_0\left(t\right)}=y^{\beta}.
%~\Leftrightarrow~\lim_{t\downarrow0}\frac{D_0\left(yt\right)/(yt)^\beta}{D_0\left(t\right)/t^\beta}=1
\end{equation}
Let $h(t)=D_0(t)/t^\beta$. \eqref{eq:construct-h} is equivalent to  \eqref{eq:slowvary}.
% for all $y>0$,
% \[
% \lim_{t\downarrow0}\frac{h\left(yt\right)}{h\left(t\right)}=1.
% \]
Let $D$ represent the cumulative distribution function of $P$. Then, the tail probability of $Q_{\tilde F}\left((1-P/\omega)^+\right)$ is
\begin{align*}
&\Pr(Q_{\tilde F}\left((1-P/\omega)^+\right)>x)
= \Pr(P < \omega\{1-\tilde{F}(x)\})\\
=&\Pr(P \le \omega\{1-\tilde{F}(x)\})
=D\left(\omega\{1-\tilde F(x)\}\right).
\end{align*}
For all $y>0$,
\begin{align*}
\lim_{x\to\infty}\frac{D\left(\omega\{1-\tilde F(xy)\}\right)}{D\left(\omega\{1-\tilde F(x)\}\right)}
=&\lim_{x\to\infty}\frac{D\left(\omega\{1-\tilde F(xy)\}\right)/D_0\left(\omega\{1-\tilde F(xy)\}\right)}{D\left(\omega\{1-\tilde F(x)\}\right)/D_0\left(\omega\{1-\tilde F(x)\}\right)}\\
&\quad\times\lim_{x\to\infty}\frac{D_0\left(\omega\{1-\tilde F(xy)\}\right)}{D_0\left(\omega\{1-\tilde F(x)\}\right)}\\
=&\lim_{x\to\infty}\frac{D_0\left(\omega y^{-\gamma}\{1-\tilde F(x)\}\right)}{D_0\left(\omega\{1-\tilde F(x)\}\right)}
=\lim_{t\downarrow0}\frac{D_0\left(y^{-\gamma}t\right)}{D_0\left(t\right)}\\
=&\lim_{t\downarrow0}\frac{y^{-\beta\gamma} t^{\beta} h\left(y^{-\gamma}t\right)}{t^\beta h\left(t\right)}
=y^{-\beta\gamma},
\end{align*}
where the second equality is because $D_0$ is continuous. By the definition of a regularly varying tailed distribution, $Q_{\tilde F}\left((1-P/\omega)^+\right)\in\mathscr{R}_{-\beta\gamma}$. This completes the proof.
\end{proof}

% Theorem 4.4
\PowerIncrease*
\begin{proof}
By \Cref{lemma:mrv-robust-to-shift}, without loss of generality, the lower bound of $F$ can be assumed as 0.
% Since the support of $F$ is $[c,\infty)$, $\tilde{F}(x)=F(x-c)$ is supported on $[0,\infty)$ and its quantile function is $Q_{\tilde{F}}(\cdot)=Q_F(\cdot)-c$. Let $\tilde{X}_i=Q_F\left((1-P_i/\omega_i)^+\right)-c$ and the random vector
% \begin{equation*}
% \tilde{\mathbf{X}} 
% = \left( Q_{\tilde F}\left((1-P_1/\omega_1)^+\right),\dots, Q_{\tilde F}((1-P_n/\omega_n)^+)\right)
% \end{equation*}
According to \Cref{prop:copula-to-mrv-distribution}, a sufficient condition to ensure that ${\mathbf{X}}$ has an MRV distribution is that $\mathbf{P}=(P_1,\dots,P_n)$ has a survival MRV copula and the dominant marginal has a regularly varying tailed distribution. \Cref{asmp:p-alternative} (ii) directly provides the former one. 
Without loss of generality, we suppose $1\in I_{\mathbf{P}}$ satisfying \Cref{asmp:p-alternative} (iii) and let
\[
b_i=\lim_{t\downarrow0}\frac{\Pr(P_i<t)}{\Pr(P_1<t)}\in(0,\infty).
\]
According to \Cref{lemma:assmp-1-iii-ensure-rv}, the cumulative distribution function of $P_1$ is $t^\beta h(t)$ where $h(t)$ is slowly varying at $0$.
Then for any $i\in I_{\mathbf{P}}$,
\begin{equation}
\label{eq:define-ci-Ip-alternative}
\begin{split}
&c_i:=\lim_{x\to\infty}\frac{\Pr(Q_F\left((1-P_i/\omega_i)^+\right)>x)}{\Pr(Q_F\left((1-P_1/\omega_1)^+\right)>x)}
=\lim_{t\downarrow0}\frac{\Pr(P_i<\omega_i t)}{\Pr(P_1<\omega_1 t)}\\
&=\left(\frac{\omega_i}{\omega_1}\right)^\beta\times\lim_{t\downarrow0}\frac{\Pr(P_i<\omega_i t)}{\omega_i^\beta t^\beta h(\omega_i t)}\times\lim_{t\downarrow0}\frac{h(\omega_i t)}{h(\omega_1 t)}
\\
&=\left(\frac{\omega_i}{\omega_1}\right)^\beta\times\lim_{t\downarrow0}\frac{\Pr(P_i<\omega_i t)}{\Pr(P_1<\omega_i t)}
=b_i\left(\frac{\omega_i}{\omega_1}\right)^\beta.    
\end{split}
\end{equation}
For any $i\notin I_{\mathbf{P}}$,
\begin{equation}
\label{eq:define-ci-not-Ip-alternative}
\begin{split}
&c_i:=\lim_{x\to\infty}\frac{\Pr(Q_F\left((1-P_i/\omega_i)^+\right)>x)}{\Pr(Q_F\left((1-P_1/\omega_1)^+\right)>x)}
=\lim_{t\downarrow0}\frac{\Pr(P_i<\omega_i t)}{\Pr(P_1<\omega_1 t)}\\
&=\left(\frac{\omega_i}{\omega_1}\right)^\beta\times\lim_{t\downarrow0}\frac{\Pr(P_i<\omega_i t)}{\omega_i^\beta t^\beta h(\omega_i t)}\times\lim_{t\downarrow0}\frac{h(\omega_i t)}{h(\omega_1 t)}
\\
&=\left(\frac{\omega_i}{\omega_1}\right)^\beta\times\lim_{t\downarrow0}\frac{\Pr(P_i<\omega_i t)}{\Pr(P_1<\omega_i t)}
=0,
\end{split}
\end{equation}
where the last equality is based on the definition of $I_{\mathbf{P}}$.
According to \Cref{lemma:assmp-1-iii-ensure-rv}, it also holds that $Q_{F}((1-P_1/\omega_1)^+)\in\mathscr{R}_{-\gamma\beta}$. That is, the dominant tail has a regularly varying tailed distribution and hence ${\mathbf{X}}\in\mathrm{MRV}_{-\gamma\beta}$ on $\Xi$ with $c_i$s defined in \eqref{eq:define-ci-Ip-alternative} and \eqref{eq:define-ci-not-Ip-alternative}.

% Let $X_i = Q_{F}((1-P_i/\omega_i)^+)$
% % and $\tilde{X}_i=Q_{\tilde F}((1-P_i/\omega_i)^+)$. 
% We use $I_{{\mathbf{X}}}$ to denote the set $\{i\in[n]:c_i\neq0\}$. By definition, $I_{\tilde{\mathbf{X}}}=I_{\mathbf{P}}$. For each $i\in I_{\tilde{\mathbf{X}}}$, $\tilde{X}_i$ has a regularly varying tailed distribution and thus,
% \begin{equation}
% \label{eq:change-x-to-tilde-x-sum}
% \begin{split}
% &\lim_{x\to\infty}\frac{\Pr\left(\sum_{i=1}^nX_i>x\right)}{\sum_{i=1}^n\Pr(X_i>x)}
% =\lim_{x\to\infty}\frac{\Pr(\sum_{i=1}^n\tilde X_i>x-nc)}{\sum_{i=1}^n\Pr(\tilde X_i>x-c)}\\
% &=\lim_{x\to\infty}\frac{\Pr(\sum_{i=1}^n\tilde X_i>x-nc)}{\sum_{i\in I_{\tilde{\mathbf{X}}}}\Pr(\tilde X_i>x-c)}
% =\lim_{x\to\infty}\frac{\Pr\left(\sum_{i=1}^n\tilde X_i>x\right)}{\sum_{i\in I_{\tilde{\mathbf{X}}}}\Pr(\tilde X_i>x)}\\
% &=\lim_{x\to\infty}\frac{\Pr\left(\frac{1}{n}\sum_{i=1}^n\tilde X_i>x\right)}{n^{-\gamma\beta}\sum_{i\in I_{\tilde{\mathbf{X}}}}\Pr(\tilde X_i>x)}
% =\lim_{x\to\infty}\frac{\Pr(\frac{1}{n}\sum_{i=1}^n\tilde X_i>x)}{n^{-\gamma\beta}\sum_{i=1}^n\Pr(\tilde X_i>x)}.   
% \end{split}
% \end{equation}
Let $X_i = Q_{F}((1-P_i/\omega_i)^+)$ and $\bar{X}=\sum_{i=1}^n X_i/n$. Then the power of the combination test can be rewritten as:
\begin{equation}
\label{eq:rewrite-power}
\begin{split}
&\Pr\left(P^{F,\bm{\omega}}_{\mathrm{comb}}\le\alpha\right)
=\Pr\left(\bar{X}>Q_F\left(1-{\alpha}/{n^{1-\gamma}}\right)\right)\\
&=\frac{\Pr\left(\bar{X}>Q_F\left(1-\frac{\alpha}{n^{1-\gamma}}\right)\right)}{\sum_{i=1}^n\Pr\left( X_i> n Q_F\left(1-\frac{\alpha}{n^{1-\gamma}}\right)\right)}\times\sum_{i=1}^n\Pr\left( X_i> n Q_F\left(1-\frac{\alpha}{n^{1-\gamma}}\right)\right)\\
&=\frac{\Pr\left(\bar{X}> Q_F\left(1-\frac{\alpha}{n^{1-\gamma}}\right)\right)}{\sum_{i=1}^n\Pr\left( X_i> n Q_F\left(1-\frac{\alpha}{n^{1-\gamma}}\right)\right)}\times\sum_{i=1}^n\Pr\left( P_i<\omega_i\overline{F}\left(nQ_F\left(1-\frac{\alpha}{n^{1-\gamma}}\right)\right)\right).
\end{split}
\end{equation}
Since every $P_i$ has a continuous distribution and 
\[
\lim_{\alpha\downarrow0}\frac{\overline{F}\left(nQ_F\left(1-\frac{\alpha}{n^{1-\gamma}}\right)\right)}{\alpha/n}
=\lim_{\alpha\downarrow0}\frac{n^{-\gamma}\overline{F}\left(Q_F\left(1-\frac{\alpha}{n^{1-\gamma}}\right)\right)}{\alpha/n}
=\lim_{\alpha\downarrow0}\frac{n^{-\gamma}\frac{\alpha}{n^{1-\gamma}}}{\alpha/n}
=1,
\]
we get
\begin{equation}
\label{eq:power-summation-individuals}
\lim_{\alpha\downarrow0}\frac{\sum_{i=1}^n\Pr\left( P_i<\omega_i\overline{F}\left(\frac{1}{a_n}Q_{F}\left(1-\frac{\alpha}{n a_n^\gamma}\right)\right)\right)}{\sum_{i=1}^n\Pr\left( P_i<\frac{\omega_i\alpha}{n}\right)}=1.
\end{equation}
Then the limiting scaled power is
\begin{align*}
&\tilde{q}(\gamma)
=\lim_{\alpha\downarrow0}\frac{\Pr\left(P^{F,\bm{\omega}}_{\mathrm{comb}}\le\alpha\right)}{\sum_{i=1}^n\Pr\left( P_i<\frac{\omega_i\alpha}{n}\right)}
=\lim_{\alpha\downarrow0}\frac{\Pr\left(\frac{1}{n}\sum_{i=1}^n X_i>Q_F\left(1-\frac{\alpha}{n^{1-\gamma}}\right)\right)}{\sum_{i=1}^n\Pr\left( X_i> n Q_F\left(1-\frac{\alpha}{n^{1-\gamma}}\right)\right)}\\
&=\lim_{x\to\infty}\frac{\Pr(\sum_{i=1}^n X_i>x)}{\sum_{i=1}^n\Pr(X_i>x)}=\lim_{x\to\infty}\frac{\Pr(\bar{X}>x)}{n^{-\gamma\beta}\sum_{i=1}^n\Pr(X_i>x)}\\
&=\int_{\punitsphereone{n}}\left(\sum_{i=1}^n(c_i\theta_i)^{1/(\gamma\beta)} \right)^{\gamma\beta}H^*(\intd\bm{\theta}),
\end{align*}
where the second-to-last equation is by \Cref{lemma:mrv-robust-to-shift}. Since  $\bigl(\sum_{i=1}^n (c_i\theta_i)^{1/(\gamma\beta)}\bigr)^{\gamma\beta}$ is non-decreasing with $\gamma$ (See \citesupp{hardy1952inequalities}) as $\beta>0$, $\tilde{q}(\gamma)$ is non-decreasing in $\gamma$.
Moreover, $\tilde{q}(\gamma)$ is a constant if and only if $H^*$ is only distributed on axes or when $|I_{\mathbf{P}}|=1$. In other words, $\tilde{q}(\gamma)$ is increasing if and only if $|I_{\mathbf{P}}|=S>1$ and $(1-P_{i_1},\dots,1-P_{i_S})$ is asymptotically dependent where $I_{\mathbf{P}}=\{i_1,\dots,i_S\}$. 
\end{proof}

\subsection{Proof of \Cref{thm:bonferroni-more-conservative}}
\label{subsec:proof-bonferroni-conservative}
\BonferroniConservative*
\begin{proof}
We follow the proof of \Cref{thm:power-increasing} assuming the support of $F$ is lower bounded by $c=0$ and get that the random vector
\begin{equation*}
{\mathbf{X}} 
= \left( Q_{F}\left((1-P_1/\omega_1)^+\right),\dots, Q_{F}((1-P_n/\omega_n)^+)\right)
\end{equation*}
belongs to $\mathrm{MRV}_{-\beta\gamma}$ for some $\beta\le1$ on $\Xi$ with its $c_i$s shown in \eqref{eq:define-ci-Ip-alternative} and \eqref{eq:define-ci-not-Ip-alternative}. In addition, $P_1$'s CDF is $t^\beta h(t)$, where $h$ is a slowly varying function at 0. %\RW{What is $h$?}
The limit in \Cref{thm:bonferroni-more-conservative} can be rewritten as:
\begin{equation}
\label{eq:change-test-to-x-tilde}
\begin{split}
&\lim_{\alpha\downarrow0}\frac{\Pr\left(P_\mathrm{comb}^{F, \bm \omega}\le\alpha\right)}{\Pr\left(P_\mathrm{bon}^{\bm{\omega}}\le\alpha\right)}
=\lim_{\alpha\downarrow0}\frac{\Pr\left(\frac{1}{n}\sum_{i=1}^n X_i > Q_F\left(1-\alpha/n^{1-\gamma}\right)\right)}{\Pr(\max_{i\in[n]}X_i>Q_F(1-\alpha/n))}\\
&=\lim_{\alpha\downarrow0}\frac{\Pr\left(\frac{1}{n}\sum_{i=1}^n X_i > Q_F\left(1-\alpha/n^{1-\gamma}\right)\right)}{\sum_{i=1}^n\Pr(X_i>Q_F(1-\alpha/n^{1-\gamma}))}\bigg/\lim_{\alpha\downarrow0}\frac{\Pr(\max_{i\in[n]}X_i> Q_F(1-{\alpha}/{n}))}{\Pr\left(\sum_{i=1}^n X_i > Q_F(1-{\alpha}/{n})\right)}\\
&\qquad\times \lim_{\alpha\downarrow0}\frac{\sum_{i=1}^n\Pr(X_i> Q_F(1-{\alpha}/{n^{1-\gamma}}))}{\sum_{i=1}^n\Pr(X_i>Q_F(1-{\alpha}/{n}))}\\
&=n^{\gamma\beta}\lim_{x\to\infty}\frac{\Pr(\frac{1}{n}\sum_{i=1}^n X_i>x)}{\sum_{i=1}^n\Pr(X_i>x)}\bigg/\lim_{x\to\infty}\frac{\Pr(\max_{i\in[n]}X_i>x)}{\sum_{i=1}^n\Pr(X_i>x)},
% &=n^{\gamma\beta}\lim_{x\to\infty}\frac{\Pr(\frac{1}{n}\sum_{i=1}^n \tilde{X}_i>x)}{\sum_{i=1}^n\Pr(\tilde{X}_i>x)}\bigg/\lim_{x\to\infty}\frac{\Pr(\max_{i\in[n]}\tilde{X}_i>x)}{\sum_{i=1}^n\Pr(\tilde{X}_i>x)},
\end{split}
\end{equation}
where the third equality plugs in the following equation
\begin{align*}
&\lim_{\alpha\downarrow0}\frac{\sum_{i=1}^n\Pr(X_i> Q_F(1-{\alpha}/{n^{1-\gamma}}))}{\sum_{i=1}^n\Pr(X_i>Q_F(1-{\alpha}/{n}))}
=\lim_{\alpha\downarrow0}\frac{\sum_{i=1}^n\Pr\left(P_i<\omega_i {\alpha}/{n^{1-\gamma}}\right)}{\sum_{i=1}^n\Pr\left(P_i<\omega_i{\alpha}/{n}\right)}\\
&=\frac{\sum_{i=1}^n\lim_{\alpha\downarrow0}\Pr\left(P_i<\omega_i {\alpha}/{n^{1-\gamma}}\right)/\Pr\left(P_1<\omega_1 {\alpha}/{n^{1-\gamma}}\right)}{\sum_{i=1}^n\lim_{\alpha\downarrow0}\Pr\left(P_i<\omega_i{\alpha}/{n}\right)/\Pr\left(P_1<\omega_1{\alpha}/{n}\right)}\times\lim_{\alpha\downarrow0}\frac{\Pr\left(P_1<\omega_1 {\alpha}/{n^{1-\gamma}}\right)}{\Pr\left(P_1<\omega_1{\alpha}/{n}\right)}\\
&=\frac{\sum_{i=1}^n c_i}{\sum_{i=1}^n c_i}\times\lim_{\alpha\downarrow0}\frac{(\omega_1\alpha/n^{1-\gamma})^\beta h(\omega_1\alpha/n^{1-\gamma})}{\left(\omega_1{\alpha}/{n}\right)^\beta h(\omega_1\alpha/n)}
=n^{\gamma\beta}.
\end{align*}
Following the proof of \Cref{prop:sum-tail-general-norm} (see \Cref{subsec:proof-sum-tail-general-norm}), the denominator can be rewritten as
\begin{equation*}
\label{eq:max-prob-tail}
\begin{split}
&\lim_{x\to\infty}\frac{\Pr(\max_{i\in[n]}{X}_i>x)}{\sum_{i=1}^n\Pr({X}_i>x)}
=1\bigg/\sum_{i=1}^n\lim_{x\to\infty}\frac{\Pr({X}_i>x)}{\Pr({\mathbf{X}}\in x[\mathbf{0},\mathbf{1}]^c)}\\
&=\ell(c_1,\dots,c_n)/\sum_{i=1}^n c_i
=\frac{1}{\sum_{i=1}^n c_i}\int_{\punitsphereone{n}} \max_{i\in[n]}(c_i\theta_i)H^*(\intd\bm{\theta}),
\end{split}
\end{equation*}
where $H^*$ is the spectral measure of ${\mathbf{X}}$ (or copula of $\mathbf{1}-\mathbf{P}$). Since $c_i\theta_i\ge0$,
\begin{equation}
\label{eq:inf_norm_expansion}
\begin{split}
&\lim_{x\to\infty}\frac{\Pr(\max_{i\in[n]}{X}_i>x)}{\sum_{i=1}^n\Pr({X}_i>x)}
=\frac{1}{\sum_{i=1}^nc_i}\int_{\punitsphereone{n}}\max_{i\in[n]}(c_i\theta_i)H^*(\intd\bm{\theta})\\
&\le\frac{1}{\sum_{i=1}^nc_i}\int_{\punitsphereone{n}}\left(\sum_{i=1}^n(c_i\theta_i)^{1/(\beta\gamma)} \right)^{\beta\gamma} H^*(\intd\bm{\theta})
=n^{\gamma\beta}\lim_{x\to\infty}\frac{\Pr(\frac{1}{n}\sum_{i=1}^n {X}_i>x)}{\sum_{i=1}^n\Pr({X}_i>x)}.
\end{split}
\end{equation}
Since $c_i>0$ for all $i\in I_{\mathbf{P}}$, the inequality in \eqref{eq:inf_norm_expansion} is strict when (1) $|I_{\mathbf{P}}|\ge2$, and (2) $H^*$ has positive mass on $\{\bm{\theta}\in \punitsphereone{n}: \theta_i \theta_j>0  \mbox{~for some~} i\neq j~\mathrm{and}~i,j\in I_{\mathbf{P}}\}$, i.e., $(1-P_{i_1},\dots,1-P_{i_S})$ is not asymptotically independent where $I_{\mathbf{P}}=\{i_1,\dots,i_S\}$. 

Plugging in \eqref{eq:change-test-to-x-tilde}-\eqref{eq:inf_norm_expansion}, the theorem follows.
\end{proof}

\subsection{Proof of \Cref{thm:bonferroni-gamma-zero}}
\label{subsec:proof-bonferroni-gamma-zero}
\BonferroniGammaZero*
\begin{proof}
Let $F\in\mathscr{R}_{-\gamma}$. By \eqref{eq:inf_norm_expansion} of the proof of \Cref{thm:bonferroni-more-conservative}, it holds that
\begin{equation}
\label{eq:thm4.5proof1}
\begin{split}
&\lim_{\gamma\downarrow0}\lim_{\alpha\downarrow0} \frac{\Pr\left(P_\mathrm{comb}^{F, \bm \omega}\leq \alpha\right)}{\Pr\left(P_{\mathrm{bon}}^{\bm{\omega}}\leq \alpha\right)}\\
&=\lim_{\gamma\downarrow0}n^{\gamma\beta}\lim_{x\to\infty}\frac{\Pr(\frac{1}{n}\sum_{i=1}^n {X}_i>x)}{\sum_{i=1}^n\Pr({X}_i>x)}\bigg/\lim_{x\to\infty}\frac{\Pr(\max_{i\in[n]}{X}_i>x)}{\sum_{i=1}^n\Pr({X}_i>x)}\\
&=\frac{\lim_{\gamma\downarrow0}\int_{\punitsphereone{n}}\left(\sum_{i=1}^n(c_i\theta_i)^{\frac{1}{\beta\gamma}} \right)^{\beta\gamma} H^*(\intd\bm{\theta})}{\int_{\punitsphereone{n}} \max_{i\in[n]}(c_i\theta_i)H^*(\intd\bm{\theta})},   
\end{split}
\end{equation}
where ${\mathbf{X}}=\left({X}_1,\dots,{X}_n\right)\in\mathrm{MRV}_{-\beta\gamma}$ is on $\Xi$ with $\beta\le1$, and its $c_i$s are defined in \eqref{eq:define-ci-Ip-alternative} and \eqref{eq:define-ci-not-Ip-alternative} and $H^*$ is its spectral measure.
Since $\left(\sum_{i=1}^n(c_i\theta_i)^{\frac{1}{\beta\gamma}} \right)^{\beta\gamma}$ is a non-increasing function of $\gamma$ (see Theorem 19 of \citesupp{hardy1952inequalities}) and  
\[
\lim_{\gamma\downarrow0}\left(\sum_{i=1}^n(c_i\theta_i)^{\frac{1}{\beta\gamma}} \right)^{\beta\gamma}
=\max_{i\in[n]}\left(c_i\theta_i\right),
\]
by the monotone convergence theorem, it holds that
\begin{align*}
&\lim_{\gamma\downarrow0}\int_{\punitsphereone{n}}\left(\sum_{i=1}^n(c_i\theta_i)^{\frac{1}{\beta\gamma}} \right)^{\beta\gamma} H^*(\intd\bm{\theta})\\
&=\int_{\punitsphereone{n}}\lim_{\gamma\downarrow0}\left(\sum_{i=1}^n(c_i\theta_i)^{\frac{1}{\beta\gamma}} \right)^{\beta\gamma} H^*(\intd\bm{\theta})\\
&=\int_{\punitsphereone{n}}\max_{i\in[n]}(c_i\theta_i) H^*(\intd\bm{\theta}).
\end{align*}
Substituting this into \eqref{eq:thm4.5proof1} yields \eqref{eq:asymp-bonferroni-gamma-zero} 
% Hence, it holds that
% \begin{equation*}
% \eqref{eq:asymp-bonferroni-gamma-zero}
% % =\frac{\int_{\punitsphereone{n}} \max_{i\in[n]}(c_i\theta_i)H^*(\intd\bm{\theta})}{\lim_{\gamma\downarrow0}\int_{\punitsphereone{n}}\left(\sum_{i=1}^n(c_i\theta_i)^{\frac{1}{\beta\gamma}} \right)^{\beta\gamma} H^*(\intd\bm{\theta})}
% =1,
% \end{equation*}
which finishes the proof.
\end{proof}

\subsection{Proof of \Cref{thm:bonferroni-more-conservative-with-dependence}}
\label{subsec:proof-bonferroni-dependence}
\BonferroniDependence*
\begin{proof}
Since $F\in\mathscr{R}_{-1}$, by \Cref{thm:combination-test-valid},
\[
\lim_{\alpha\downarrow0}\frac{\Pr_0\left(P_\mathrm{comb}^{F, \bm \omega}\leq \alpha\right)}{\alpha}=1.
\]
Then, the asymptotic type-I error ratio between the weighted combination test and the weighted Bonferroni test is
\begin{align*}
&r(C^*)
=\lim_{\alpha\downarrow0}\frac{\Pr_0\left(P^{F,\bm{\omega}}_\mathrm{comb}\leq \alpha\right)}{\Pr_0\left(P^{\bm{\omega}/n}_{\mathrm{bon}}\leq \alpha\right)}
=\lim_{\alpha\downarrow0}\frac{\alpha}{\Pr_0(n\min_{i\in[n]} P_i/\omega_i\le\alpha)}\\
&=\lim_{\alpha\downarrow0}\frac{\alpha}{1-\Pr_0(1-P_1<1-\omega_1\alpha/n,\dots,1-P_n<1-\omega_n\alpha/n)}\\
%&=n\lim_{\alpha\downarrow0}\frac{\alpha/n}{1-C(1-\omega_1\alpha/n,\dots,1-\omega_n\alpha/n)}
&=n/\ell(\omega_1,\dots,\omega_n) = -n/\log C^*(e^{-\omega_1},\ldots,e^{-\omega_n}) ,
\end{align*}
where $\ell$ is the stable tail dependence function of the copula $C^*$ and the last equality is due to \eqref{eq:c-star-and-l} in the main text. 
% Since $C_1\preceq_{\mathrm{pw}}C_2$ indicates that their limiting copulas hold
% $C^*_1\preceq_{\mathrm{pw}}C^*_2$. Accordingly, by \eqref{eq:c-star-and-l},
  Since $C^*_1\preceq_{\mathrm{pw}}C^*_2$, % by ,
% \[
% \ell_1(\omega_1,\dots,\omega_n)\ge \ell_2(\omega_1,\dots,\omega_n),
% \]
% and hence
we have $ 
% r(C_1)\le r(C_2).
r(C_1^*)\le r(C_2^*).
$ This completes the proof.
\end{proof}

\subsection{Proof of \Cref{thm:bonferroni-more-conservative-with-dependence-power}}
\label{subsec:proof-power-gain-dependence}
\PowerGainDependence*
\begin{proof}
Without loss of generality, we suppose $\mu_1\ge\dots\ge\mu_n\ge0$.
Following \Cref{prop:alternative-a} and the proof of \Cref{thm:bonferroni-more-conservative} in \Cref{subsec:proof-bonferroni-conservative}, $\beta=1$ and hence
\begin{equation}
\label{eq:power-gain-dependence-type-a-final}
\begin{split}
&\tilde{r}(C^*)
=\lim_{\alpha\downarrow0}\frac{\Pr\left(P^{F,\bm{\omega}}_\mathrm{comb}\leq \alpha\right)}{\Pr\left(P^{\bm{\omega}/n}_{\mathrm{bon}}\leq \alpha\right)}
=\frac{\int_{\punitsphereone{n}}\sum_{i=1}^nc_i\theta_i H^*(\intd\bm{\theta})}{\int_{\punitsphereone{n}} \max_{i\in[n]}(c_i\theta_i)H^*(\intd\bm{\theta})}\\
&=\frac{\sum_{i=1}^nc_i}{\int_{\punitsphereone{n}} \max_{i\in[n]}(c_i\theta_i)H^*(\intd\bm{\theta})}
=\frac{\sum_{i=1}^nc_i}{\ell(c_1,\dots,c_n)}
=\frac{\sum_{i=1}^nc_i}{-\log C^*(e^{-c_1},\dots,e^{-c_n})},    
\end{split}
\end{equation}
where $H^*$ and $\ell$ are the spectral measure and the stable tail dependence function of $C$, and
\[
c_i:=\lim_{t\downarrow0}\frac{1-G_i\left(G_i^{-1}(1-\omega_i t)-\mu_i\right)}{1-G_1\left(G_1^{-1}(1-\omega_1 t)-\mu_1\right)}\in[0,\infty).
\]
When $C_1^*\preceq_{\mathrm{pw}}C_2^*$, by \eqref{eq:power-gain-dependence-type-a-final}, $\tilde{r}(C_1^*)\le \tilde{r}(C_2^*)$. This completes the proof.
\end{proof}

\subsection{Proof of supplementary theoretical results}
\label{subsec:other-theory-results}

\Intermediate*
\begin{proof}
We first prove the general conclusion \eqref{eq:measure-equation}. Then, we plug in specific sets to get \eqref{eq:cube} and \eqref{eq:ball}.

\noindent\textbf{General conclusion.}
Since ${\bf X}\in {\rm MRV}_{-\gamma}$ with $\gamma>0$, \eqref{eq:260205-1} holds with $\lambda({\bf x})=\nu\!\left([0,\mathbf{x}]^{c}\right)$. 
Define $F_0(t):=1-\Pr\left(\mathbf{X}\in t[\mathbf{0},\mathbf{1}]^c\right)$. By the homogeneity of $\lambda(\cdot)$ (or Radon measure $\nu(\cdot)$), we have
\[
\lim_{t\to+\infty} \frac{\overline{F}_0(ta)}{\overline{F}_0(t)}=a^{-\gamma},~a>0.
\]
That is $F_0\in\mathscr{R}_{-\gamma}$. Then, for any $x>0$,
\begin{equation}
\label{eq:260206-1}
\lim_{t\to+\infty}\frac{\Pr(X_i>tx)}{\overline{F}_0(t)}
=\lim_{t\to+\infty}\frac{\Pr(X_i>tx)}{\overline{F}_0(tx)}
\times\lim_{t\to+\infty}\frac{\overline{F}_0(tx)}{\overline{F}_0(t)}
=c_ix^{-\gamma}.   
\end{equation}
where $c_i = \nu(B_i)$ with $B_i=\{\mathbf{x}\in\R_+^n:x_i>1\}$, $i\in [n]$.
Note that the convergence in \eqref{eq:260206-1} is uniform for $x\ge x_0>0$. 
Denote $C$ the copula of $\mathbf{X}$. It follows that for any $x_1,\dots,x_n$ where $\min_i x_i > 0$:
\begin{align}
\nu\!\left([0,\mathbf{x}]^{c}\right)
&=  \lim_{t\to\infty}\frac{1-C(F_1(tx_1),\dots,F_n(tx_n))}{\overline{F}_0(t)}\notag\\
&= \lim_{t\to\infty}\frac{1-C(1-\overline{F}_1(tx_1),\dots,1-\overline{F}_n(tx_n))}{\overline{F}_0(t)} \notag\\
&= \lim_{t\to\infty}\frac{1-C(1-c_1x_1^{-\gamma}\overline{F}_0(t),\dots,1-c_nx_n^{-\gamma}\overline{F}_0(t))}{\overline{F}_0(t)}\notag\\ 
& = \ell(c_1x_1^{-\gamma},\dots,c_nx_n ^{-\gamma})\notag\\
& = \mu^*\left([\mathbf{0},\left(c_1^{-1}x_1^{\gamma},\dots,c_n^{-1}x_n^{\gamma}\right)]^c\right)
\label{eq:260206-4}
\end{align}
where the third equality follows from the monotonicity of copula and the uniform convergence of regularly varying  function, i.e., the convergence of (\ref{eq:260206-1}). The last equality plugs in \eqref{eq:260127-10}. By the definition of the Borel set, we can get \eqref{eq:measure-equation}, i.e., for any $B\in[\mathbf{0},\bm{\infty})\setminus\{\mathbf{0}\}$
\[
\nu(B) = \mu^*\left(\left\{\mathbf{x}:~\left(\mathbf{c}\odot\mathbf{x}\right)^{1/\gamma}\in B\right\}\right).
\]

% below is the special case
\noindent\textbf{Special cases.}
When plugging in \eqref{eq:v2-260127-4} into \eqref{eq:260206-4}, \eqref{eq:cube} is a straightforward conclusion. Then we prove \eqref{eq:ball}
\begin{align*}
\nu(\{\mathbf{x}: \|\mathbf{x}\|_1>n \}) 
& = \mu^*( \{ \mathbf{x}: \|( {\mathbf c}\odot \mathbf{x})^{^{1/\gamma}}\|_1>n\}) \\
& = \mu^*\circ T^{-1}(\{(r,  \bm{\theta}): r^{^{1/\gamma}}\|( {\mathbf c}\odot \bm{\theta})^{^{1/\gamma}}\|_1>n\}) \\
& = \mu^*\circ T^{-1}(\{(r,  \bm{\theta}): r> n^{\gamma}\left(\|\left( {\mathbf c}\odot \bm{\theta}\right)^{^{1/\gamma}}\|_1\right)^{-\gamma}\}) \\
& = \int_{\punitsphereone{n}}\left(\int_{n^{\gamma}\left(\|( {\mathbf c}\odot \bm{\theta})^{^{1/\gamma}}\|_1\right)^{-\gamma}}^\infty r^{-2}\intd r\right)H^*(\mathrm{d}\bm{\theta})\\
& = n^{-\gamma}\int_{\punitsphereone{n}}\|((c_1{\theta}_1)^{1/\gamma},\dots,(c_n{\theta}_n)^{1/\gamma})\|_1^\gamma H^*(\mathrm{d}\bm{\theta}).
\end{align*} 
This completes the proof. 
\end{proof}

\PartialOrder*
\begin{proof}
Let $S_{\perp}^*$ and $S_c^*$ denote the spectral measures of $C_\mathrm{indep}$ and $C_{\mathrm{cdep}}$, respectively. Denote the spectral measure of $C$ as $H^*$. Since the spectral measure of the asymptotic independence consists of point masses at the $n$ vertices of the simplex,  while that of the asymptotic complete dependence consists of
a single point mass of size $n$  at the center point $\mathbf{1}=(1,\dots,1)\in\R^n$, we have for any convex function $\phi$:
$$\int_{\punitsphereone{n}} \phi(\bm\theta)S_c^*(\intd\bm\theta) = n\phi\left(\frac{1}{n}, \dots, \frac{1}{n}\right), \quad \int_{\punitsphereone{n}} \phi(\bm\theta)S_\perp^*(\intd\bm\theta) = \sum_{i=1}^n\phi(\mathbf e_i).$$
Notice that 
\[
\int_{\punitsphereone{n}} \phi(\bm\theta)H^*(\intd\bm\theta)\le\int_{\punitsphereone{n}}\sum_{i=1}^n\theta_i\phi(\mathbf{e}_i)H^*(\intd\bm\theta)=\sum_i\phi(\mathbf e_i)=\int_{\punitsphereone{n}} \phi(\bm\theta)S_\perp^*(\intd\bm\theta)
\]
indicating that $H^*\le_{\rm cx} S_{\perp}^*$.
On the other hand, as $H^*/n$ is a probability measure on $\punitsphereone{n}$, using Jensen's inequality:
\begin{align*}
&\int_{\punitsphereone{n}} \phi(\bm\theta)H^*(\intd\bm\theta)/n \\
\geq&\phi\left(\int_{\punitsphereone{n}}\theta_1 H^*(\intd\bm\theta)/n,\dots,\int_{\punitsphereone{n}}\theta_n H^*(\intd\bm\theta)/n\right) \\
=&  \phi\left(\frac{1}{n}, \dots, \frac{1}{n}\right)    
\end{align*}
indicating that $H^*\ge_{\rm cx} S_c^*$.
\end{proof}

\QVSH*
\begin{proof}
The result follows from \eqref{eq:mrv-tail-general-norm} of \Cref{prop:sum-tail-general-norm} in the main text %the \eqref{eq:rewrite-q} and \eqref{eq:error-limit} 
and that the integral function  in the right-hand side of  \eqref{eq:mrv-tail-general-norm} is convex in $\boldsymbol{\theta}\in\mathcal N_{+,\|\cdot\|_1}^n.$
\end{proof}

% Proposition S4.1
\AbsoluteGauss*
\begin{proof}
For convenience, we let $Z_i=|T_i|$, $i=1,\dots,n$.
Note that asymptotic independence of $(Z_1,\dots,Z_n)$ is equivalent to that pairwise asymptotic independence holds for all $Z_i$, $Z_j$ pairs:
\begin{equation}
\label{eq:pairwise-indep-check}
\lim_{s\downarrow0}\Pr[F_i(Z_i)>1-s\mid F_j(Z_j)>1-s]
=\lim_{s\downarrow0}\Pr[Z_i > \overline{F}_i^{-1}(s)\mid Z_j > \overline{F}_j^{-1}(s)]
=0,
\end{equation}
where $F_i$ and $F_j$ are cumulative distribution functions of $Z_i$ and $Z_j$.
We check the pairwise condition instead. Let $\mu_i$ denote the mean of $T_i$. Then the survival function of $Z_i$ is
\[
\overline{F}_i(t) = \Pr(|T_i|> t)
=\Pr(T_i> t)+\Pr(T_i<- t)
=\overline{\Phi}(t-\mu_i)+\overline{\Phi}(t+\mu_i).
\]
Without loss of generality, we assume $\mu_i,\mu_j\ge0$. 
Let $q_i(s)=\overline{F}_i^{-1}(s)$, $i=1,\dots,n$. Then Equation \eqref{eq:pairwise-indep-check} is equivalent to
\[
\Pr\left[Z_i > q_i(s),Z_j > q_j(s)\right] = o(s),
\]
as $s\downarrow0$.
Since $\mu_k\ge 0$, for $k=i,j$,
\[
\overline{\Phi}(q_k(s)-\mu_k)
\le s
\le 2\overline{\Phi}(q_k(s)-\mu_k),
\]
and hence
\[
\overline{\Phi}^{-1}(s)+\mu_k\le q_k(s)\le \overline{\Phi}^{-1}(s/2)+\mu_k.
\]
Since
\[
\lim_{s\downarrow0}\frac{\overline{\Phi}^{-1}(s)}{\sqrt{-2\log(s)}}=1,
\]
it follows that for $k=i,j$,
\begin{equation}
\label{eq:qk-order}
\lim_{s\downarrow0}\frac{q_k(s)}{\overline{\Phi}^{-1}(s)}
=\lim_{s\downarrow0}\frac{q_k(s)}{\sqrt{-2\log(s)}}
=1.
\end{equation}

Let $X_k=T_k-\mu_k$, $k=i,j$. For any $\varepsilon,\eta\in\{-1,1\}$,
\[
\Pr\left[\varepsilon T_i>q_i(s),\eta T_j>q_j(s)\right]
=
\Pr\left[\varepsilon X_i>q_i(s)-\varepsilon\mu_i,\eta X_j>q_j(s)-\varepsilon\mu_j\right].
\]
Since the pair $\left(\varepsilon X_i,\eta X_j\right)$ is bivariate standard normal with correlation
$r=\varepsilon\eta\rho$ and $|r|<1$,
\begin{align*}
&\Pr(\varepsilon X_i > q_i(s)-\varepsilon\mu_i,\eta X_j > q_j(s)-\varepsilon\mu_j)\\
\le&\Pr(\varepsilon X_i+\eta X_j > q_i(s)+q_j(s)-\varepsilon\mu_i-\eta\mu_j)\\
=&\overline{\Phi}\left(\frac{q_i(s)+q_j(s)-\varepsilon\mu_i-\eta\mu_j}{\sqrt{2(1+r)}}\right). 
\end{align*}
Plugging in \eqref{eq:qk-order}, we have
\[
\lim_{s\downarrow0}
\frac{q_i(s)+q_j(s)-\varepsilon\mu_i-\eta\mu_j}{\sqrt{2(1+r)}} 
\bigg/\sqrt{\frac{2}{1+r}}\overline{\Phi}^{-1}(s)
=1.
\]
Since $\sqrt{2/(1+r)}>1$, by Mills' ratio,
\[
\overline{\Phi}\left(\frac{q_i(s)+q_j(s)-\varepsilon\mu_i-\eta\mu_j}{\sqrt{2(1+r)}}\right)
=o\left(\overline{\Phi}\left(\overline{\Phi}^{-1}(s)\right)\right)
=o(s)
\]
Summing over the four choices of \((\varepsilon,\eta)\) gives
\begin{align*}
\Pr\left(Z_i>q_i(s),Z_j>q_j(s)\right)=o(s).
\end{align*}
Thus every pair $(Z_i,Z_j)$ is asymptotically independent, and consequently $(|T_1|,\dots,$ $|T_n|)$ is asymptotically independent.
\end{proof}

% Proposition S4.2
\Absolutet*
\begin{proof}
Write
$\mathbf{T}:=\mathbf{u}+\mathbf{Y}$,
where $\mathbf{Y}$ is a centered multivariate $t$ random vector with
degrees of freedom $\nu$ and scale matrix $\Sigma$. It is well known that
$\mathbf Y$ is multivariate regularly varying with tail index $\nu$ \citep{nikoloulopoulos2009extreme}.
Hence there exists a nonzero Radon measure $\mu$ on
$\mathbb R^n\setminus\{\mathbf{0}\}$ such that
\[
\frac{\Pr(t^{-1}\mathbf{Y}\in \cdot)}
     {\Pr(\|\mathbf{Y}\|>t)}
\xrightarrow{v}
\mu(\cdot),
\qquad t\to\infty,
\]
where $\xrightarrow{v}$ stands for vague convergence.
Define the componentwise absolute-value map
\[
\phi:\mathbb R^n\setminus\{\boldsymbol 0\}
\to \mathbb R_+^n\setminus\{\boldsymbol 0\},
\qquad
\phi(\mathbf{x})=(|x_1|,\ldots,|x_n|).
\]
The map $\phi$ is continuous and homogeneous, and
$\phi(\mathbf{x})\neq \boldsymbol 0$ whenever
$\mathbf{x}\neq \boldsymbol 0$.  Notice that the regular variation of
$\mathbf{Y}$ is on the whole punctured space, not only on the positive
orthant. Hence the limiting measure $\mu$ contains tail contributions from
all $2^n$ orthants. The map $\phi$ folds these orthant-wise contributions
onto $\mathbb R_+^n\setminus\{\boldsymbol 0\}$.

Let $A\subset \mathbb R_+^n\setminus\{\boldsymbol 0\}$ be a Borel set
bounded away from $\boldsymbol 0$ such that
\[
\mu\bigl(\phi^{-1}(\partial A)\bigr)=0 .
\]
Then $\phi^{-1}(A)$ is bounded away from $\boldsymbol 0$, and by the vague
convergence above,
\[
\begin{aligned}
\frac{\Pr(t^{-1}\phi(\mathbf{Y})\in A)}
     {\Pr(\|\mathbf{Y}\|>t)}
&=
\frac{\Pr(\phi(t^{-1}\mathbf{Y})\in A)}
     {\Pr(\|\mathbf{Y}\|>t)} \\
&=
\frac{\Pr(t^{-1}\mathbf{Y}\in \phi^{-1}(A))}
     {\Pr(\|\mathbf{Y}\|>t)}
\longrightarrow
\mu(\phi^{-1}(A)).
\end{aligned}
\]
Therefore $\phi\left(\mathbf{Y}\right)$ is multivariate regularly varying on $\mathbb R_+^n\setminus\{\boldsymbol 0\}$ with tail index $\nu$ and limit measure $\mu_{|\cdot|}(A)=\mu(\phi^{-1}(A)).$
%In particular, the limiting measure of $|\boldsymbol Y|$ is obtained by folding the tail measure of $\boldsymbol Y$ over all orthants onto the positive orthant.

Next we show that the location shift does not change the tail limit.
For every $\mathbf{Y}\in\mathbb R^n$,
$\big\|\phi(\mathbf{u}+\mathbf{Y})-\phi(\mathbf{Y})\big\|
\le \|\mathbf{u}\|$.
Hence
\[
t^{-1}\big\|\phi(\mathbf{u}+\mathbf{Y})-\phi(\mathbf{Y})\big\|
\le t^{-1}{\|\mathbf{u}\|}\to 0.
\] Therefore $\phi(\mathbf{u}+\mathbf{Y})$ and
$\phi(\mathbf{Y})$ are asymptotically equivalent at scale $t$. By the
standard stability of multivariate regular variation under
asymptotically negligible perturbations, $\phi(\mathbf{u}+\mathbf{Y})$
has the same tail limit as $\phi(\mathbf{Y})$. Consequently,
$(|T_1|,\ldots,|T_n|)$
is multivariate regularly varying on
$\mathbb R_+^n\setminus\{\boldsymbol 0\}$ with tail index $\nu$.

It remains to connect this MRV distribution to its copula. Since
$\Sigma$ is positive definite, each marginal $T_i$ is a non-degenerate
univariate $t$ random variable with degrees of freedom $\nu$. Hence
$|T_i|$ has a continuous marginal distribution and
\[
\Pr(|T_i|>t)\sim c_i t^{-\nu},
\qquad t\to\infty,
\]
for some constant $c_i>0$. Thus the marginals of $|\mathbf{T}|$ are
continuous and regularly varying with the same tail index $\nu$.
By the connection between MRV distributions and MRV copulas used in
Section~\ref{subsec:mrv-copula-and distribution}, or equivalently by
\citep[Corollary~5.18]{resnick2008extreme} and  \citep[Eq.~(8.79)]{beirlant2006statistics}, the copula of
$|\mathbf{T}|$ is an MRV copula. This completes the proof.
\end{proof}
 
% Proposition S4.3
\OverlappingGeneSet*
\begin{proof}
    For each $i$, put $B_i(t)=\{\min_{j\in A_i}U_j\le t/m_i\}$. We first show that
\begin{equation}\label{eq:gene-event-equivalence}
\Pr_0\bigl(\{P_i\le t\}\mathbin{\triangle}B_i(t)\bigr)=o(t),\qquad t\downarrow0,
\end{equation}
where $E\triangle B:=(E\setminus B)\cup(B\setminus E)$.
By Example~\ref{exmp:Gaussian} and Proposition~\ref{prop:abs-preserve-independence}, the $U_j$ are uniform and satisfy, for fixed $a,b>0$,
\begin{equation}\label{eq:gene-pair-small}
\Pr_0(U_j\le as,U_k\le bs)=o(s),\qquad j\ne k.
\end{equation}
Hence $\Pr_0(B_i(t))=t+o(t)$ by the Bonferroni inequalities.
Set $E_i(t)=B_i(t)$ for minimum-$p$, $E_i(t)=\{D_i\le t\}$ for Simes, and $E_i(t)=\{S_i>Q_F(1-t/m_i)\}$ for heavy-tailed combinations.
For Simes, {$B_i(t)\subseteq E_i(t)$, and} $E_i(t)\setminus B_i(t)$ requires two distinct $U_j\le t$, so its probability is $o(t)$.
For heavy-tailed combinations, Theorem~3 of \citet{gui2023aggregating} gives $\Pr_0(E_i(t)\triangle B_i(t))=o(t)$.
Thus $\Pr_0(E_i(t))=t+o(t)$. Direct outputs have rejection event $E_i(t)$, proving \eqref{eq:gene-event-equivalence}. For exact outputs, continuity and nesting give
\[
\Pr_0\bigl(E_i(t)\triangle\{P_i\le t\}\bigr)
=\bigl|\Pr_0(E_i(t))-t\bigr|=o(t),
\]
which proves \eqref{eq:gene-event-equivalence} as well.

For fixed $\mathbf v\in[0,\infty)^n$, \eqref{eq:gene-event-equivalence} and \eqref{eq:gene-pair-small} yield
\begin{align*}
\Pr_0\!\left(\bigcup_i\{P_i\le sv_i\}\right)
&=\Pr_0\!\left(\bigcup_{j\in\cup_i A_i}
\left\{U_j\le s\max_{i:j\in A_i}\frac{v_i}{m_i}\right\}\right)+o(s)\\
&=s\ell(\mathbf v)+o(s),
\end{align*}
where
\[
\ell(\mathbf v)=\sum_{j\in\cup_i A_i}\max_{i:j\in A_i}\frac{v_i}{m_i}.
\]
Taking $\mathbf v=\mathbf e_i$ gives $\Pr_0(P_i\le s)\sim s$, so marginal standardization and Definition~\ref{def:mrv-copula} give MRV. Inclusion-exclusion gives
\[
\Pr_0(P_i\le s,P_k\le s)
=\frac{s|A_i\cap A_k|}{\max\{m_i,m_k\}}+o(s),
\]
and division by $\Pr_0(P_k\le s)\sim s$ proves the limit. This completes the proof.
\end{proof}

% Proposition 4.4
\AlternativeA*
\begin{proof}
Since
\begin{equation*}
    (1-P_1,\dots,1-P_n) \overset{d}{=}(G_1(G^{-1}_1(U_1)+\mu_1),\dots, G_n(G^{-1}_n(U_n)+\mu_n)),
\end{equation*}
and all $G_i$s are continuous, \Cref{asmp:p-alternative} (i) and (ii) are straightforward conclusions. For (iii), we first compute
\[
c_{ij}:=\lim_{t\downarrow0}\frac{\Pr(P_i\le t)}{\Pr(P_j\le t)}
\]
for all $i, j$ pairs. 

When all $G_i$s belong to $\mathscr{D}$ (i),
\begin{align*}
&c_{ij} 
= \lim_{t\downarrow0}\frac{\Pr(1-G_i\left(G_i^{-1}(U_i)+\mu_i\right)\le t)}{\Pr(1-G_j\left(G_j^{-1}(U_j)+\mu_j\right)\le t)}
= \lim_{t\downarrow0}\frac{\Pr(\overline{G}_i\left(\overline{G}_i^{-1}(1-U_i)+\mu_i\right)\le t)}{\Pr(\overline{G}_j\left(\overline{G}_j^{-1}(1-U_j)+\mu_j\right)\le t)}\\
&=\lim_{t\downarrow0}\frac{\overline{G}_i\left(\overline{G}_i^{-1}(t)-\mu_i\right)}{\overline{G}_j\left(\overline{G}_j^{-1}(t)-\mu_j\right)}
=\lim_{t\downarrow0}\frac{\overline{G}_i\left(\overline{G}_i^{-1}(t)\right)}{\overline{G}_j\left(\overline{G}_j^{-1}(t)\right)}
=\lim_{t\downarrow0}\frac{t}{t}
=1.
\end{align*}
That is, $I_{\mathbf{P}}=[n]$. To confirm (iii) and $\beta=1$, it suffices to prove that for any $F\in\mathscr{R}_{-1}$, $Q_F(1-P_1)\in\mathscr{R}_{-1}$. The tail probability of $Q_F(1-P_1)$ is
\begin{align*}
&\Pr(Q_F(1-P_1)>x)
=\Pr\left(G_1(G^{-1}_1(U_1)+\mu_1)>F(x)\right)\\
&=\Pr\left(\overline{G}_1(\overline{G}^{-1}_1(1-U_1)+\mu_1)<\overline{F}(x)\right)
=\overline{G}_1\left(\overline{G}_1^{-1}(\overline{F}(x))-\mu_1\right).
\end{align*}
Since for all $y>0$,
\[
\lim_{x\to\infty}\frac{\overline{G}_1\left(\overline{G}_1^{-1}(\overline{F}(xy))-\mu_1\right)}{\overline{G}_1\left(\overline{G}_1^{-1}(\overline{F}(x))-\mu_1\right)}
=\lim_{x\to\infty}\frac{\overline{G}_1\left(\overline{G}_1^{-1}(\overline{F}(xy))\right)}{\overline{G}_1\left(\overline{G}_1^{-1}(\overline{F}(x))\right)}
=\lim_{x\to\infty}\frac{\overline{F}(xy)}{\overline{F}(x)}
=y^{-1},
\]
$Q_F(1-P_1)\in\mathscr{R}_{-1}$ by the definition.

When all $G_i$s belong to $\mathscr{D}$ (ii), i.e., all $G_i$s are standard normals,
\begin{align*}
&c_{ij} 
=\lim_{t\downarrow0}\frac{\overline{G}_i\left(\overline{G}_i^{-1}(t)-\mu_i\right)}{\overline{G}_j\left(\overline{G}_j^{-1}(t)-\mu_j\right)}
=\lim_{t\downarrow0}\frac{\overline{\Phi}\left(\overline{\Phi}^{-1}(t)-\mu_i\right)}{\overline{\Phi}\left(\overline{\Phi}^{-1}(t)-\mu_j\right)}\\
&=\lim_{t\downarrow0}\frac{\overline{\Phi}\left(\overline{\Phi}^{-1}(t)\right)}{\overline{\Phi}\left(\overline{\Phi}^{-1}(t)\right)}\times\exp\left\{\mu_i\overline{\Phi}^{-1}(t)-\mu_j\overline{\Phi}^{-1}(t)\right\}\\
&=\lim_{t\downarrow0}\frac{\exp\left\{\mu_i\sqrt{-2\log{t}-\log\log\frac{1}{t}-\log(4\pi)+o(1)}\right\}}{\exp\left\{\mu_j\sqrt{-2\log{t}-\log\log\frac{1}{t}-\log(4\pi)+o(1)}\right\}}\\
&=\lim_{t\downarrow0}\exp\left\{\frac{-2(\mu_i^2-\mu_j^2)\log{t}}
{\mu_i\sqrt{-2\log{t}}+\mu_j\sqrt{-2\log{t}}}\right\}
=\begin{cases}
    0,~~&\mu_i<\mu_j\\
    1,~~&\mu_i=\mu_j\\
    \infty,~~&\mu_i>\mu_j
\end{cases}
.
\end{align*}
That is, $I_{\mathbf{P}}=\{i:\mu_i=\max_{j\in[n]}\mu_j\}$. Without loss of generality, we suppose $\mu_1=\max_{j\in[n]}\mu_j$. To confirm (iii) and $\beta=1$, it suffices to prove that for any $F\in\mathscr{R}_{-1}$, $Q_F(1-P_1)\in\mathscr{R}_{-1}$. The tail probability of $Q_F(1-P_1)$ is
\begin{align*}
&\Pr(Q_F(1-P_1)>x)
=\overline{G}_1\left(\overline{G}_1^{-1}(\overline{F}(x))-\mu_1\right)
=\overline{\Phi}\left(\overline{\Phi}^{-1}(\overline{F}(x))-\mu_1\right)
\end{align*}
We check the definition of regularly varying tailed distribution: for all $y>0$,
\begin{align*}
&\lim_{x\to\infty}\frac{\overline{\Phi}\left(\overline{\Phi}^{-1}(\overline{F}(xy))-\mu_1\right)}{\overline{\Phi}\left(\overline{\Phi}^{-1}(\overline{F}(x))-\mu_1\right)}
=\lim_{t\downarrow0}\frac{\overline{\Phi}\left(\overline{\Phi}^{-1}\left(y^{-1}t\right)-\mu_1\right)}{\overline{\Phi}\left(\overline{\Phi}^{-1}\left(t\right)-\mu_1\right)}\\
&=\lim_{t\downarrow0}\frac{\Phi\left(\overline{\Phi}^{-1}\left(y^{-1}t\right)\right)}{\Phi\left(\overline{\Phi}^{-1}\left(t\right)\right)}\times\exp\left\{\mu_1\left(\overline{\Phi}^{-1}\left(y^{-1}t\right)-\overline{\Phi}^{-1}\left(t\right)\right)\right\}\\
&=\lim_{t\downarrow0}\frac{y^{-1}t}{t}\times\frac{\exp\left[\mu_1\sqrt{-2\log(y^{-1})-2\log{t}-\log\log(y^{-1}/t)-\log(4\pi)+o(1)}\right]}{\exp\left[\mu_1\sqrt{-2\log{t}-\log\log(1/t)-\log(4\pi)+o(1)}\right]}\\
&=y^{-1}\times \lim_{t\downarrow0}\exp\left\{\frac{\log{y}}
{\mu_1\sqrt{-2\log{t}}}\right\}
=y^{-1}.
\end{align*}
Then the proposition follows.
\end{proof}

% Proposition S4.5
\AlternativeB*
\begin{proof}
Since
\begin{equation*}
    (1-P_1,\dots,1-P_n) \overset{d}{=} (1-(1-U_1)^{\beta_1},\dots,1-(1-U_n)^{\beta_n}),
\end{equation*}
\Cref{asmp:p-alternative} (i) and (ii) are straightforward conclusions. For (iii), we first compute 
\[
c_{ij}:=\lim_{t\downarrow0}\frac{\Pr(P_i\le t)}{\Pr(P_j\le t)}
\]
to find $I_{\mathbf{P}}$. Since
\begin{align*}
&c_{ij}
=\lim_{t\downarrow0}\frac{\Pr((1-U_i)^{\beta_i}\le t)}{\Pr((1-U_j)^{\beta_j}\le t)}
=\lim_{t\downarrow0}t^{\frac{1}{\beta_i}-\frac{1}{\beta_j}}
=\begin{cases}
    0,\quad &\beta_i<\beta_j\\
    1,\quad &\beta_i=\beta_j\\
    \infty,\quad &\beta_i>\beta_j\\
\end{cases}~,
\end{align*}
$I_{\mathbf{P}}=\{i\in[n]:\beta_i=\beta_1\}$. For any $F\in\mathscr{R}_{-1}$, we check the distribution of $Q_F(1-P_1)$. Its tail probability is
\[
\Pr(Q_F(1-P_1)>x)=\Pr\left(1-(1-U_1)^{\beta_1}>F(x)\right)=\overline{F}(x)^{1/\beta_1}.
\]
For all $y>0$,
\[
\lim_{x\to\infty}\frac{\overline{F}(xy)^{1/\beta_1}}{\overline{F}(x)^{1/\beta_1}}
=\lim_{x\to\infty}\frac{y^{-1/\beta_1}\overline{F}(x)^{1/\beta_1}}{\overline{F}(x)^{1/\beta_1}}
=y^{-1/\beta_1}.
\]
Accordingly, $Q_F(1-P_1)\in\mathscr{R}_{-1/\beta_1}$ and this completes the proof.
\end{proof}

\bibliographystylesupp{plainnat}
\bibliographysupp{paper/ref}

\newpage

\clearpage
\section{Supplementary figures}
\label{sec:more-exp-results}

\begin{center}
\includegraphics[width=0.8\textwidth]{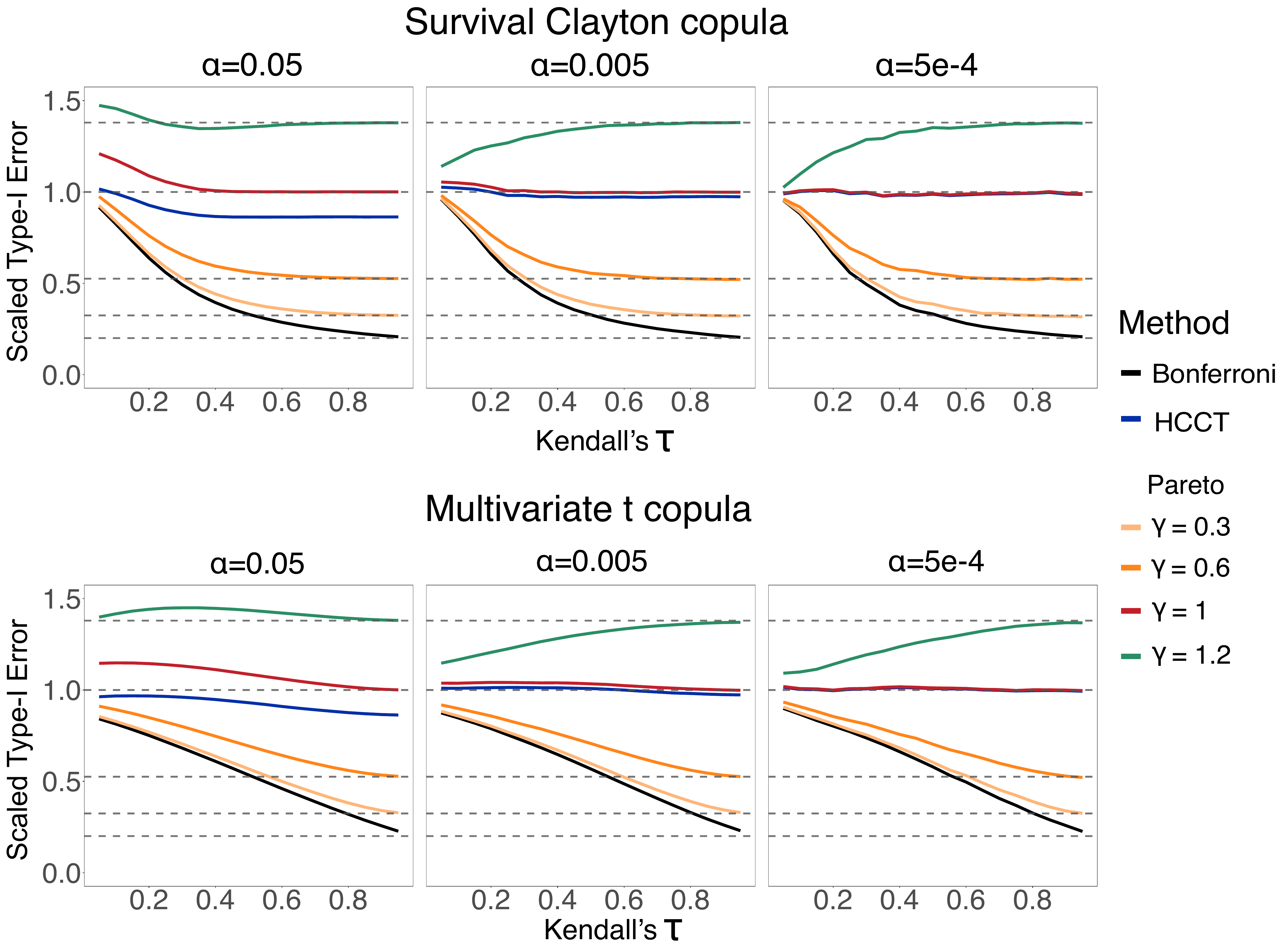}
\captionof{figure}{Empirical scaled Type-I error of the combination test using Pareto distributions and $n = 5$. The $5$ dashed horizontal lines, from bottom to top, indicate the bound $n^{\gamma -1}$ for $\gamma = 0, 0.3, 0.6, 1$ and $1.2$.}
\label{fig:validity-pareto-5}
\end{center}

% \newpage
% \vspace*{\fill}
\begin{figure}[!ht]
\centering
\includegraphics[width=0.8\textwidth]{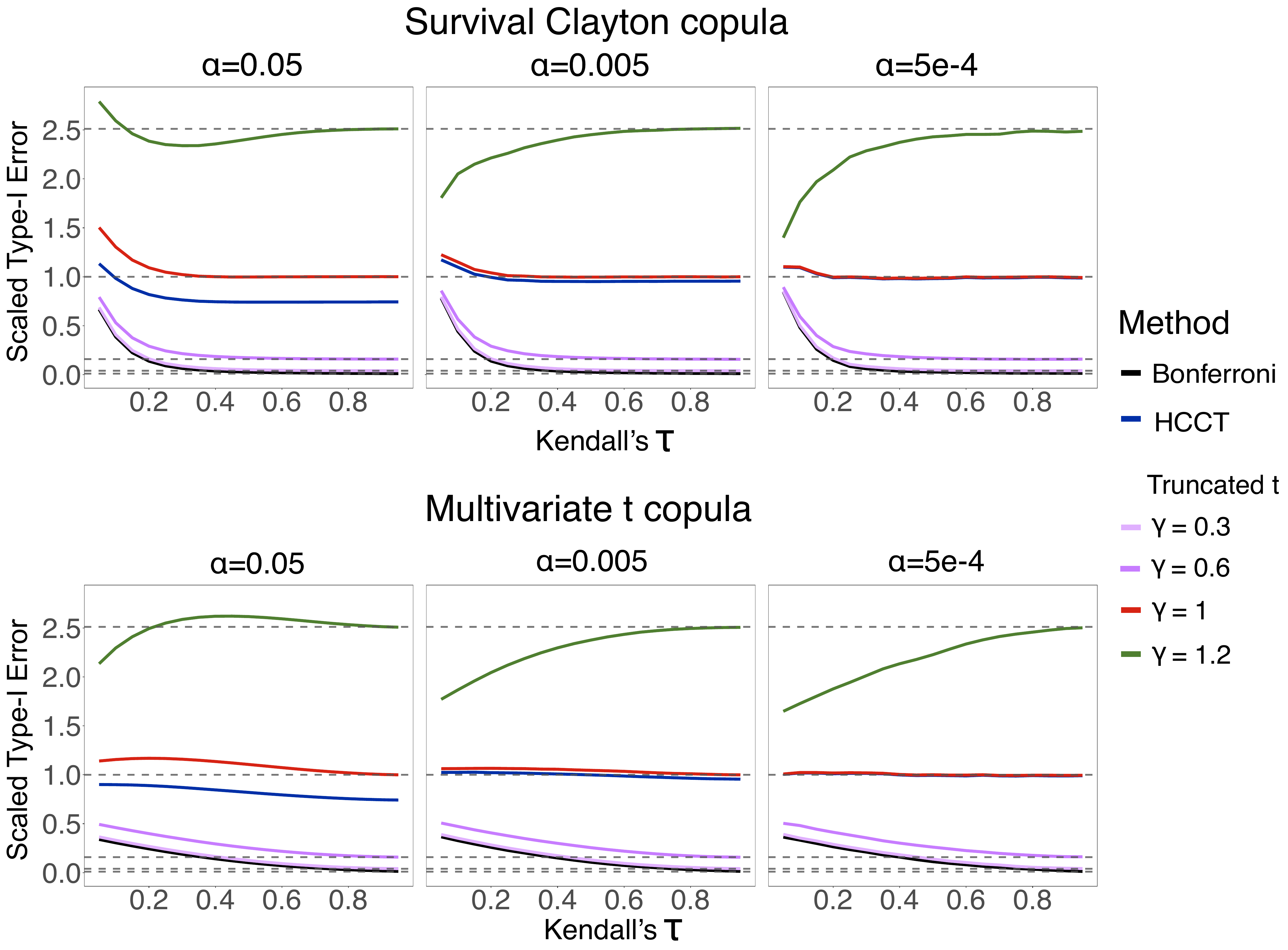}
\caption{Empirical scaled type-I-error of the combination test using truncated $t$ distributions and $n = 100$. The $5$ dashed horizontal lines, from bottom to top, indicate the bound $n^{\gamma -1}$ for $\gamma = 0, 0.3, 0.6, 1$ and $1.2$. }
\label{fig:validity-tt-100}
\end{figure}
% \vspace*{\fill}
   
% \newpage
% \vspace*{\fill}
\begin{figure}[!ht]
\centering
\includegraphics[width=0.9\textwidth]{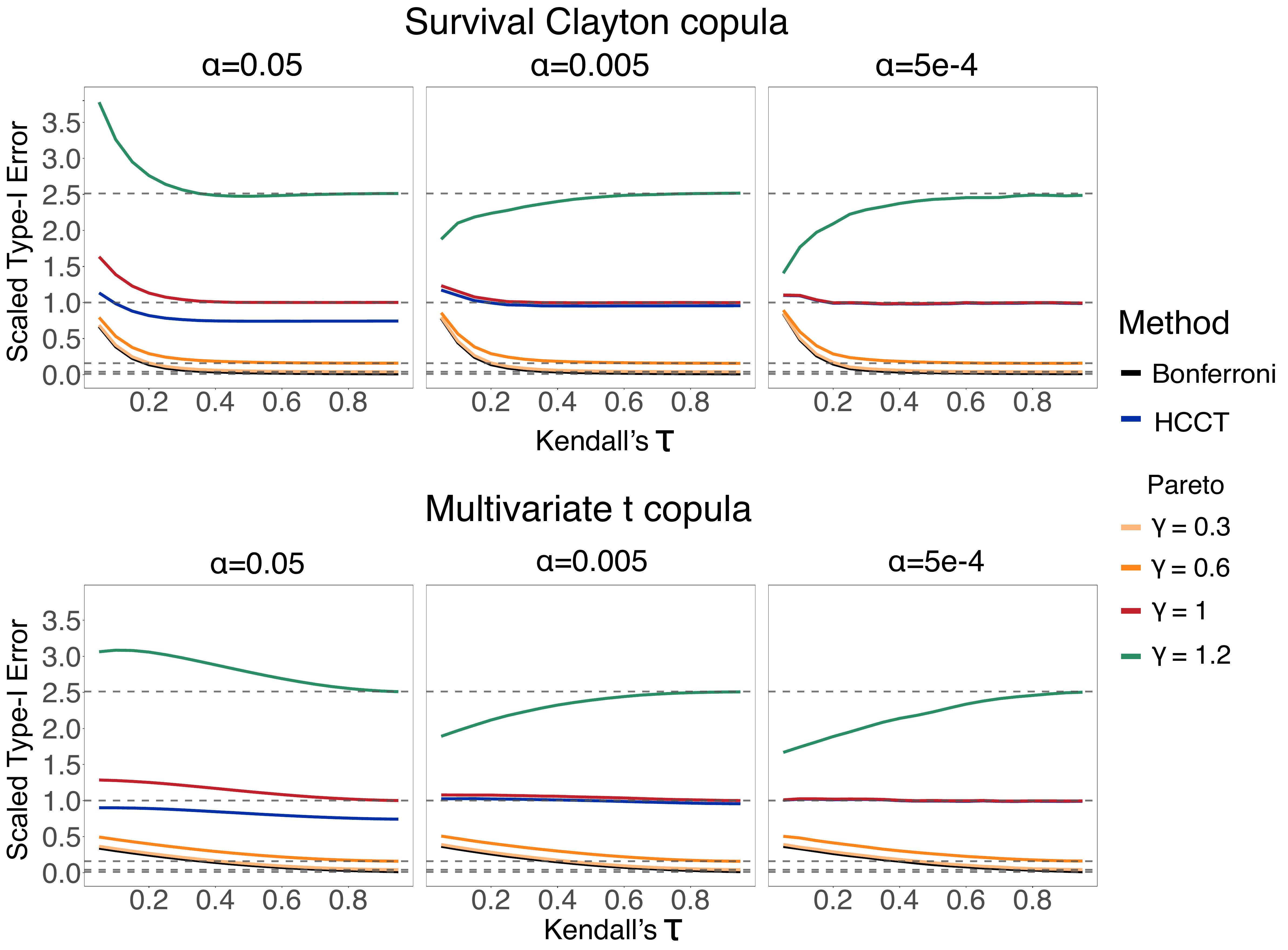}
 \caption{Empirical scaled Type-I error of the combination test using Pareto
distributions and $n = 100$. The $5$ dashed horizontal lines, from bottom to top, indicate the bound $n^{\gamma -1}$ for $\gamma = 0, 0.3, 0.6, 1$ and $1.2$. }
\label{fig:validity-pareto-100}
\end{figure}
% \vspace*{\fill}

%*************** Power dependence starts here *************** %
% \newpage
% \vspace*{\fill}
\begin{figure}[!ht]
\centering
\includegraphics[width=\textwidth]{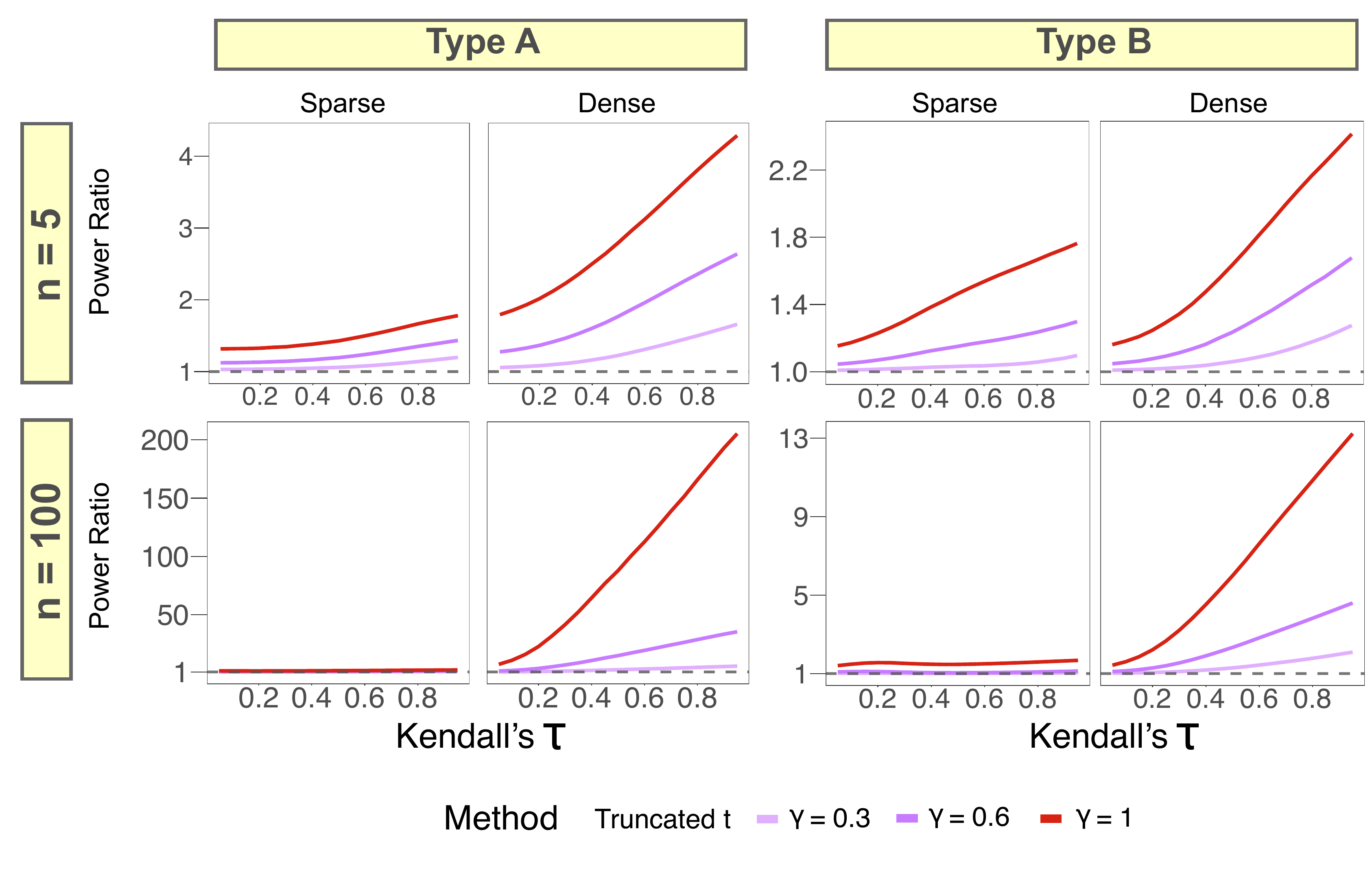}
\caption{
Power ratio of the combination test to the Bonferroni test versus Kendall's $\tau$ using truncated $t$ distributions under different alternative types at significance level $\alpha = 0.05$. The underlying dependence among base p-values is modeled using the Clayton copula. %Truncated t and Pareto lines have large overlaps, so some parts of Pareto lines do not show explicitly in the figures.
}
\label{fig:power-dep-clayton-5e-2}
\end{figure}
% \vspace*{\fill}

% \newpage
% \vspace*{\fill}
\begin{figure}[!ht]
\centering
\includegraphics[width=\textwidth]{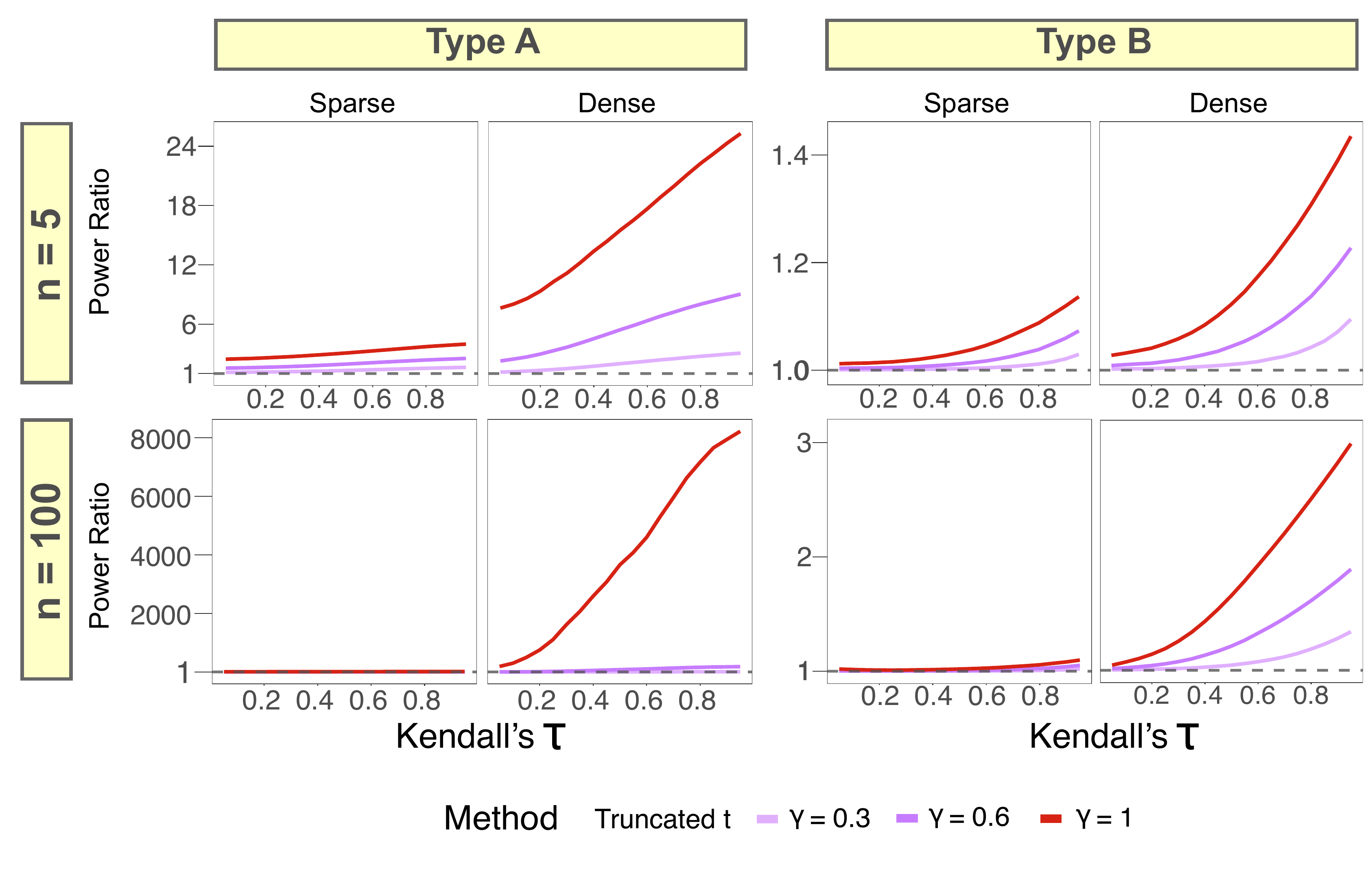}
\caption{Power ratio of the combination test to the Bonferroni test versus Kendall's $\tau$ using truncated $t$ distributions under different alternative types at significance level $\alpha = 5\times 10^{-4}$. The underlying dependence among base p-values is modeled using the Clayton copula. %Truncated t and Pareto lines completely overlap, so no Pareto lines show explicitly in the figures.
}
\label{fig:power-dep-clayton-5e-4}
\end{figure}
% \vspace*{\fill}

% \newpage
% \vspace*{\fill}
\begin{figure}[!ht]
\centering
\includegraphics[width=\textwidth]{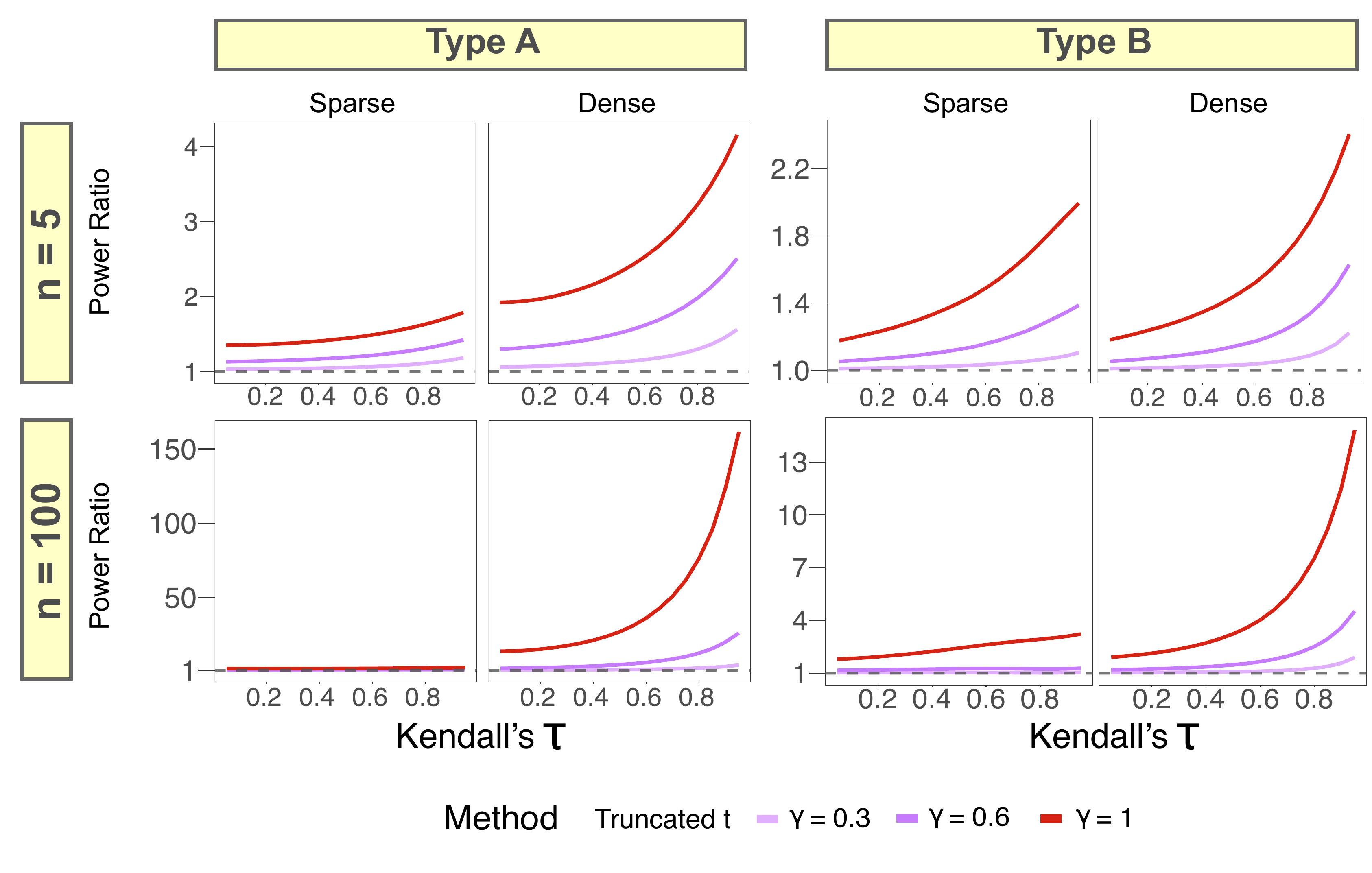}
\caption{Power ratio of the combination test to the Bonferroni test versus Kendall's $\tau$ using truncated $t$ distributions under different alternative types at significance level $\alpha = 0.05$. The underlying dependence among base p-values is modeled using the multivariate $t$ copula. %Truncated t and Pareto lines have large overlaps, so some parts of Pareto lines do not show explicitly in the figures.
}
\label{fig:power-dep-t-5e-2}
\end{figure}
% \vspace*{\fill}

% \newpage
% \vspace*{\fill}
\begin{figure}[!ht]
\centering
\includegraphics[width=\textwidth]{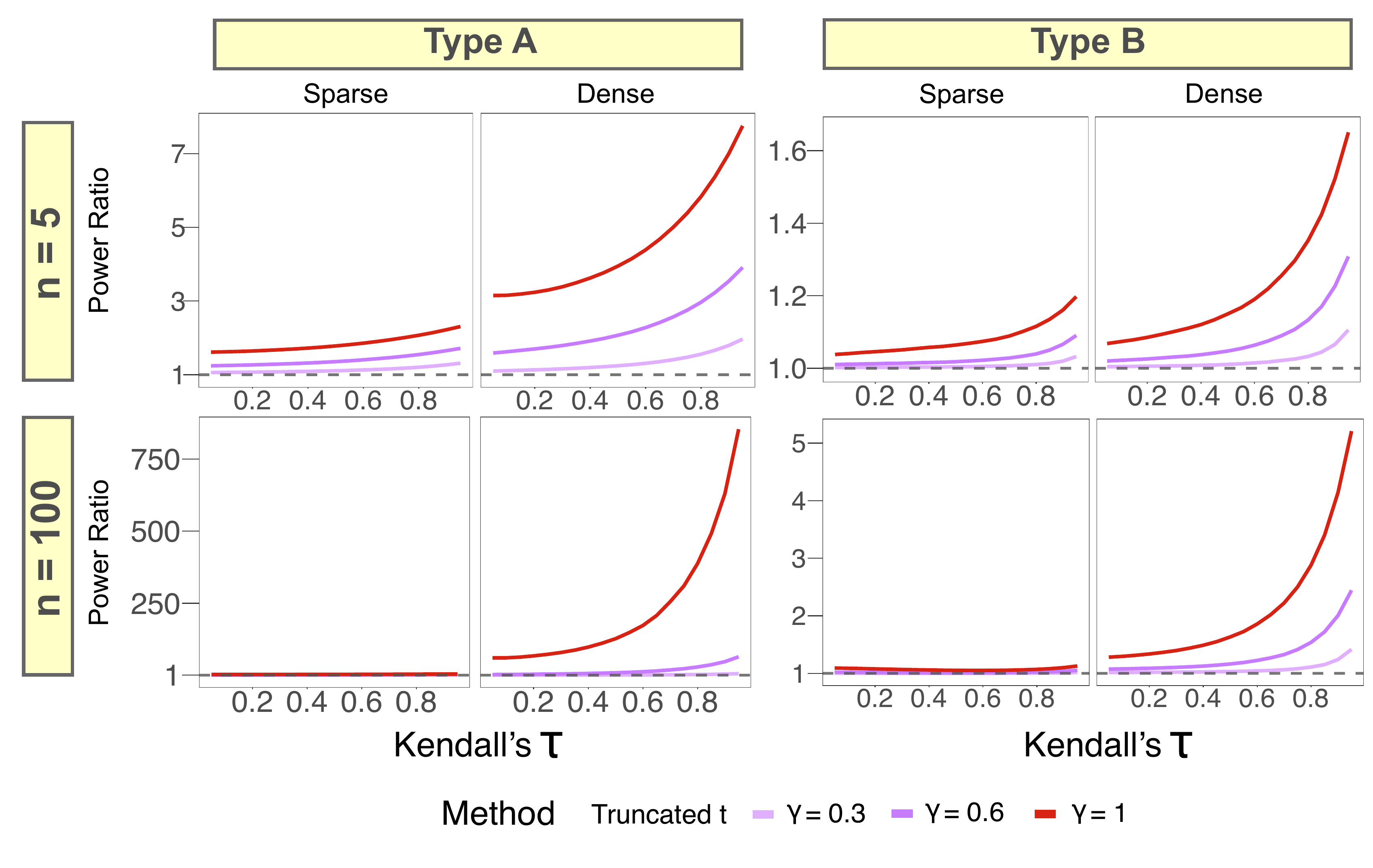}
\caption{Power ratio of the combination test to the Bonferroni test versus Kendall's $\tau$ using truncated t distributions under different alternative types at significance level $\alpha = 5\times 10^{-3}$. The underlying dependence among base p-values is modeled using the multivariate $t$ copula. %Truncated t and Pareto lines completely overlap, so no Pareto lines show explicitly in the figures.
}
\label{fig:power-dep-t-5e-3}
\end{figure}
% \vspace*{\fill}

% \newpage
% \vspace*{\fill}
\begin{figure}[!ht]
\centering
\includegraphics[width=\textwidth]{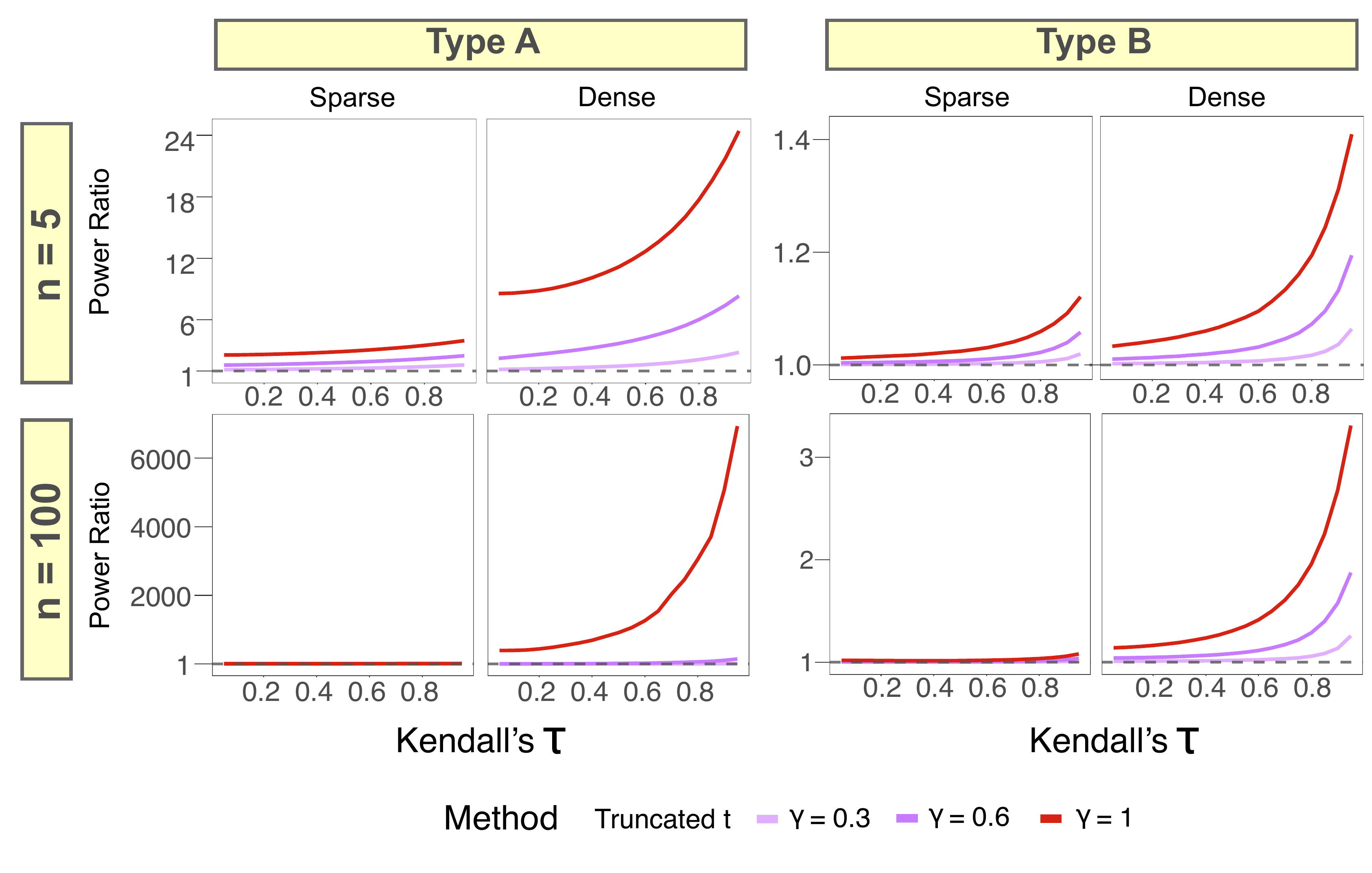}
\caption{Power ratio of the combination test to the Bonferroni test versus Kendall's $\tau$ using truncated $t$ distributions under different alternative types at significance level $\alpha = 5\times 10^{-4}$. The underlying dependence among base p-values is modeled using the multivariate $t$ copula.  %Truncated t and Pareto lines completely overlap, so no Pareto lines show explicitly in the figures.
}
\label{fig:power-dep-t-5e-4}
\end{figure}
% \vspace*{\fill}

% \newpage
% \vspace*{\fill}
\begin{figure}[!ht]
\centering
\includegraphics[width=\textwidth]{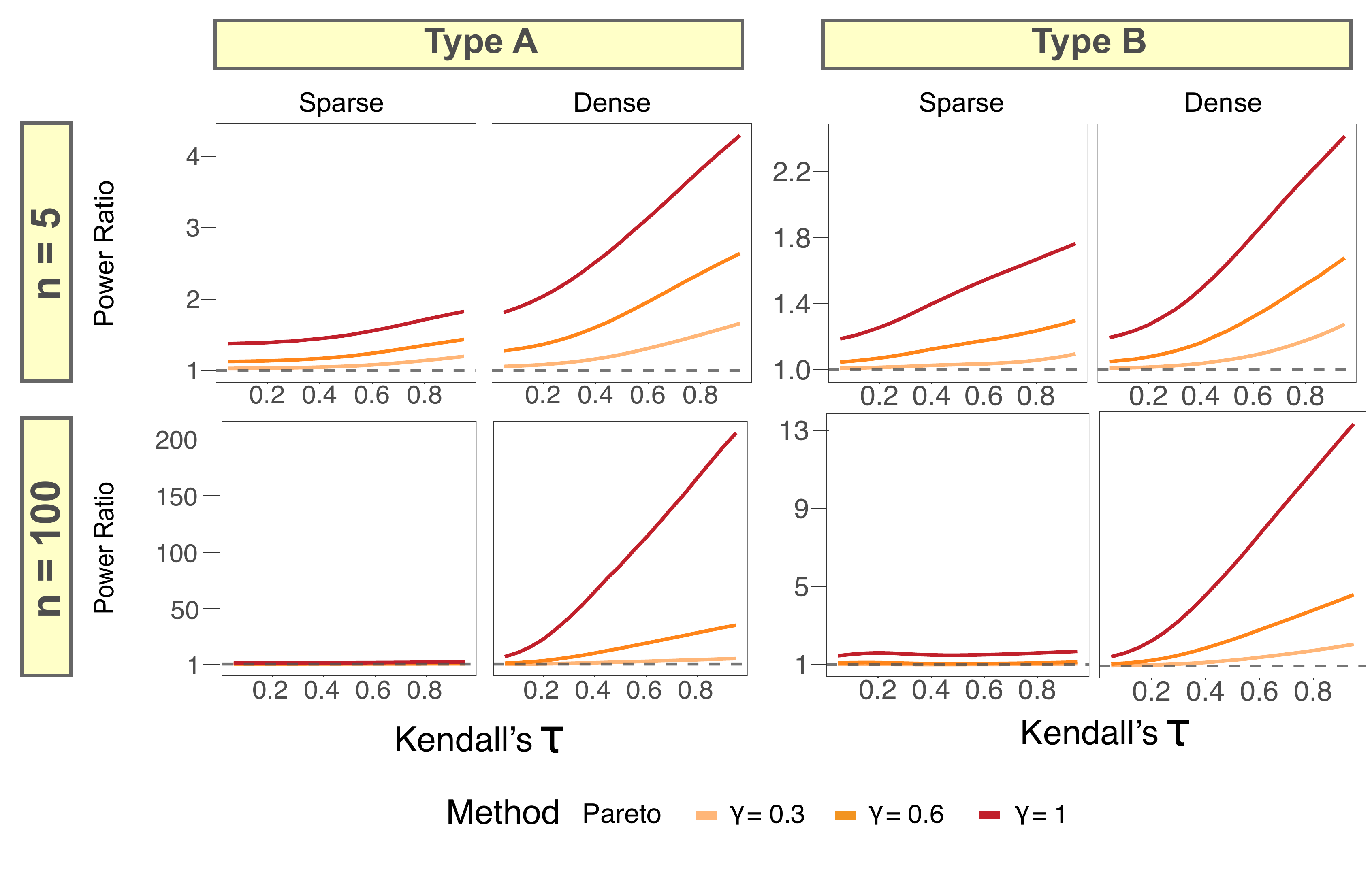}
\caption{Power ratio of the combination test to the Bonferroni test versus Kendall's $\tau$ using Pareto distributions under different alternative types at significance level $\alpha = 0.05$. The underlying dependence among base p-values is modeled using the Clayton copula. %Truncated t and Pareto lines have large overlaps, so some parts of Pareto lines do not show explicitly in the figures.
}
\label{fig:pareto-power-dep-clayton-5e-2}
\end{figure}
% \vspace*{\fill}

% \newpage
% \vspace*{\fill}
\begin{figure}[!ht]
\centering
\includegraphics[width=\textwidth]{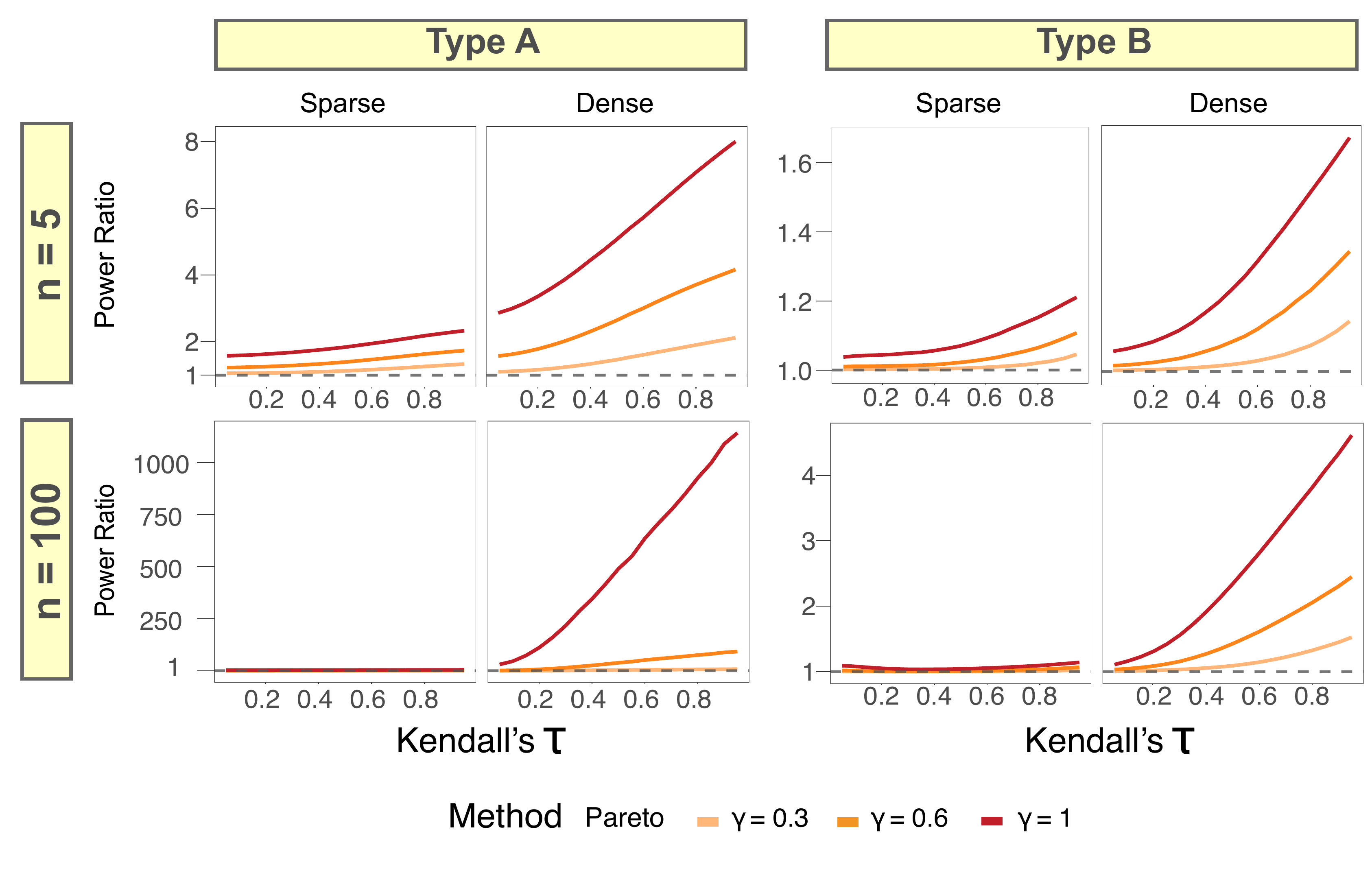}
\caption{Power ratio of the combination test to the Bonferroni test versus Kendall's $\tau$ using Pareto distributions under different alternative types at significance level $\alpha = 5\times 10^{-3}$. The underlying dependence among base p-values is modeled using the Clayton copula. %Truncated t and Pareto lines completely overlap, so no Pareto lines show explicitly in the figures.
}
\label{fig:paretopower-dep-clayton-5e-3}
\end{figure}
% \vspace*{\fill}

% \newpage
% \vspace*{\fill}
\begin{figure}[!ht]
\centering
\includegraphics[width=\textwidth]{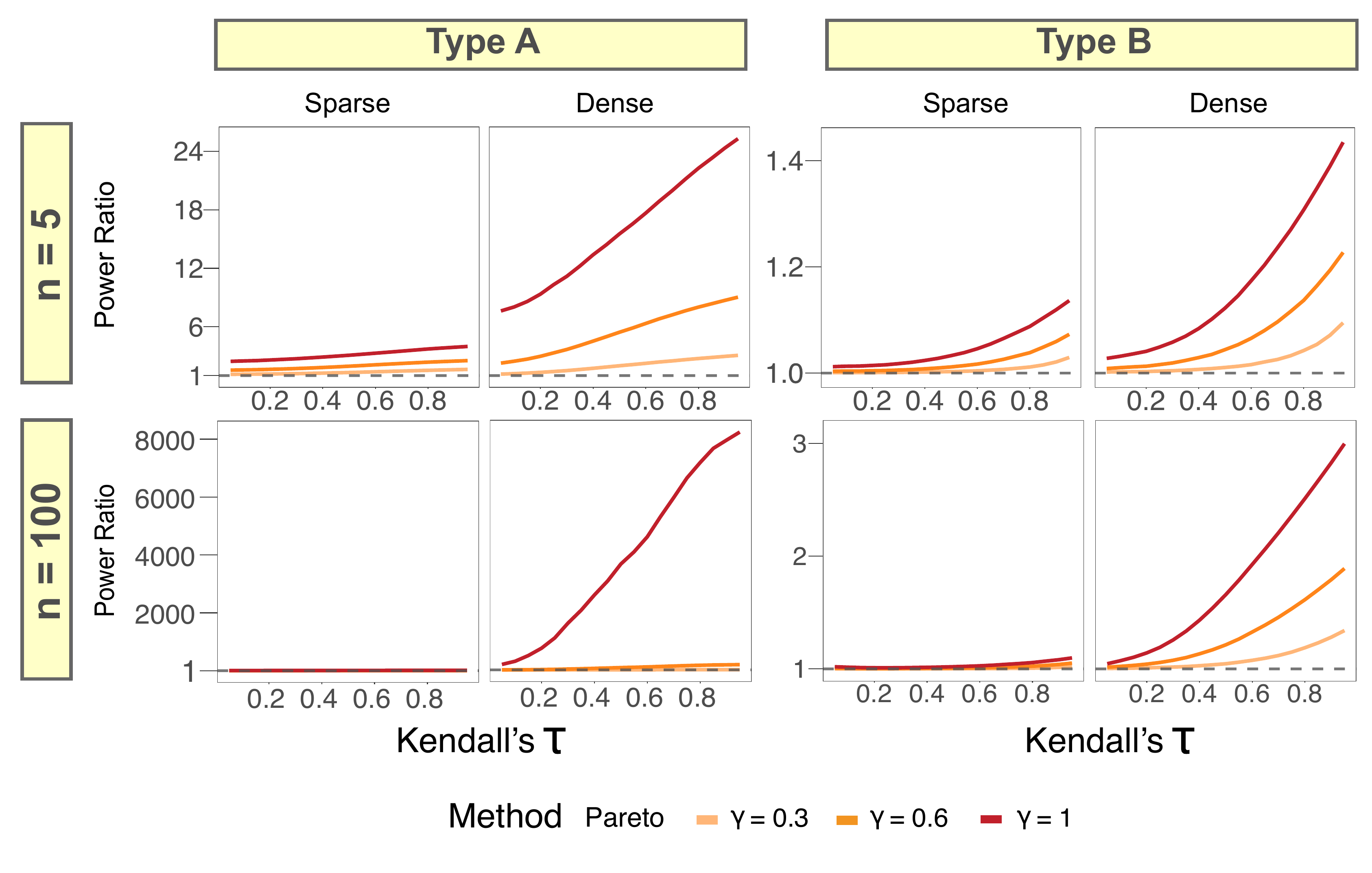}
\caption{Power ratio of the combination test to the Bonferroni test versus Kendall's $\tau$ using Pareto distributions under different alternative types at significance level $\alpha = 5\times 10^{-4}$. The underlying dependence among base p-values is modeled using the Clayton copula.  %Truncated t and Pareto lines completely overlap, so no Pareto lines show explicitly in the figures.
}
\label{fig:pareto-power-dep-clayton-5e-4}
\end{figure}
% \vspace*{\fill}

% \newpage
% \vspace*{\fill}
\begin{figure}[!ht]
\centering
\includegraphics[width=\textwidth]{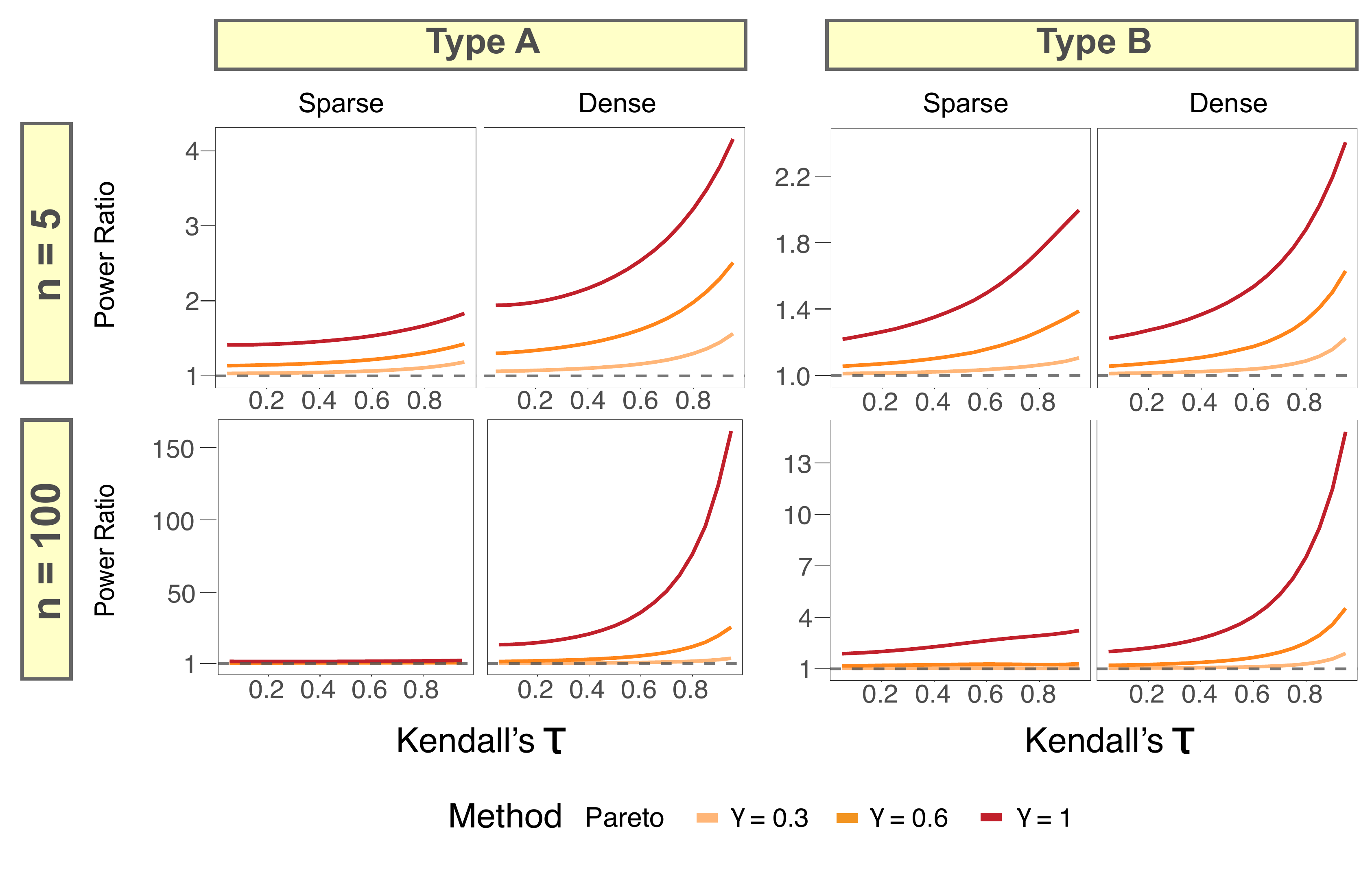}
\caption{Power ratio of the combination test to the Bonferroni test versus Kendall's $\tau$ using Pareto distributions under different alternative types at significance level $\alpha = 0.05$. The underlying dependence among base p-values is modeled using the multivariate $t$ copula. %Truncated t and Pareto lines have large overlaps, so some parts of Pareto lines do not show explicitly in the figures.
}
\label{fig:pareto-power-dep-t-5e-2}
\end{figure}
% \vspace*{\fill}

% \newpage
% \vspace*{\fill}
\begin{figure}[!ht]
\centering
\includegraphics[width=\textwidth]{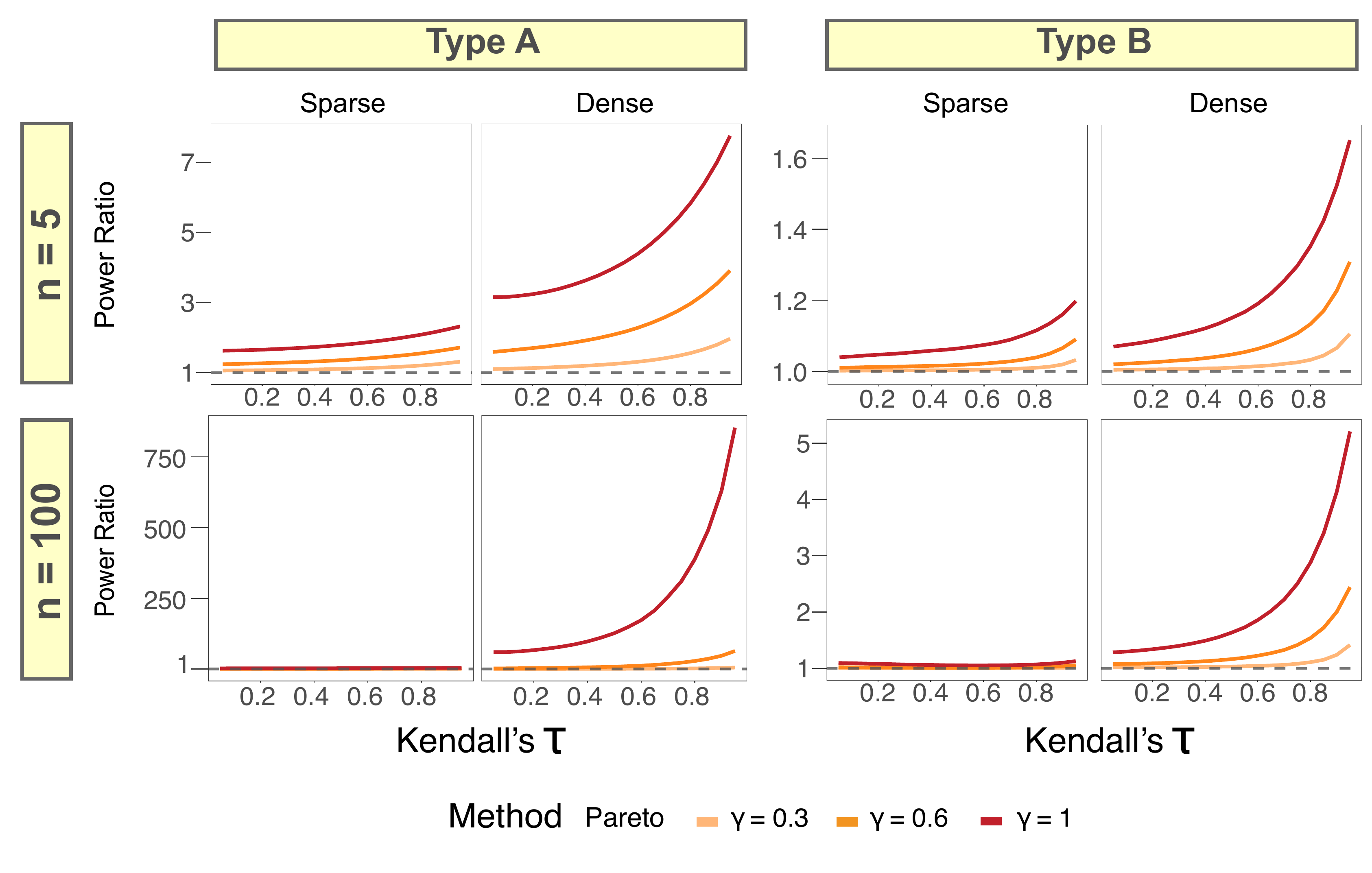}
\caption{Power ratio of the combination test to the Bonferroni test versus Kendall's $\tau$ using Pareto distributions under different alternative types at significance level $\alpha = 5\times 10^{-3}$. The underlying dependence among base p-values is modeled using the multivariate $t$ copula. %Truncated t and Pareto lines completely overlap, so no Pareto lines show explicitly in the figures.
}
\label{fig:pareto-power-dep-t-5e-3}
\end{figure}
% \vspace*{\fill}

% \newpage
% \vspace*{\fill}
\begin{figure}[!ht]
\centering
\includegraphics[width=\textwidth]{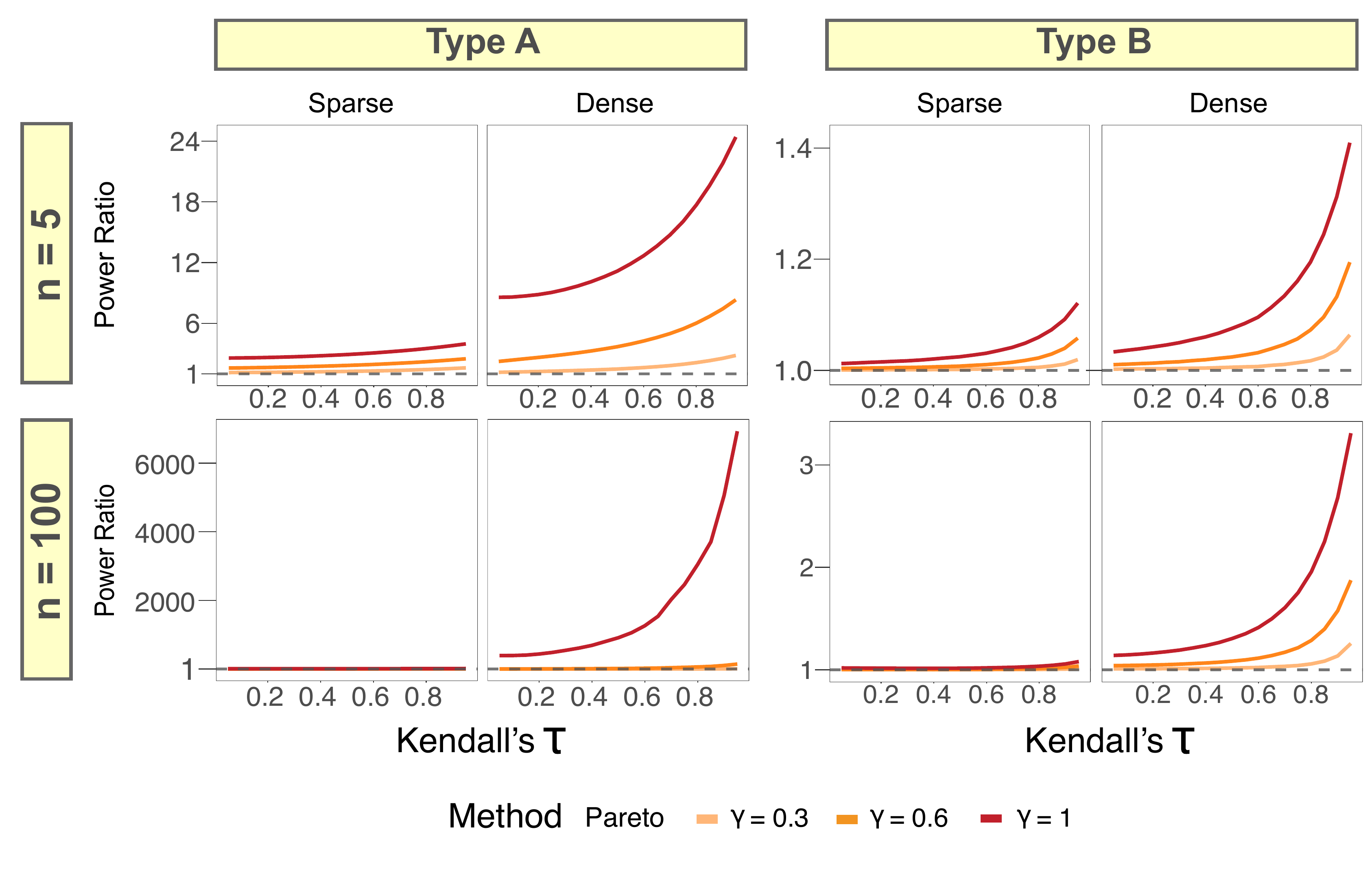}
\caption{Power ratio of the combination test to the Bonferroni test versus Kendall's $\tau$ using Pareto distributions under different alternative types at significance level $\alpha = 5\times 10^{-4}$. The underlying dependence among base p-values is modeled using the multivariate $t$ copula.  %Truncated t and Pareto lines completely overlap, so no Pareto lines show explicitly in the figures.
}
\label{fig:pareto-power-dep-t-5e-4}
\end{figure}
% \vspace*{\fill}

%*************** Signals starts here *************** %
% \newpage
% \vspace*{\fill}
\begin{figure}[t]
\centering
\includegraphics[width=\textwidth]{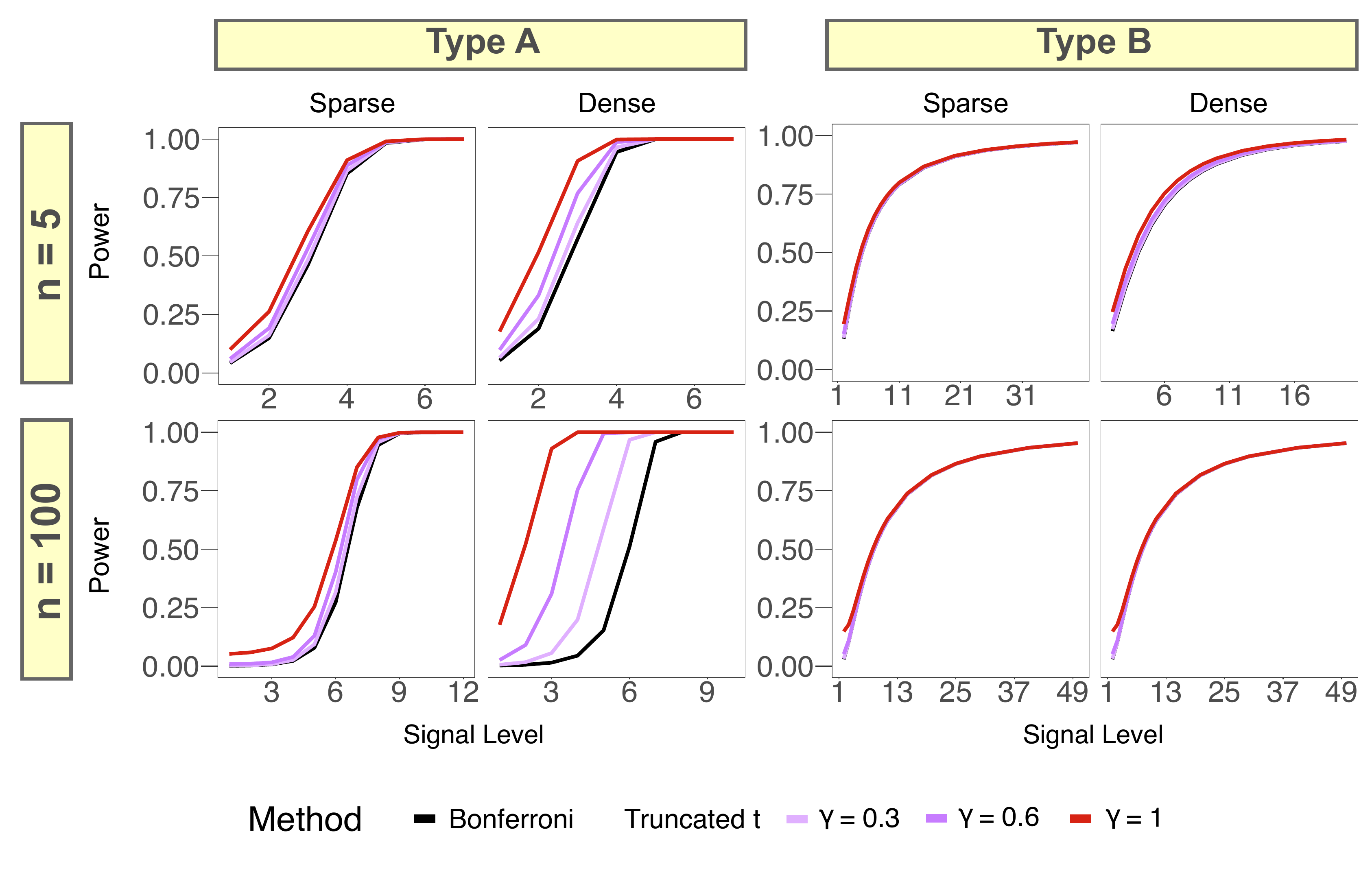}
\caption{The power of the combination test using truncated $t$ distributions and the Bonferroni test versus signal levels under different alternative types at significance level $\alpha=0.05$. The underlying
dependence among base p-values is modeled using the Clayton copula.}
\label{fig:power-signal-clayton-5e-2}
\end{figure}
% \vspace*{\fill}

% \newpage
% \vspace*{\fill}
\begin{figure}[t]
\centering
\includegraphics[width=\textwidth]{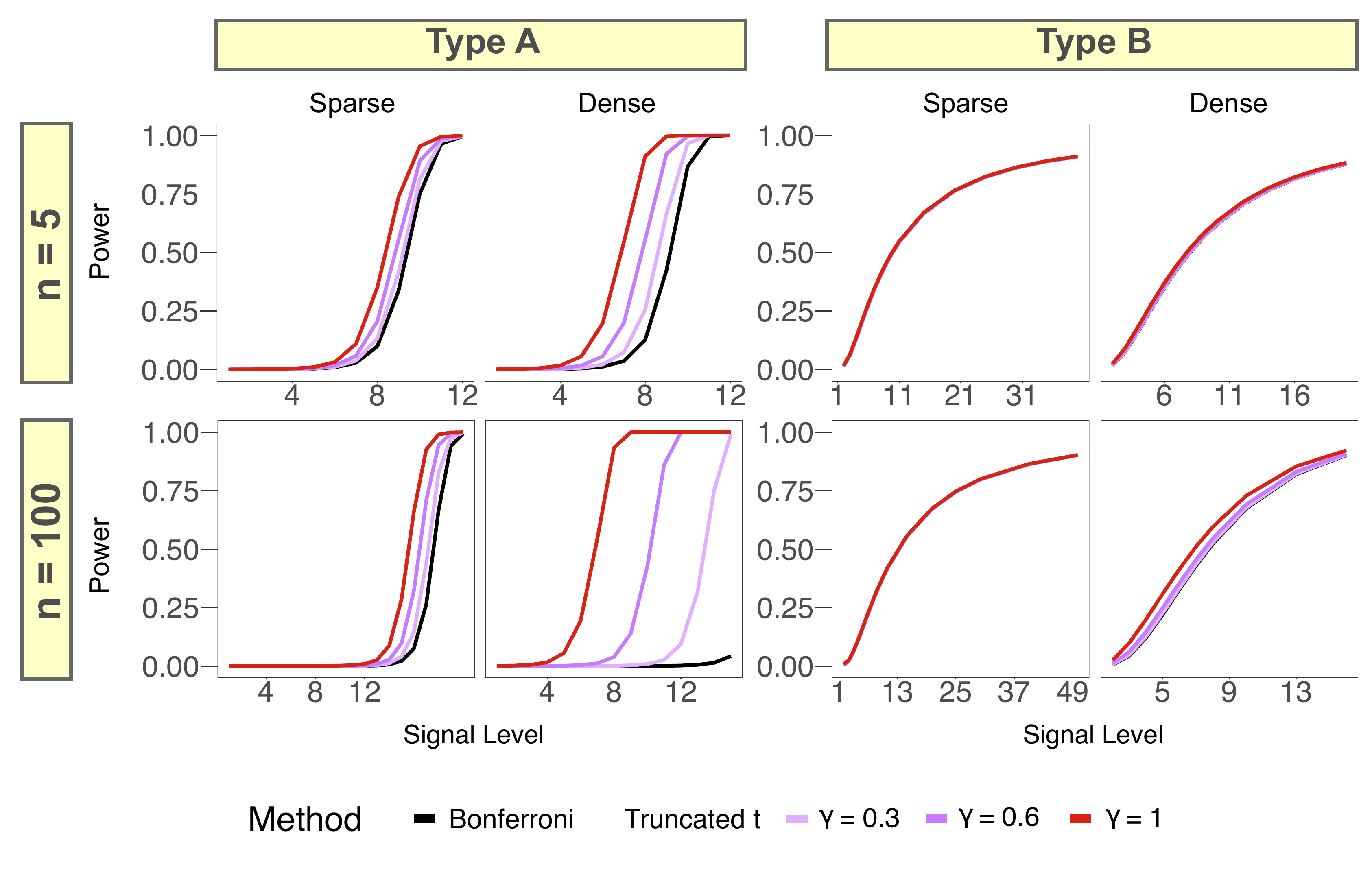}
\caption{The power of the combination test using truncated $t$ distributions and the Bonferroni test versus signal levels under different alternative types at significance level $\alpha=5\times10^{-4}$. The underlying
dependence among base p-values is modeled using the Clayton copula.}
\label{fig:power-signal-clayton-5e-4}
\end{figure}
% \vspace*{\fill}

% \newpage
% \vspace*{\fill}
\begin{figure}[t]
\centering
\includegraphics[width=\textwidth]{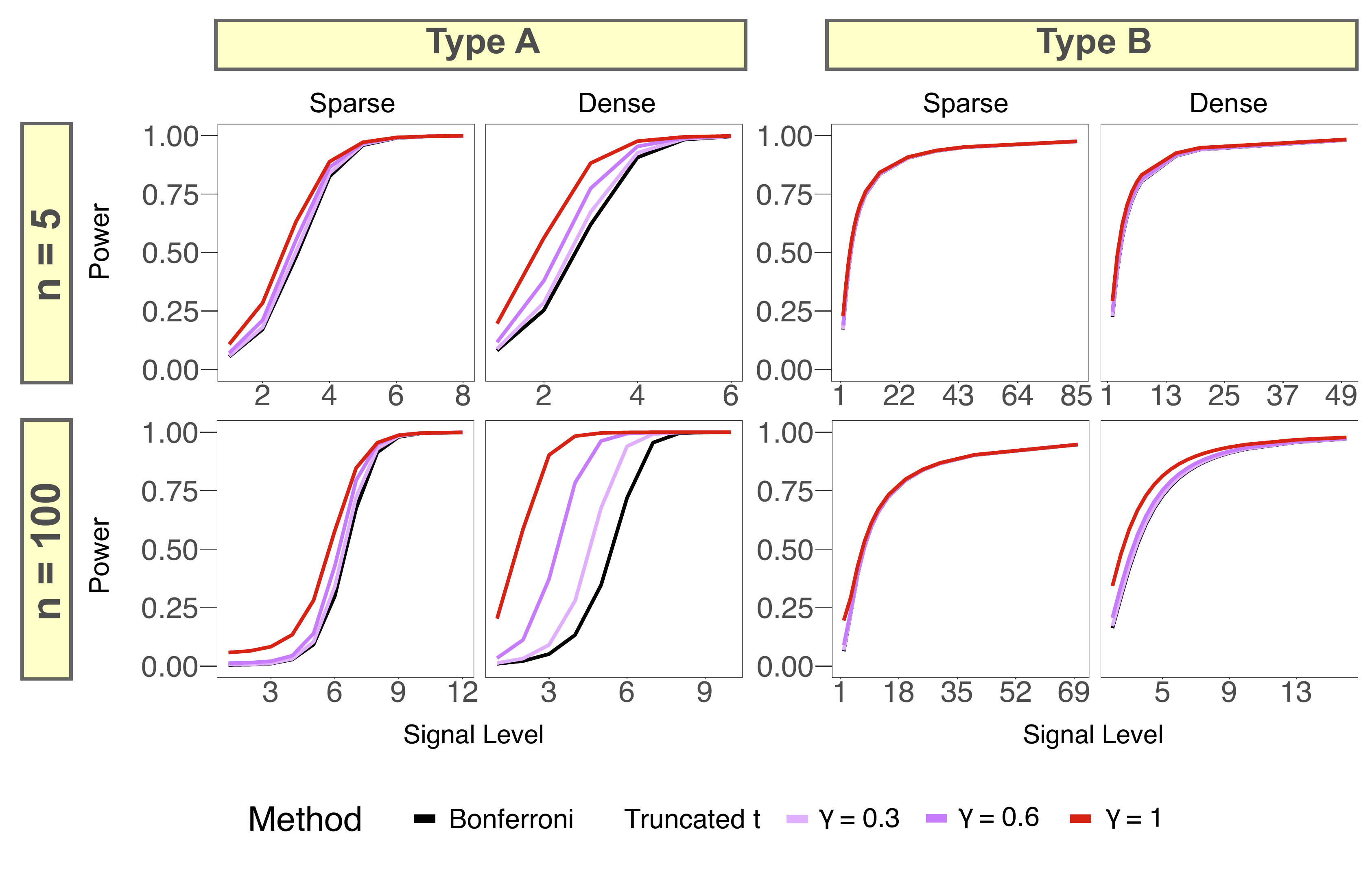}
\caption{The power of the combination test using truncated $t$ distributions and the Bonferroni test versus signal levels under different alternative types at significance level $\alpha=0.05$. The underlying
dependence among base p-values is modeled using the multivariate $t$ copula.}
\label{fig:power-signal-t-5e-2}
\end{figure}
% \vspace*{\fill}

% \newpage
% \vspace*{\fill}
\begin{figure}[t]
\centering
\includegraphics[width=\textwidth]{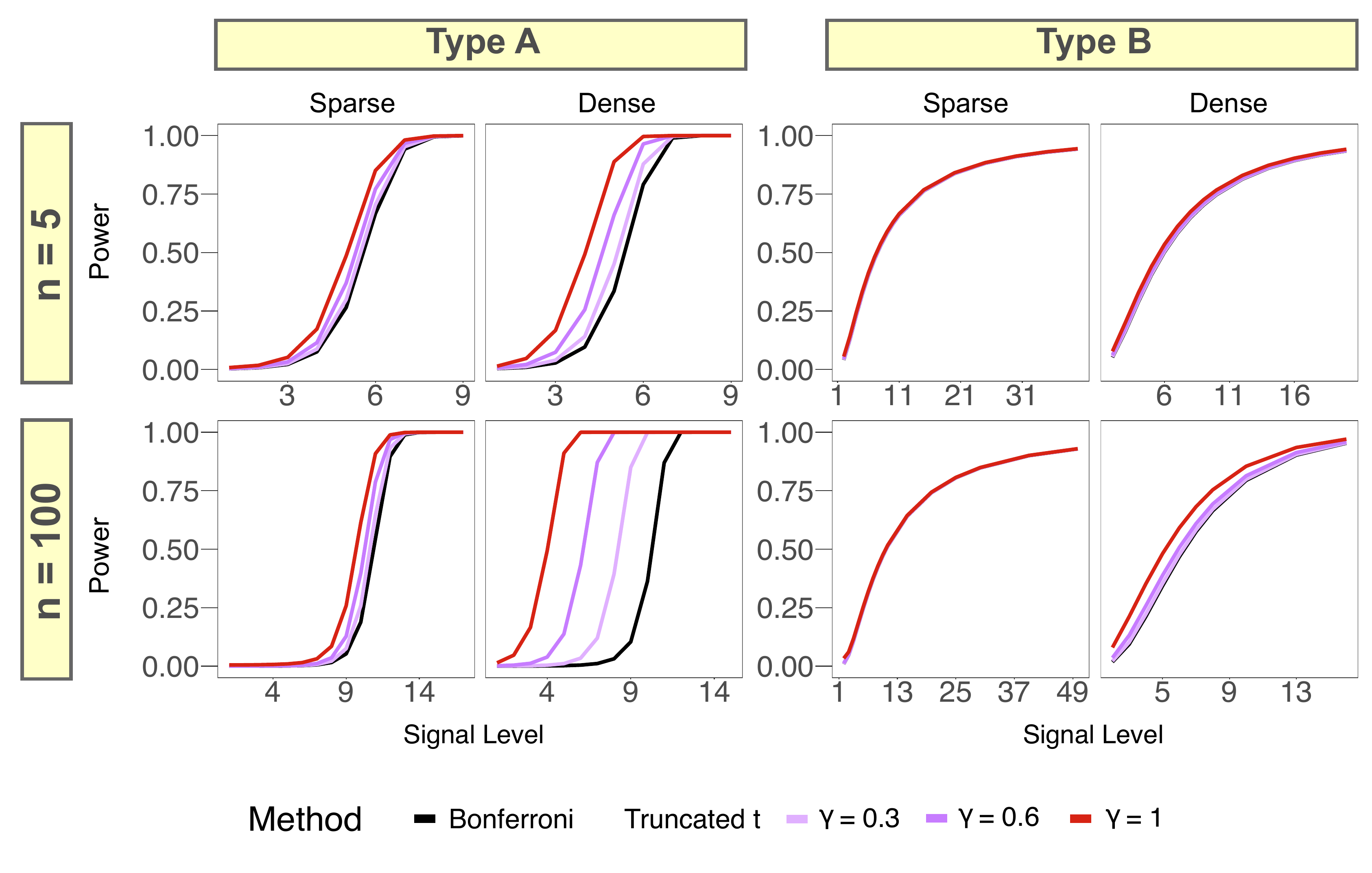}
\caption{The power of the combination test using truncated $t$ distributions and the Bonferroni test versus signal levels under different alternative types at significance level $\alpha=5\times10^{-3}$. The underlying
dependence among base p-values is modeled using the multivariate $t$ copula.}
\label{fig:power-signal-t-5e-3}
\end{figure}
% \vspace*{\fill}

% \newpage
% \vspace*{\fill}
\begin{figure}[t]
\centering
\includegraphics[width=\textwidth]{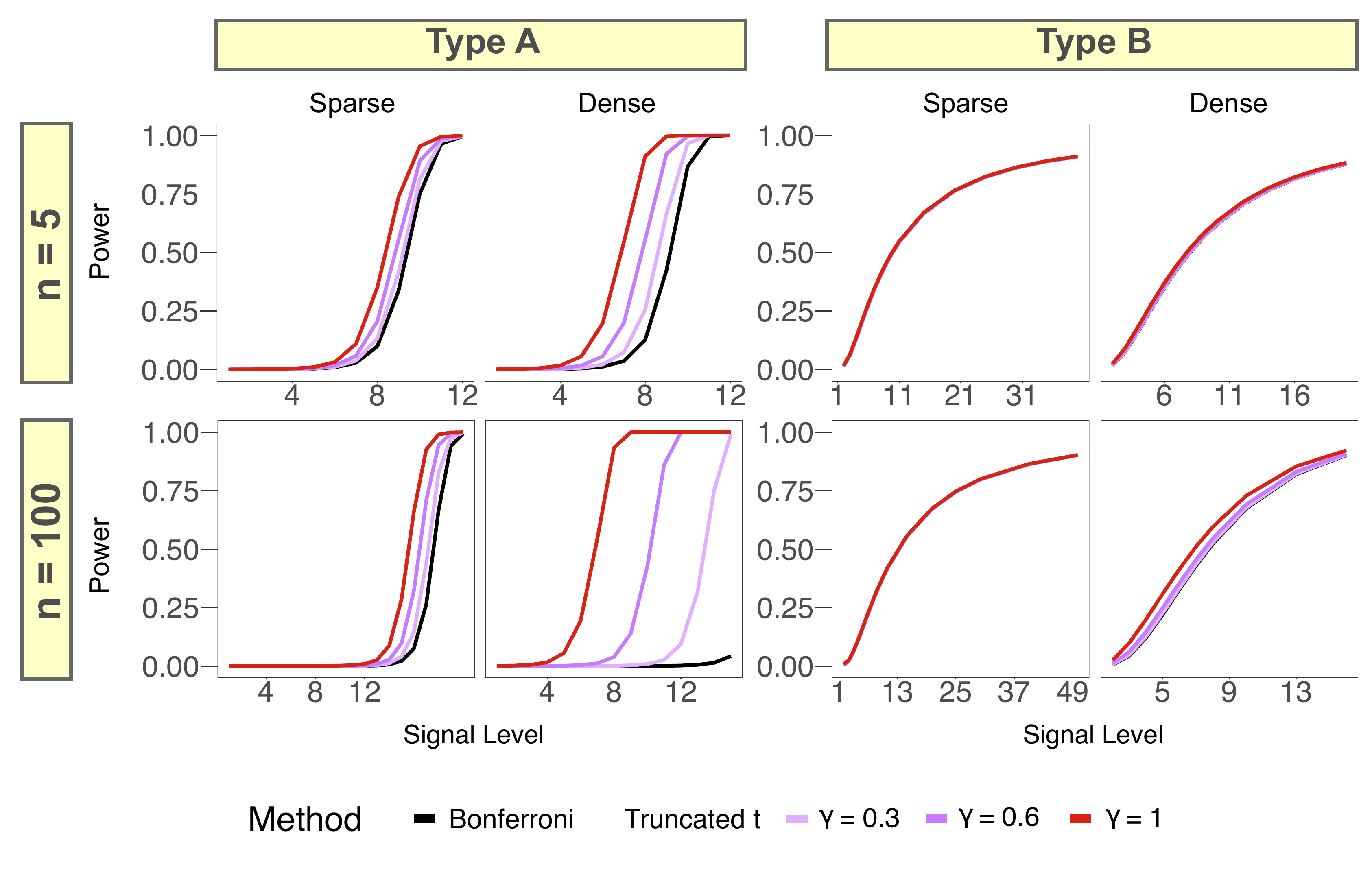}
\caption{The power of the combination test using truncated $t$ distributions and the Bonferroni test versus signal levels under different alternative types at significance level $\alpha=5\times10^{-4}$. The underlying
dependence among base p-values is modeled using the multivariate $t$ copula.}
\label{fig:power-signal-t-5e-4}
\end{figure}
% \vspace*{\fill}

% \newpage
% \vspace*{\fill}
\begin{figure}[t]
\centering
\includegraphics[width=\textwidth]{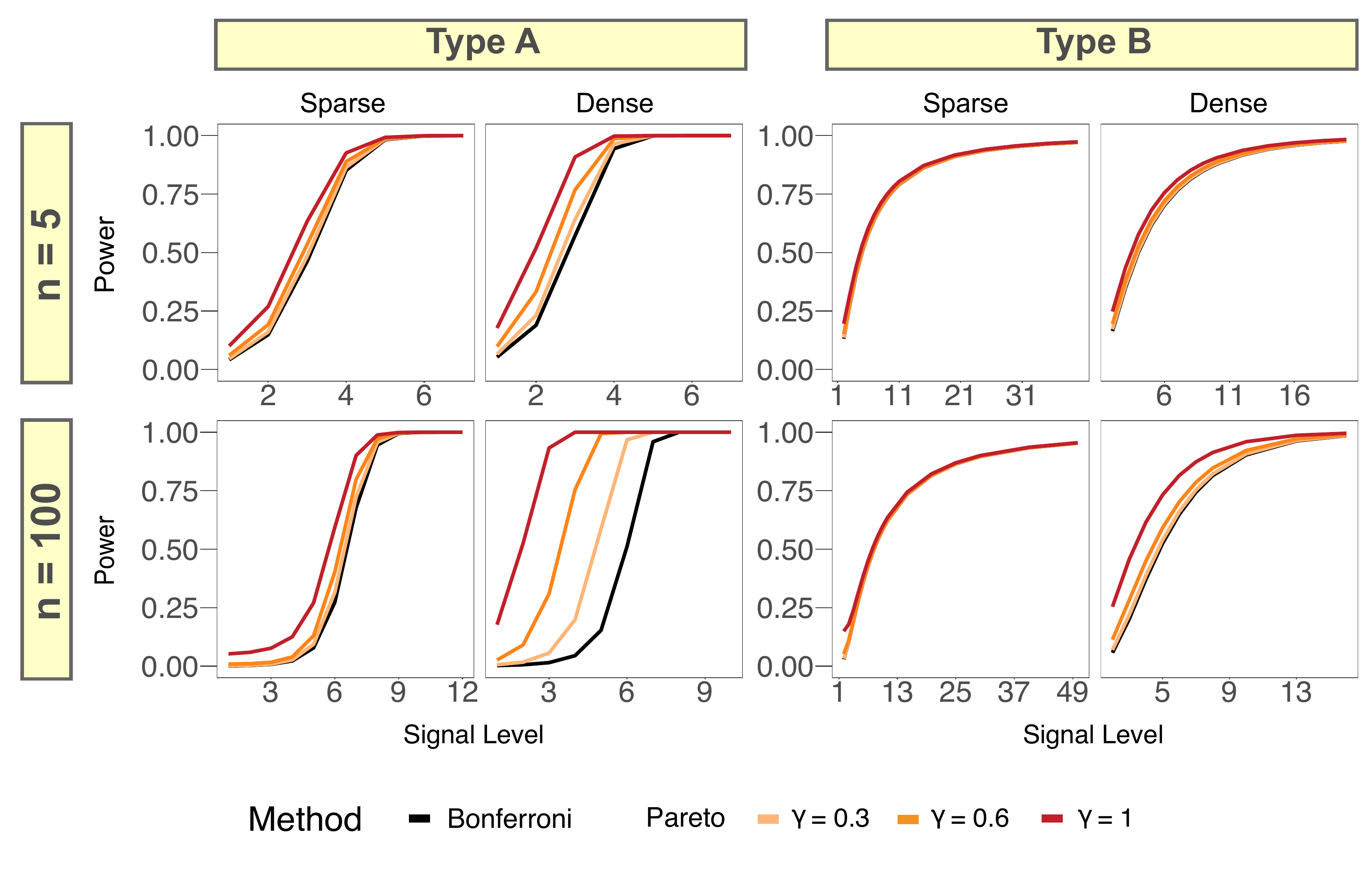}
\caption{The power of the combination test using Pareto distributions and the Bonferroni test versus signal levels under different alternative types at significance level $\alpha=0.05$. The underlying
dependence among base p-values is modeled using the Clayton copula.}
\label{fig:pareto-power-signal-clayton-5e-2}
\end{figure}
% \vspace*{\fill}

% \newpage
% \vspace*{\fill}
\begin{figure}[t]
\centering
\includegraphics[width=\textwidth]{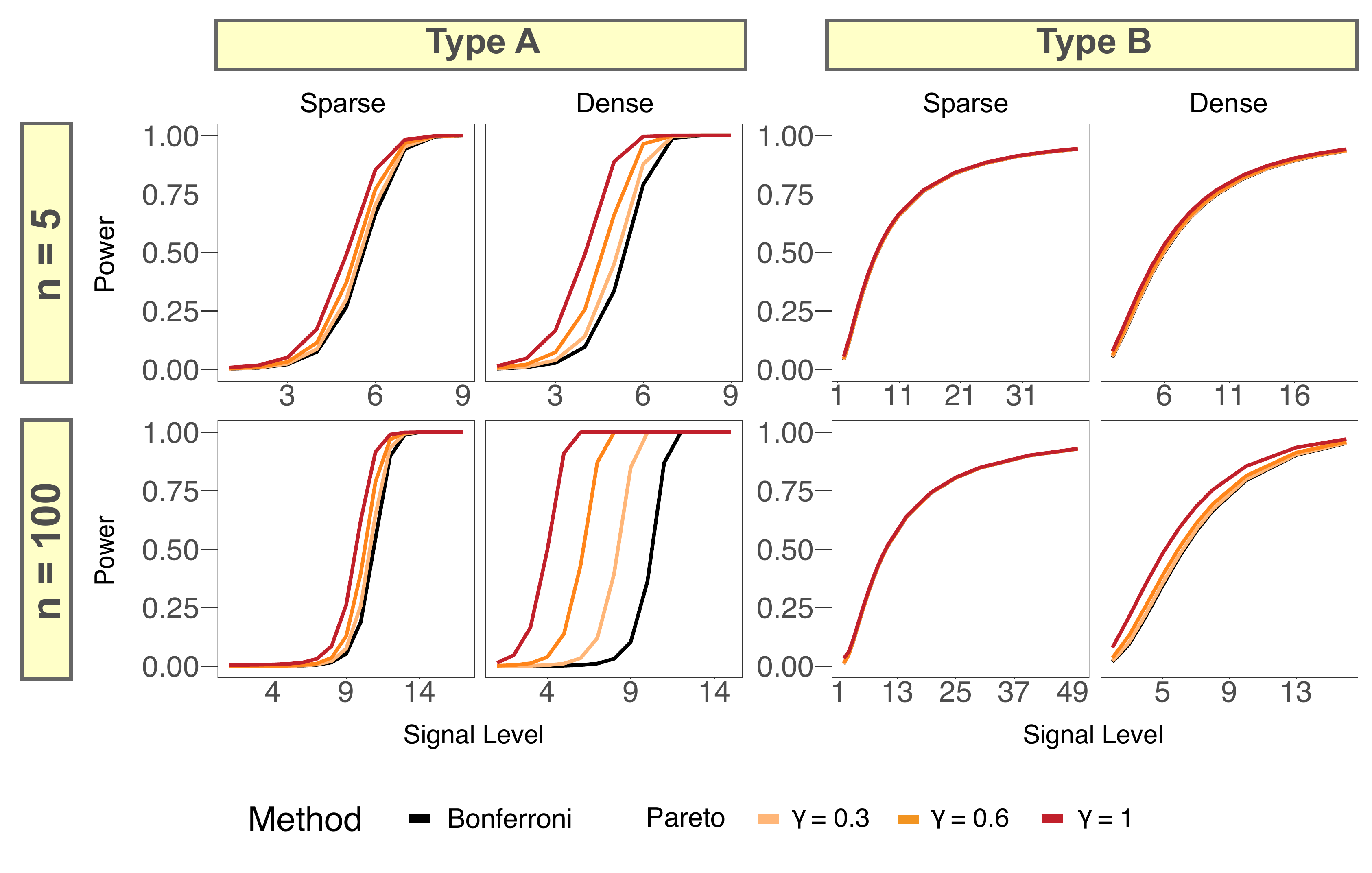}
\caption{The power of the combination test using Pareto distributions and the Bonferroni test versus signal levels under different alternative types at significance level $\alpha=5\times10^{-3}$. The underlying
dependence among base p-values is modeled using the Clayton copula.}
\label{fig:pareto-power-signal-clayton-5e-3}
\end{figure}
% \vspace*{\fill}

% \newpage
% \vspace*{\fill}
\begin{figure}[t]
\centering
\includegraphics[width=\textwidth]{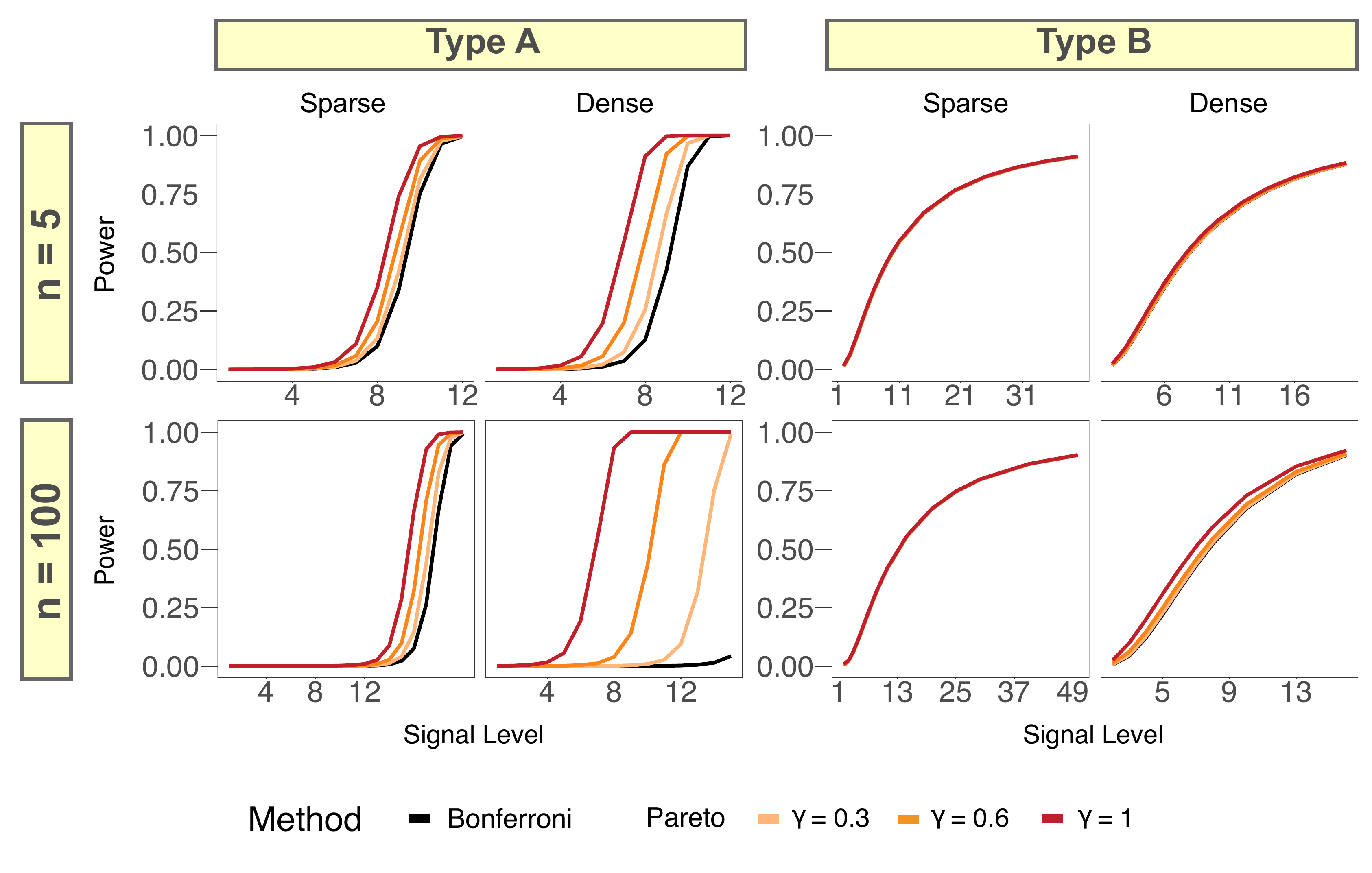}
\caption{The power of the combination test using Pareto distributions and the Bonferroni test versus signal levels under different alternative types at significance level $\alpha=5\times10^{-4}$. The underlying
dependence among base p-values is modeled using the Clayton copula.}
\label{fig:pareto-power-signal-clayton-5e-4}
\end{figure}
% \vspace*{\fill}

% \newpage
% \vspace*{\fill}
\begin{figure}[t]
\centering
\includegraphics[width=\textwidth]{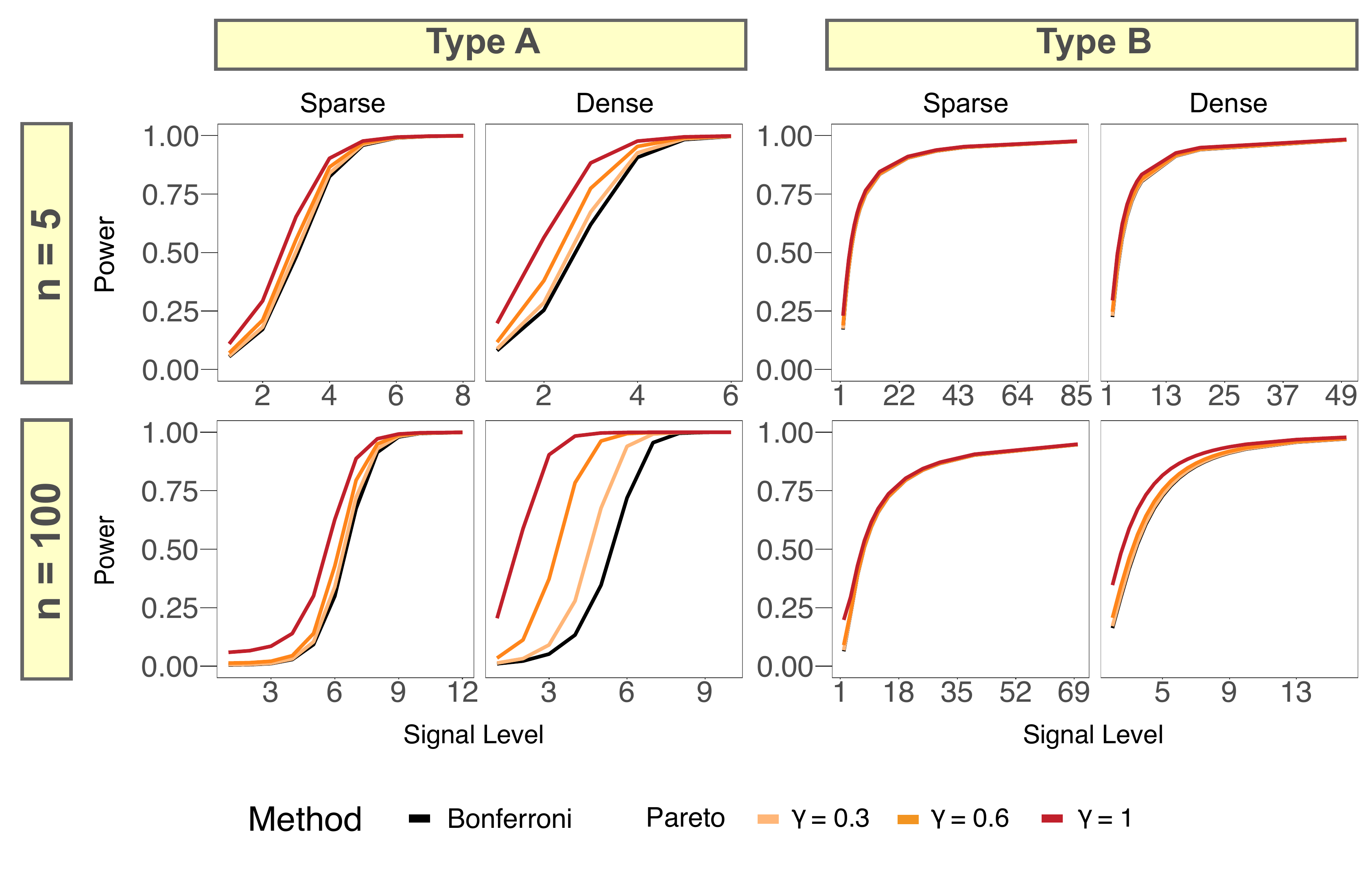}
\caption{The power of the combination test using Pareto distributions and the Bonferroni test versus signal levels under different alternative types at significance level $\alpha=0.05$. The underlying
dependence among base p-values is modeled using the multivariate $t$ copula.}
\label{fig:pareto-power-signal-t-5e-2}
\end{figure}
% \vspace*{\fill}

% \newpage
% \vspace*{\fill}
\begin{figure}[t]
\centering
\includegraphics[width=\textwidth]{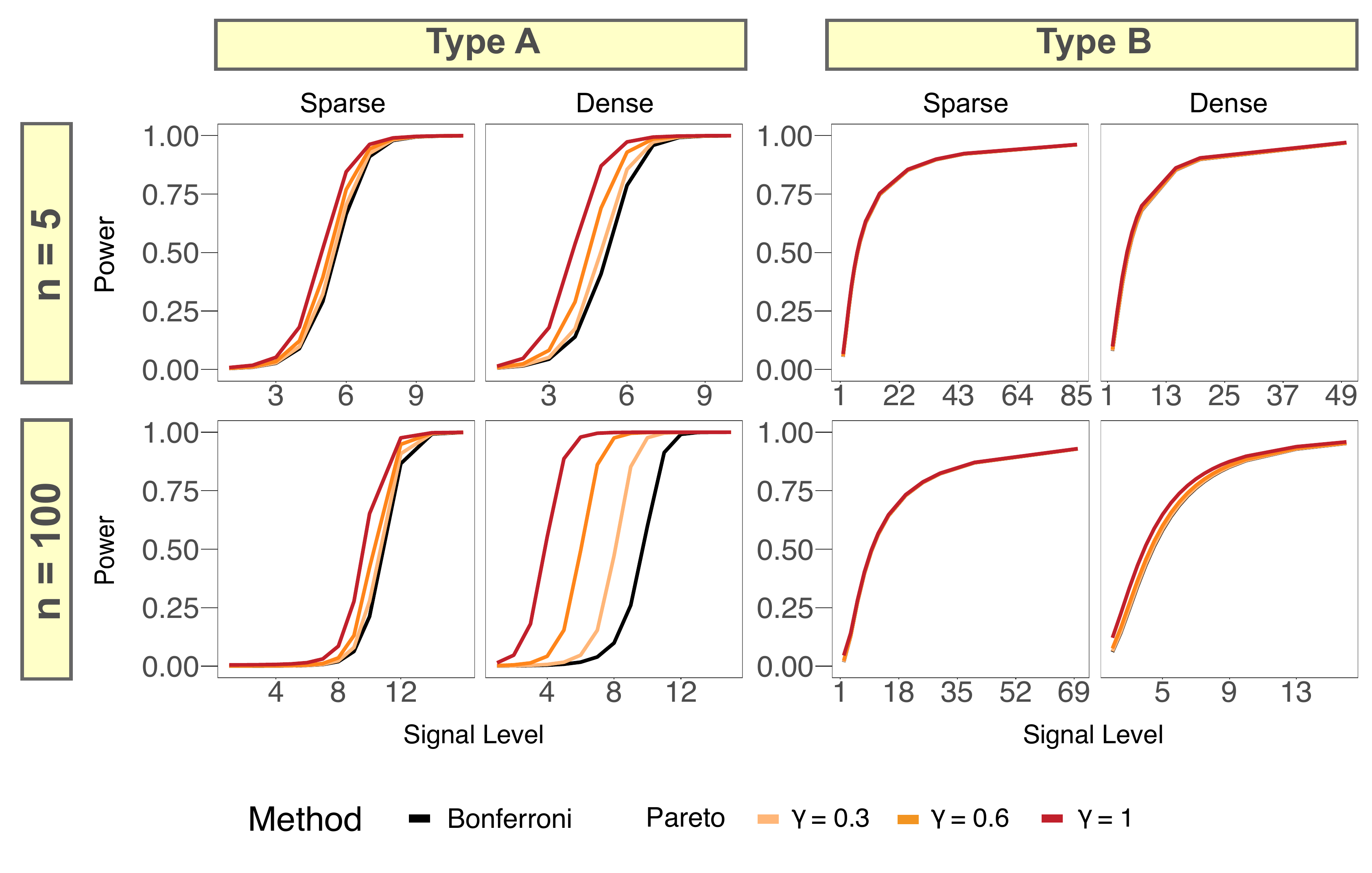}
\caption{The power of the combination test using Pareto distributions and the Bonferroni test versus signal levels under different alternative types at significance level $\alpha=5\times10^{-3}$. The underlying
dependence among base p-values is modeled using the multivariate $t$ copula.}
\label{fig:pareto-power-signal-t-5e-3}
\end{figure}
% \vspace*{\fill}

% \newpage
% \vspace*{\fill}
\begin{figure}[t]
\centering
\includegraphics[width=\textwidth]{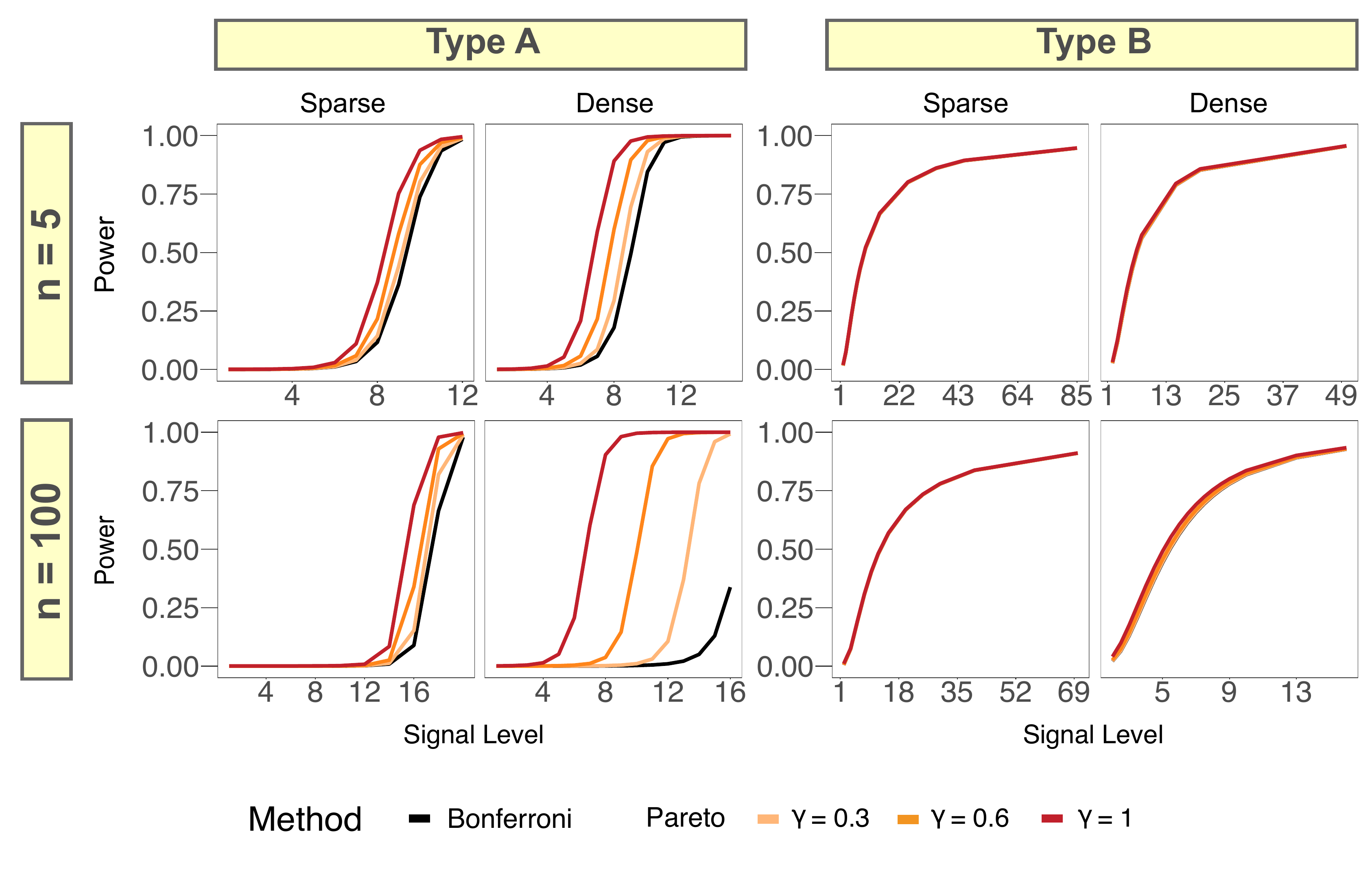}
\caption{The power of the combination test using Pareto distributions and the Bonferroni test versus signal levels under different alternative types at significance level $\alpha=5\times10^{-4}$. The underlying
dependence among base p-values is modeled using the multivariate $t$ copula.}
\label{fig:pareto-power-signal-t-5e-4}
\end{figure}
% \vspace*{\fill}

% \newpage
% \vspace*{\fill}
\begin{figure}[t]
\centering
\includegraphics[width=0.7\textwidth]{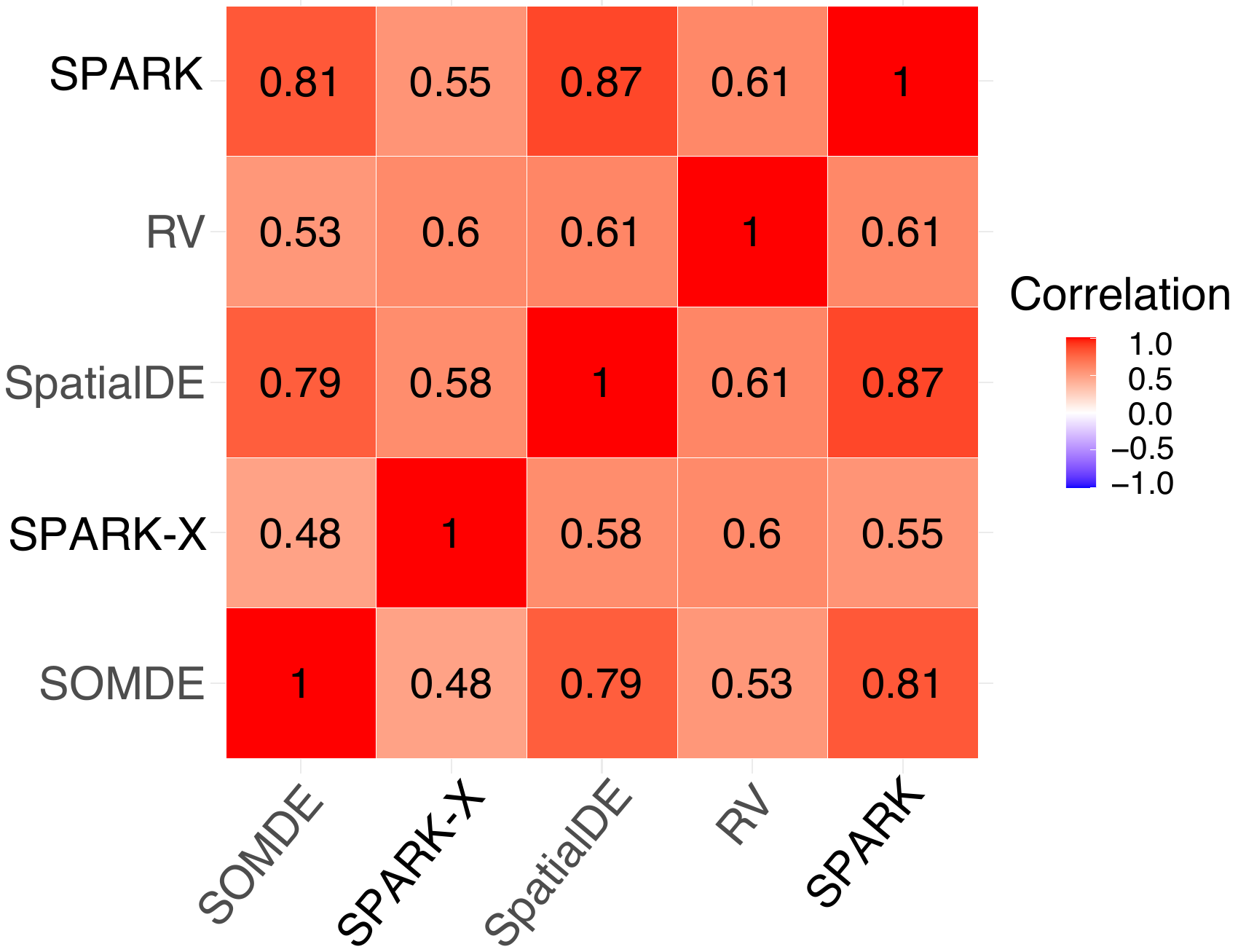}
\caption{The Spearman correlations between p-values by different methods, averaged across $67$ datasets.}
\label{fig:p-vals-corr}
\end{figure}
% \vspace*{\fill}

\begin{figure}[t]
    \centering
    \includegraphics[width=0.7\linewidth]{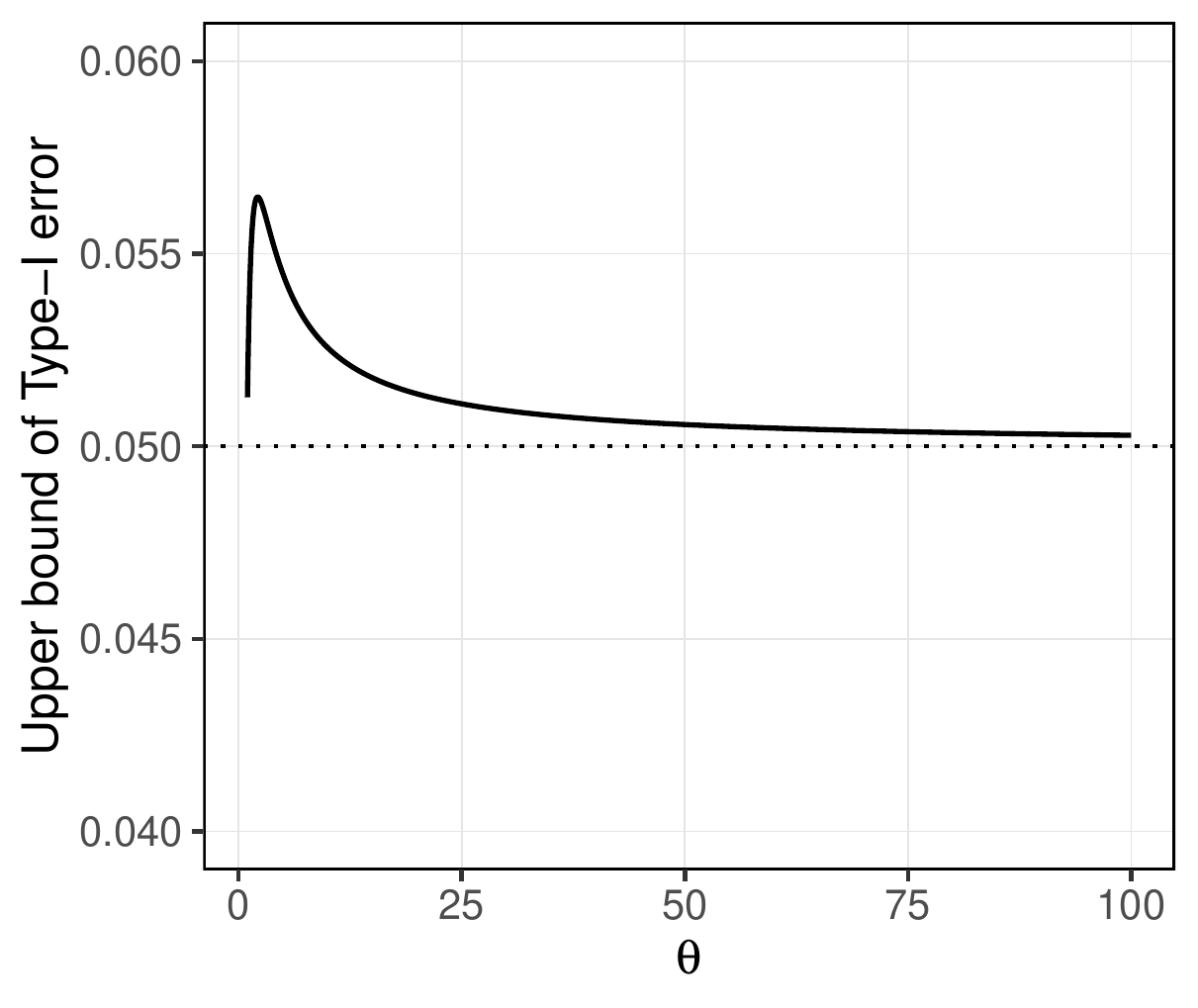}
    \caption{Upper bound on the type-I error of the heavy-tailed combination test under a Clayton copula with Pareto transformation, plotted as a function of the Clayton parameter $\theta$ at significance level $\alpha = 0.05$. The bound, from Theorem 2 of \citet{chen2025subuniformity}, applies when $\theta \geq 1$ (Kendall's $\tau \geq 1/3$) and $\gamma \in [1, \theta]$; here we use $\gamma = 1$. The closed form is $G_{1/\theta}(1/(\alpha^{-\theta} - 1))$, where $G_{1/\theta}$ denotes the CDF of a Gamma distribution with shape parameter $1/\theta$ and rate $1$. The upper bound remains below $0.06$ across the plotted range.}
    \label{fig:nonasymptotic_typeI}
\end{figure}

\begin{table}[ht!]
  \caption{Selected Spatial Transcriptomics Datasets}
  \label{tab:selected_datasets}
  \centering
{\scriptsize
  \begin{tabular}{@{}p{0.26\textwidth}p{0.36\textwidth}p{0.28\textwidth}@{}}
    \toprule
    \textbf{Name} & \textbf{Origin/Tissue} & \textbf{Slide Name} \\
    \midrule
    lnDLPC \citesupp{maynard2021transcriptome} & Human Brain (DLPFC) & 151507-151510,151669-151676 \\
    \addlinespace
    GSE166692 & Developing Mouse Embryos (DME) & 1D-G,1H,3D,3F-H \\
    \addlinespace
    GSE147747 & Mouse Brain (BRAM) & 01,06,11,13,18,20,26-30,AC \\
    \addlinespace
    GSE111672 & Cancer (PDAC) & PDAC-A \\
    \addlinespace
    GSE98364 & Mouse Brain (Cortex) & Cortex \\
    \addlinespace
    STAltmapMBR \citesupp{wang2018three} & Mouse Brain & 20180905\_BY3\_16genes \\
    \addlinespace
    CNP0001343 \citesupp{chen2022spatiotemporal} & Mouse Embryo & E9.5 E2\&3-E2\&4 \\
    \addlinespace
    10xBCBCA & FFPE Human Breast Cancer & - \\
    10xCECA & FFPE Human Cervical Cancer & - \\
    10xINCA & FFPE Human Intestine Cancer & - \\
    10xMBR & FFPE Human Normal Prostate & - \\
    10xMPR & FFPE Mouse Brain & - \\
    10xMKR & FFPE Mouse Kidney & - \\
    10xAMBR & V1 Adult Mouse Brain & - \\
    10xHE & V1 Human Heart & - \\
    10xBRA & V1 Mouse Brain Sagittal Anterior Section2 & - \\
    10xBRB & V1 Mouse Brain Sagittal Posterior Section2 & - \\
    10xMKI & V1 Mouse Kidney & - \\
    10xMOB & V1 Mouse Olfactory Bulb & - \\
    \addlinespace
    EGA00001000851 \citesupp{andersson2021spatial}& Human Breast Cancer & F1-F3 \\
    \addlinespace
    GSE16926CO & Colon & 2101-2112  \\
    \addlinespace
    GSE16926LI & Liver & 2104-2107  \\
    \addlinespace
    GSE11706LI & Mouse Embryo & E10 whole 1-3 \\
    \addlinespace
    GSE13796ET & Mouse Embryo Tails & E11 tail 1-2 \\
    \bottomrule
  \end{tabular}
}
\end{table}

%%  MR numbers will be added by VTeX.                      %%
%%                                                         %%
%%  Use \cite{...} to cite references in text.             %%
%%                                                         %%
%%%%%%%%%%%%%%%%%%%%%%%%%%%%%%%%%%%%%%%%%%%%%%%%%%%%%%%%%%%%%

%% if your bibliography is in bibtex format, uncomment commands:
%\bibliographystyle{imsart-number} % Style BST file (imsart-number.bst or imsart-nameyear.bst)
%\bibliography{bibliography}       % Bibliography file (usually '*.bib')

%% or include bibliography directly:
% \begin{thebibliography}{}
% \bibitem{b1}
% \end{thebibliography}

\end{document}